\documentclass[a4paper,11pt]{article}
\usepackage{amsmath, amsfonts, amscd, amssymb, amsthm, enumerate, xypic}

\def\bfB{\mathbf{B}}

\def\op{\mathrm{op}}

\newcommand{\Isom}{\operatorname{Isom}}
\newcommand{\Mat}{\operatorname{M}}
\newcommand{\charac}{\chi}

\newcommand{\Rad}{\operatorname{Rad}}

\newcommand{\id}{\operatorname{id}}
\newcommand{\GL}{\operatorname{GL}}

\newcommand{\Ker}{\operatorname{Ker}}
\newcommand{\Nil}{\operatorname{Nil}}
\newcommand{\Co}{\operatorname{Co}}
\newcommand{\Irr}{\operatorname{Irr}}
\newcommand{\End}{\operatorname{End}}
\newcommand{\Hom}{\operatorname{Hom}}

\newcommand{\Vect}{\operatorname{span}}
\newcommand{\im}{\operatorname{Im}}

\newcommand{\rk}{\operatorname{rk}}

\renewcommand{\setminus}{\smallsetminus}

\def\F{\mathbb{F}}
\def\K{\mathbb{K}}
\def\R{\mathbb{R}}
\def\C{\mathbb{C}}

\def\N{\mathbb{N}}
\def\M{\mathbb{M}}
\def\Z{\mathbb{Z}}

\renewcommand{\L}{\mathbb{L}}

\def\calA{\mathcal{A}}

\def\calC{\mathcal{C}}

\def\calP{\mathcal{P}}

\def\calS{\mathcal{S}}

\def\calW{\mathcal{W}}

\def\lcro{\mathopen{[\![}}
\def\rcro{\mathclose{]\!]}}

\theoremstyle{definition}
\newtheorem{Def}{Definition}[section]
\newtheorem{Not}[Def]{Notation}

\theoremstyle{plain}
\newtheorem{theo}{Theorem}[section]
\newtheorem{prop}[theo]{Proposition}
\newtheorem{cor}[theo]{Corollary}
\newtheorem{lemma}[theo]{Lemma}

\theoremstyle{plain}

\theoremstyle{remark}
\newtheorem{Rems}{Remarks}
\newtheorem{Rem}[Rems]{Remark}

\title{Sums of two square-zero selfadjoint or skew-selfadjoint endomorphisms}
\author{Cl\'ement de Seguins Pazzis\footnote{Universit\'e de Versailles Saint-Quentin-en-Yvelines, Laboratoire de Math\'ematiques
de Versailles, 45 avenue des Etats-Unis, 78035 Versailles cedex, France}
\footnote{e-mail address: dsp.prof@gmail.com}}

\begin{document}

\thispagestyle{plain}

\maketitle

\begin{abstract}
Let $V$ be a finite-dimensional vector space over a field $\F$, equipped with a symmetric or alternating non-degenerate bilinear form $b$. When the characteristic of $\F$ is not $2$, we characterize the endomorphisms $u$ of $V$ that split into $u=a_1+a_2$
for some pair $(a_1,a_2)$ of $b$-selfadjoint (respectively,
$b$-skew-selfadjoint) endomorphisms of $V$ such that $(a_1)^2=(a_2)^2=0$.
In the characteristic $2$ case, we obtain a similar classification for the endomorphisms of $V$ that split into the sum of two square-zero $b$-alternating endomorphisms of $V$ when $b$ is alternating (an endomorphism $v$ is called $b$-alternating whenever
$b(x,v(x))=0$ for all $x \in V$).

Finally, if the field $\F$ is equipped with a non-identity involution, we
characterize the pairs $(h,u)$ in which $h$ is a Hermitian form on a finite-dimensional space over $\F$, and $u$ is the sum of two square-zero
$h$-selfadjoint endomorphisms.
\end{abstract}

\vskip 2mm
\noindent
\emph{AMS Classification:} 15A23, 11E04, 15A21

\vskip 2mm
\noindent
\emph{Keywords:} Symmetric bilinear forms, Symplectic forms, Square-zero endomorphisms, Decomposition,
Selfadjoint endomorphisms, Skew-selfadjoint endomorphisms


\section{Introduction}

Throughout, $\N$ denotes the set of all nonnegative integers, and $\N^*$ the set of all positive ones.

\subsection{The problem}

The present article is set in the general context of decompositions of elements of an algebra into
quadratic objects. Given an associative algebra $\calA$ over a field $\F$, an element $x \in \calA$
is called quadratic whenever it is annihilated by a polynomial of degree $2$, i.e.\
$x^2 \in \Vect_\F(1_\calA,x)$.
In \cite{dSPsum1,dSPsum2}, we have characterized, in algebras of endomorphisms of finite-dimensional vector spaces,
the elements that split into $a+b$ where $a$ and $b$ are quadratic elements with prescribed annihilating polynomials
of degree $2$. For example, when both polynomials equal $t^2$, this describes the
endomorphisms that are the sum of two square-zero endomorphisms (see \cite{WuWang} and \cite{Bothasquarezero}), when both equal
$t^2-t$ this describes the endomorphisms that are the sum of two idempotents (see \cite{HP} for the seminal results).
The main innovation in \cite{dSPsum1,dSPsum2} allowed to deal with the case where at least one of the prescribed annihilating polynomials is irreducible over $\F$.

In the present work, we wish to start a similar study in the context of vector spaces equipped with non-degenerate
symmetric or alternating bilinear forms. This topic is not entirely new, yet to this day the
research has been mostly concentrated in the study of decomposing elements of orthogonal or symplectic groups into products of two quadratic elements of those groups. The oldest result is a theorem of M.J. Wonenburger \cite{Wonenburger}, which states that in an orthogonal group over a field with characteristic distinct from $2$, every element is the product of two involutions. This result was generalized by Gow \cite{Gow} to fields of characteristic $2$ for orthogonal groups of regular quadratic forms (but not for orthogonal groups of symmetric bilinear forms); see also \cite{Bunger,KnuppelNielsen}.

A few years ago, de La Cruz \cite{delaCruz} considered the problem of sums of square-zero matrices:
among several results, he showed that over the complex field every Hamiltonian matrix is the sum of two square-zero Hamiltonian matrices,
whereas a skew-Hamiltonian matrix is the sum of two square-zero skew-Hamiltonian matrices if and only if it is similar to its opposite.

Due to the appeal of group theory, decompositions into products tend to be more popular,
but we view decompositions into sums as equally interesting.

Now, let us be more precise about our aims. Throughout, all the vector spaces under consideration are assumed to be \emph{finite-dimensional}
(we will only restate this assumption in the main theorems, for reference purpose). All those vector spaces are over a fixed field $\F$
whose characteristic we denote by $\charac(\F)$.
We denote by $\Irr(\F)$ the set of all monic irreducible polynomials of $\F[t]$,
and by $\Irr_0(\F)$ the set of all \emph{even} monic irreducible polynomials of $\F[t]$, i.e.\
$\Irr_0(\F)=\Irr(\F) \cap \F[t^2]$.

Let $V$ be a vector space over $\F$, and
$$b : V \times V \longrightarrow \F$$
be a bilinear form. We assume that $b$ is either \textbf{symmetric} ($\forall (x,y)\in V^2, \; b(x,y)=b(y,x)$)
or \textbf{alternating} ($\forall x\in V, \; b(x,x)=0$).
We say that $b$ is skew-symmetric whenever $\forall (x,y) \in V^2, \; b(y,x)=-b(x,y)$.
The alternating bilinear forms are the skew-symmetric ones if $\charac(\F) \neq 2$,
whereas if $\charac(\F)= 2$ an alternating form is necessarily symmetric but the converse does not hold in general.

In general, the radical of $b$ is defined as
$$\Rad(b):=\{x \in V : \; \forall y \in V, \; b(x,y)=0\}.$$
We shall systematically assume that this radical is zero, i.e.\ $b$ is \emph{non-degenerate}.
Given a linear subspace $W$ of $V$, its orthogonal under $b$ is denoted by $W^{\bot_b}$, or more simply by $W^{\bot}$
in lack of a possible confusion. Remember that $\dim W^{\bot}=\dim V-\dim W$, that $(W^{\bot})^{\bot}=W$, and that $V=W \oplus W^{\bot}$
if and only if the restriction of $b$ to $W^2$ is non-degenerate (in which case we say that $W$ is a \textbf{$b$-regular} subspace).

Given an endomorphism $u$ of $V$, there is a unique endomorphism of $V$, denoted by $u^\star$ and called the $b$-adjoint of $u$, such that
$$\forall (x,y)\in V^2, \; b\bigl(u^\star(x),y\bigr)=b\bigl(x,u(y)\bigr).$$
We say that $u$ is \textbf{$b$-selfadjoint} whenever $u^\star=u$, and that $u$ is \textbf{$b$-skew-selfadjoint}
whenever $u^\star=-u$ (of course, these two notions are equivalent if $\charac(\F)=2$).
We say that $u$ is \textbf{$b$-alternating} whenever $\forall x\in V, \; b(x,u(x))=0$, i.e.\ the bilinear form
$(x,y) \mapsto b(x,u(y))$ is alternating. If $\charac(\F) \neq 2$ and $b$ is symmetric (respectively,
$\charac(\F)\neq 2$ and $b$ is alternating), then $u$ is $b$-alternating  if and only if it is $b$-skew-selfadjoint
(respectively, $b$-selfadjoint). If $\charac(\F)=2$, every $b$-alternating endomorphism
is $b$-selfadjoint but the converse does not hold in general.

The sets of all $b$-selfadjoint, $b$-skew-selfadjoint and $b$-alternating endomorphisms are respectively denoted by
$\calS_{b+}$, $\calS_{b-}$ and $\calA_b$. Obviously, they are linear subspaces (but not subalgebras!) of the algebra $\End(V)$ of all endomorphisms
of $V$.

\vskip 3mm
Now, we can state our main problem:
\begin{center}
In each subspace $\calS_{b+}$, $\calS_{b-}$ and $\calA_b$,
characterize the elements that split into the sum of two square-zero elements.
\end{center}

To this end, the following notation will be useful:

\begin{Not}
Given a linear subspace $\calW$ of $\End(V)$, we denote by $\calW^{[2]}$
the subset of $\calW$ consisting of the elements of the form $a_1+a_2$, where $a_1,a_2$ are elements of $\calW$
such that $a_1^2=a_2^2=0$.
\end{Not}

\subsection{A review of the case where $W=\End(V)$}

At this point, it is important to explain what is exactly meant by ``characterize" in the above problem.

So, let us first go back to the classical case of $\End(V)$ where $V$ is a (finite-dimensional) vector space.
In that case, the problem is invariant under replacing $u$ with a conjugate endomorphism:
indeed, if $\varphi$ is an automorphism of $V$ and $u=a_1+a_2$ for square-zero endomorphisms of $V$,
then $\varphi u \varphi^{-1}=b_1+b_2$ where $b_1:=\varphi a_1 \varphi^{-1}$ and $b_2:=\varphi a_2 \varphi^{-1}$
are square-zero endomorphisms of $V$. Hence, $\End(V)^{[2]}$ is a union of conjugacy classes: if we have (computable) invariants for conjugacy classes
then, in theory, belonging to $\End(V)^{[2]}$ could be expressed in terms of these invariants only.

So, let us recall the relevant invariants for conjugacy classes in the algebra $\End(V)$.
A module over a ring is called of rank $1$ when it is non-zero and generated by a single element.
Every module of rank $1$ over $\F[t]$ is isomorphic to $\F[t]/(p)$ for a unique zero or nonconstant monic polynomial $p \in \F[t]$, which we call the minimal
polynomial of the said module; the module under consideration is finite-dimensional as a vector space over $\F$ if and only if $p$ is nonzero.

Let $u$ be an endomorphism of a vector space $V$. By putting $p(t).x:=p(u)[x]$,
we enrich the vector space $V$ into an $\F[t]$-module $V^u$.
If $V^u$ is of rank $1$, it is isomorphic to $\F[t]/(p)$ where $p$ denotes the minimal polynomial of $u$.

In general $V^u$ splits into the direct sum $V_1 \oplus \cdots \oplus V_r$
of non-zero submodules of rank $1$ whose respective minimal polynomials are so that
$p_{i+1}$ divides $p_i$ for all $i \in \lcro 1,r-1\rcro$.
The polynomials $p_1,\dots,p_r$ are then uniquely determined by $u$: they are its \textbf{invariant factors}.
They characterize the similarity (i.e.\ conjugacy) class of $u$, i.e.\ the orbit of $u$ under the action
of the general linear group $\GL(V)$ by $a.u:=a ua^{-1}$.

Now, we can reformulate the classical result of Wang-Wu \cite{WuWang} and Botha \cite{Bothasquarezero} (see also
the appendix of \cite{dSPsum3} for an alternative proof):

\begin{theo}\label{theoBotha}
Let $V$ be a finite-dimensional vector space, and let $u \in \End(V)$.
Then $u$ belongs to $\End(V)^{[2]}$ if and only if each invariant factor of $u$ is even or odd.
\end{theo}

Alternatively, we can characterize similarity classes in terms of primary invariants
and so-called Jordan numbers. An $\F[t]$-module of rank $1$ is called \textbf{primary} whenever
its minimal polynomial is a power of some (monic) irreducible polynomial.
Using the Chinese remainder theorem, one can split each
module of the form $\F[t]/(p)$ into a direct sum of primary submodules.
It follows that $V^u$ is isomorphic to a direct sum of primary modules:
then, the non-ordered list of minimal polynomials of the summands in such a primary decomposition depends only on the similarity class of $u$,
and these minimal polynomials are called the \textbf{primary invariants} of $u$.
Finally, we can recover those invariants as follows: for each $p \in \Irr(\F)$,
and each positive integer $k$, we denote by $n_{p,k}(u)$ the number of such invariants that equal $p^k$.
The mapping $p(u)$ induce an injective homomorphism of $\F[t]$-modules
$$\Ker p(u)^{k+1}/\Ker p(u)^{k} \longrightarrow \Ker p(u)^{k}/\Ker p(u)^{k-1}.$$
Each one of those modules has a natural structure of vector space over the residue field $\L:=\F[t]/(p)$,
and one checks that the integer $n_{p,k}(u)$ is the dimension over $\L$ of the cokernel of that module.
In other words,
$$n_{p,k}(u)=\dim_{\F[t]/(p)}\Bigl(\Ker p(u)^k/\bigl(\Ker p(u)^{k-1}+(\im p(u) \cap \Ker p(u)^k)\bigr)\Bigr).$$

\begin{Not}
Given a non-zero polynomial $p \in \F[t]$, we define its opposite polynomial as
$$p^{\op}:=(-1)^{\deg p}\, p(-t).$$
If $p$ is monic and irreducible, then so is $p^{\op}$. Note in any case that $(p^{\op})^{\op}=p$ (because $p^{\op}$ has the same degree as $p$).
\end{Not}

Using Jordan numbers, the previous theorem can be reformulated as follows:

\begin{theo}
Let $V$ be a finite-dimensional vector space over the field $\F$, and let $u \in \End(V)$. Assume that $\charac(\F)\neq 2$.
The following conditions are equivalent:
\begin{enumerate}[(i)]
\item The endomorphism $u$ belongs to $\End(V)^{[2]}$.
\item For every $p\in \Irr(\F)$ and every integer $k>0$, one has
$n_{p,k}(u)=n_{p^{\op},k}(u)$.
\item For every $p \in \Irr(\F) \setminus (\Irr_0(\F) \cup \{t\})$ and every integer $k>0$, one has
$n_{p,k}(u)=n_{p^{\op},k}(u)$.
\item $u$ is similar to $-u$.
\end{enumerate}
\end{theo}

Now, we turn to the characteristic $2$ case.

\begin{theo}
Let $V$ be a finite-dimensional vector space over the field $\F$, and let $u \in \End(V)$. Assume that $\charac(\F)= 2$.
The following conditions are equivalent:
\begin{enumerate}[(i)]
\item The endomorphism $u$ belongs to $\End(V)^{[2]}$.
\item For every $p \in \Irr(\F) \setminus (\Irr_0(\F) \cup \{t\})$ and every odd integer $k \geq 1$, one has
$n_{p,k}(u)=0$.
\item For every non-zero eigenvalue $\lambda$
of $u$ in an algebraic closure of $\F$, the Jordan cells of $u$ for the eigenvalue $\lambda$ are all even-sized.
\end{enumerate}
\end{theo}

Notice that, according to these theorems, every nilpotent endomorphism of a finite-dimensional vector space is the sum of two square-zero ones!

\subsection{The viewpoint of pairs}\label{section:pairs}

In the context of spaces equipped with a bilinear form (except in very special cases) there are known sets of invariants
that characterize the orbits of the elements of
$\calS_{b+}$, $\calS_{b-}$ and $\calA_b$ under the action of the isometry group $\Isom(b)$
by conjugation. This is relevant because our problem is invariant under conjugation.
Actually, the relevant invariants are not associated to the endomorphism $u$ itself but to the pair $(b,u)$.

The following terminology will be of help:

\begin{Def}
Let $(\varepsilon,\eta) \in \{-1,1\}^2$ and $(b,u)$ be a pair consisting of a non-degenerate bilinear form $b$
on a vector space $V$, and of an endomorphism $u$ of $V$.
We say that $(b,u)$ is an $(\varepsilon,\eta)$-pair when:
\begin{itemize}
\item $b$ is symmetric and $u$ is $b$-selfadjoint, if $\varepsilon=\eta=1$;
\item $b$ is symmetric and $u$ is $b$-alternating (i.e.\ $b$-skew-selfadjoint if $\charac(\F)\neq 2$), if $\varepsilon=1$ and $\eta=-1$;
\item $b$ is alternating and $u$ is $b$-alternating (i.e.\ $b$-selfadjoint if $\charac(\F)\neq 2$) if $\varepsilon=-1$ and $\eta=1$;
\item $b$ is alternating and $u$ is $b$-skew-selfadjoint if $\varepsilon=-1$ and $\eta=-1$.
\end{itemize}
\end{Def}

Hence, in an $(\varepsilon,\eta)$-pair $(b,u)$:
\begin{itemize}
\item $\varepsilon$ is the parity of $b$;
\item $\eta$ is the parity of $u$ with respect to $b$ (i.e.\ $u^\star=\eta u$);
\item $\varepsilon \eta$ is the parity of the bilinear form $(x,y) \mapsto b(x,u(y))$.
\end{itemize}

\begin{Def}
Let $(\varepsilon,\eta) \in \{1,-1\}^2$.
An $(\varepsilon,\eta)$-pair $(b,u)$, with underlying vector space $V$,
is said to have the \textbf{square-zero splitting property} whenever there exist
endomorphisms $a_1$ and $a_2$ of $V$ such that $u=a_1+a_2$, $a_1^2=a_2^2=0$ and
$(b,a_1)$ and $(b,a_2)$ are $(\varepsilon,\eta)$-pairs.
\end{Def}

Next, we consider the notion of an induced pair.
Let $(b,u)$ be an $(\varepsilon,\eta)$-pair with underlying vector space $V$.
Assume that we have a linear subspace $W$ of $V$ such that $u(W) \subset W$.
Since $u$ is $b$-selfadjoint or $b$-skew-selfadjoint, this leads to $u(W^{\bot_b}) \subset W^{\bot_b}$.
It follows that $u$ induces an endomorphism $\overline{u}$ of the quotient space $W/(W \cap W^{\bot_b})$.
Moreover, since $W \cap W^{\bot_b}$ is the radical of the bilinear form induced by $b$ on $W$,
it turns out that $b$ induces a non-degenerate bilinear form $\overline{b}$ on $W/(W \cap W^{\bot_b})$.
It is easily checked that the pair $(\overline{b},\overline{u})$ is an $(\varepsilon,\eta)$-pair.
Note that if $W$ is already $b$-regular, then this new pair is defined on $W$ itself (with the usual identification between $W$ and $W/\{0\}$).

\begin{Def}
Let $(b,u)$ be an $(\varepsilon,\eta)$-pair with underlying vector space $V$, and let
$W$ be a linear subspace of $V$ that is stable under $u$. We denote by $(b,u)^W$ the pair induced by $(b,u)$
on the quotient space $W/(W \cap W^{\bot_b})$. We call it the pair \textbf{induced by $(b,u)$ for $W$.}
\end{Def}

\begin{Rem}\label{remark:squarezerosplittinginduced}
The square-zero splitting property is not inherited by induced pairs in general, yet in specific instances it is:
assume indeed that $(b,u)$ is an $(\varepsilon,\eta)$-pair, with underlying vector space $V$,
and assume that there is a splitting $u=a_1+a_2$ where each $a_i$ has square zero and is such that $(b,a_i)$ is an $(\varepsilon,\eta)$-pair.
Let $W$ be a linear subspace of $V$ that is not only stable under $u$, but also under $a_1$ and $a_2$.
Denote by $\overline{u},\overline{a_1},\overline{a_2}$ the endomorphisms of $W/(W \cap W^{\bot_b})$
induced by $u,a_1,a_2$, respectively. Then, $\overline{a_1}$ and $\overline{a_2}$ have square zero and sum up to $\overline{u}$, and of course
$(b,a_1)^W$ and $(b,a_2)^W$ are $(\varepsilon,\eta)$-pairs. It follows that $(b,u)^W$ has the square-zero splitting property.
\end{Rem}

\vskip 3mm
Next, given two bilinear forms $b$ and $c$ on respective vector spaces $V$ and $W$,
remember that an isometry from $b$ to $c$ is a vector space isomorphism $f : V \overset{\simeq}{\longrightarrow} W$
such that $c(f(x),f(y))=b(x,y)$ for all $(x,y)\in V^2$.

\begin{Def}
Given endomorphisms $u \in \End(V)$ and $v \in \End(W)$, the pairs $(b,u)$ and $(c,v)$
are called \textbf{isometric} if there exists an isometry $f$ from $b$ to $c$ such that $v=f \circ u \circ f^{-1}$.
\end{Def}

This defines an equivalence relation on the collection of all pairs $(b,u)$ consisting of a bilinear form $b$ on a vector space
and of an endomorphism $u$ of it.
One checks that if $(b,u)$ and $(c,v)$ are isometric pairs, and $(\varepsilon,\eta)$ is an arbitrary element of $\{-1,1\}^2$,
then $(b,u)$ is an $(\varepsilon,\eta)$-pair if and only if $(c,v)$ is an $(\varepsilon,\eta)$-pair; also, in that case
$(b,u)$ has the square-zero splitting property if and only if $(c,v)$ has the square-zero splitting property.

Shortly, we will recall a complete set of invariants for the $(\varepsilon,\eta)$-pairs $(b,u)$ under isometry.
Before we do so, it is important that we also introduce the notion of the orthogonal direct sum of pairs.

\begin{Def}[Orthogonal direct sum of pairs]
Let $(b,u)$ and $(b',u')$ be two $(\varepsilon,\eta)$-pairs, with respective underlying vector spaces $V$ and $V'$.
Their orthogonal direct sum is the pair $(B,U)=(b,u)\bot (b',u')$ defined
with
$$B : ((x,x'),(y,y')) \in (V \times V')^2 \longmapsto b(x,y)+b'(x',y')$$
and
$$U : (x,x') \in V \times V' \longmapsto \bigl(u(x),u'(x')\bigr),$$
i.e.\ $B$ is the orthogonal direct sum $b \bot b'$ and $U$ is the (external) direct sum $u \oplus u'$.
Note that $(B,U)$ is itself an  $(\varepsilon,\eta)$-pair.
\end{Def}

\begin{Rem}\label{remark:orthogonalsumsplitting}
Let $(b,u)$ and $(b',u')$ be $(\varepsilon,\eta)$-pairs, with underlying spaces $V$ and $V'$,
each with the square-zero splitting property.
Let us split $u=a_1+a_2$ and $u'=a'_1+a'_2$ where the $a_i$'s and $a'_j$'s have square zero, and $(b,a_i)$ and $(b',a'_i)$ are
$(\varepsilon,\eta)$-pairs.
Then, $u \oplus u'=(a_1 \oplus a'_1) + (a_2 \oplus a'_2)$, the endomorphisms $a_1 \oplus a'_1$ and $a_2 \oplus a'_2$
have square zero, and $(b,a_1) \bot (b',a'_1)$ and $(b,a_2) \bot (b',a'_2)$ are $(\varepsilon,\eta)$-pairs.
This shows that $(b,u) \bot (b',u')$ has the square-zero splitting property.
\end{Rem}

Beware that the converse does not hold in general. It is possible for $(b,u)\bot (b',u')$ to have the square-zero splitting property
without $(b,u)$ nor $(b',u')$ having this property.

There is however an interesting special situation which we will sometimes use.
Assume that we have an $(\varepsilon,\eta)$-pair $(b,u)$ with the square-zero splitting property, and consider corresponding square-zero endomorphisms
$a_1$ and $a_2$. Denote the underlying space of $(b,u)$ by $V$, and assume that we have a splitting $V=V_1 \oplus V_2$ into
$b$-orthogonal subspaces $V_1$ and $V_2$ that are both stable under $a_1$ and $a_2$.
Then the induced pairs $(b,u)^{V_1}$ and $(b,u)^{V_2}$, which are respectively defined on $V_1$ and $V_2$ (because they are $b$-regular subspaces), have the square-zero splitting property. In that case, we also note that $(b,u) \simeq (b,u)^{V_1} \bot (b,u)^{V_2}$.

\subsection{A review of the classification of selfadjoint and skew-selfadjoint endomorphisms}\label{section:invariants}

We are ready to review the relevant invariants for the $(\varepsilon,\eta)$-pairs.
The simplest case is the one where $\eta=1$.

\subsubsection{Extension of a bilinear form}\label{section:extensionbilinform}

Let $p \in \Irr(\F)$, whose degree we denote by $d$.
We consider the residue field $\L:=\F[t]/(p)$
and the $\F$-linear form $e_p$ on $\L$ that takes the class $\overline{1}$ to $1$, and the class $\overline{t^k}$ to $0$
for all $k \in \lcro 1,d-1\rcro$. The $\F$-bilinear form $(x,y)\in \L^2 \mapsto e_p(xy)$ is then non-degenerate, and hence
every $\F$-linear form on $\L$ reads $x \mapsto e_p(x\alpha)$ for a unique $\alpha \in \L$.

Let $W$ be an $\L$-vector space equipped with an $\F$-bilinear form $B : W^2 \rightarrow \F$
such that
$$\forall \lambda \in \L, \; \forall (x,y)\in W^2, \quad B(\lambda x,y)=B(x,\lambda y).$$
For all $x,y$ in $W$, the $\F$-linear form $\lambda \in \L \mapsto B(x,\lambda y)$ can be written
as $\lambda \mapsto e_p(\lambda B^{\L}(x,y))$ for a unique $B^{\L}(x,y)\in \L$.
It is then easily found that $B^{\L}$ is $\F$-bilinear; one also checks that $B^{\L}$ is right-$\L$-linear.
To sum up, $B^{\L}$ is defined by the following property :
$$\forall \lambda \in \L, \; \forall (x,y)\in W^2, \;  B(x,\lambda y)=e_p(\lambda B^{\L}(x,y)).$$

If $B$ is symmetric (respectively, skew-symmetric) then so is $B^{\L}$,
and in that case $B^{\L}$ turns out to be $\L$-bilinear (because it is already right-$\L$-linear),
and better still $B^{\L}$ is non-degenerate if in addition $b$ is non-degenerate.

There is a specific difficulty related to the characteristic $2$ case.
Assume that $B$ is not only alternating but that
$B(x,\lambda x)=0$ for all $x \in W$ and all $\lambda \in \L$.
Then, by letting $x \in W$, one finds that $e_p(\lambda B^{\L}(x,x))=B(x,\lambda x)=0$
for all $\lambda \in \F$, which yields $B^{\L}(x,x)=0$.
Hence, $B^\L$ is alternating. It is not true however that if $B$ is alternating
then so is $B^\L$ (there is a special difficulty when $\F$ has characteristic $2$ and $p$ is even).

We will also need a variation of the above construction in case $p$ is even.
So, assume that $p$ is even and $\chi(\F) \neq 2$.
The mapping that takes $q(t)$ to $q(-t)$ induces a non-identity involution $x \mapsto x^\bullet$ of $\L$
(that takes the class of $t$ to its opposite). The subfield $\K:=\{x \in \L : x^\bullet=x\}$
of $\L$ has index $2$ in $\L$ (and the extension $\K-\L$ is separable). Note that $\K$ is simply the subfield of $\L$ generated by the class of $t^2$, and that
$\forall x \in \L, \; e_p(x^\bullet)=e_p(x)$.
Now, let $W$ be an $\L$-vector space and $B : W^2 \rightarrow \F$ be an
$\F$-bilinear form such that
$$\forall \lambda \in \L, \; \forall (x,y)\in W^2, \; B(x,\lambda y)=B(\lambda^\bullet x,y).$$
Like before, there is a unique $\F$-bilinear mapping $B^\L : W^2 \rightarrow \L$
such that
$$\forall \lambda \in \L, \; \forall (x,y)\in W^2, \; B(x,\lambda y)=e_p(\lambda B^\L(x,y)).$$
Again, $B^\L$ is right-$\L$-linear.
Now, say that $B$ is symmetric (in which case set $\varepsilon:=1$) or alternating (in which case $\varepsilon:=-1$).
Then, for all $(x,y,\lambda)\in W^2 \times \L$,
$$B(y,\lambda x)=B(\lambda^\bullet y,x)=\varepsilon B(x,\lambda^\bullet y)=\varepsilon
e_p(\lambda^\bullet B^\L(x,y))=e_p(\lambda \varepsilon B^\L(x,y)^\bullet).$$
It follows that
$$\forall (x,y)\in W^2, \;  B^\L(y,x)=\varepsilon B^\L(x,y)^\bullet.$$
We conclude that $B^\L$ is Hermitian if $B$ is symmetric, and skew-Hermitian if $B$ is alternating.
In any case, $B^\L$ is left-$\L$-semilinear.
In any of these cases, $B$ is non-degenerate if and only if $B^\L$ is non-degenerate.

Note finally that skew-Hermitian forms over $\L$ are in natural one-to-one correspondence with Hermitian forms:
indeed we can choose $\eta \in \L \setminus \{0\}$ such that $\eta^\bullet=-\eta$ (it suffices to take $\eta$ as the class of $t$) and then $C \mapsto \eta C$ yields a bijection between Hermitian forms on $W$ and skew-Hermitian forms on $W$, and this bijection and its inverse bijection turn equivalent forms into equivalent forms.

\subsubsection{Symmetric endomorphisms}

We are ready to recall the classification of $(\varepsilon,1)$-pairs.

Let $p \in \Irr(\F)$, and set $\L:=\F[t]/(p)$.
Let $P:=(b,u)$ be an $(\varepsilon,1)$-pair for some $\varepsilon \in \{-1,1\}$.
There, $u$ is $b$-selfadjoint so one sees that $q(u)$ is $b$-selfadjoint for all $q \in \F[t]$,
and it classically turns out that $(\Ker q(u))^{\bot_b}=\im q(u)$.

Let $k$ be a positive integer. Consider the $\F$-bilinear form
$$\begin{cases}
\bigl(\Ker p(u)^{k}\bigr)^2 & \longrightarrow \F \\
(x,y) & \longmapsto b(x,p(u)^{k-1}[y])
\end{cases}$$
on the vector space $\Ker p(u)^k$.
Note that is is symmetric (respectively, skew-symmetric) if
$b$ is symmetric (respectively, skew-symmetric).

Let us look at its radical: clearly it is the intersection of $\Ker p(u)^k$
with the inverse image of $(\Ker p(u)^k)^{\bot_b}=\im p(u)^k$ under $p(u)^{k-1}$, and one easily
checks that this inverse image equals $\Ker p(u)^{k-1}+(\Ker p(u)^k \cap \im p(u))$.
Hence, the preceding bilinear form induces a non-degenerate $\F$-bilinear form $\overline{b_{p,k}}$ on the quotient space $V_{p,k}:=\Ker p(u)^k/(\Ker p(u)^{k-1}+(\Ker p(u)^k \cap \im p(u)))$, i.e.\ on the cokernel of the mapping from
$\Ker p(u)^{k+1}/\Ker p(u)^k$ to $\Ker p(u)^{k}/\Ker p(u)^{k-1}$ induced by $p(u)$.
Remember that these quotient spaces are naturally seen as vector spaces over $\L$, and hence so does the said cokernel.
Since $u$ is $b$-selfadjoint it turns out that
$\overline{b_{p,k}}(x,\lambda y)=\overline{b_{p,k}}(\lambda x,y)$ for all
$\lambda \in \L$ and all $(x,y)\in (V_{p,k})^2$.
Hence, we can consider the extended non-degenerate $\L$-bilinear form
$P_{p,k}:=\overline{b_{p,k}}^{\L}$, which is symmetric if $\varepsilon=1$.
If $\varepsilon=1$, we say that $P_{p,k}$ is the \textbf{quadratic invariant} of $P$ with respect to $(p,k)$
(in fact, this quadratic invariant should be defined not as $P_{p,k}$ but as its equivalence ``class"; however, in practice we prefer to
deal with $P_{p,k}$ directly).

Assuming that $\varepsilon=-1$, we also check that
$P_{p,k}$ is actually alternating. Indeed, we start from the observation that
$b(x,x)=0=b(x,u(x))$ for all $x \in W$, in that case.
Then, for every integer $k \geq 0$, we gather that $b(x,u^{2k}(x))=b(u^k(x),u^k(x))=0$
and $b(x,u^{2k+1}(x))=b(u^k(x),u(u^k(x)))=0$ for all $x \in V$, and it follows that
$b(x,q(u)[x])=0$ for all $q \in \F[t]$ and all $x \in V$.
Hence, $\overline{b_{p,k}}(x,\lambda x)=0$ for all $\lambda \in \L$ and all $x \in V_{p,k}$.
As explained in the previous section, it follows that $P_{p,k}$ is alternating (and hence symplectic).

From the above construction, it is clear that the equivalence class of the $\L$-bilinear form $P_{p,k}$
depends only on the equivalence class of the pair $P=(b,u)$ and on the pair $(p,k)$.
It turns out that the family of those equivalence classes, when $p$ ranges over $\Irr(\F)$ and $k$ ranges over the positive integers, is sufficient to understand $P$ up to equivalence: the following theorem can be retrieved from Section 2 of \cite{Waterhouse}
(there, the definition of the quadratic invariants is slightly different, but it is not difficult to see that the invariants defined by
Waterhouse correspond to ours up to multiplication with fixed nonzero scalars).

\begin{theo}\label{theo:epsilon1pairs}
Let $\varepsilon \in \{-1,1\}$, and assume that $\varepsilon=1$ if $\charac(\F) \neq 2$.
Let $(b,u)$ and $(c,v)$ be $(\varepsilon,1)$-pairs.

The following conditions are equivalent:
\begin{enumerate}[(i)]
\item $(b,u)$ is isometric to $(c,v)$;
\item For every $p \in \Irr(\F)$ and every integer $k \geq 1$, the forms $(b,u)_{p,k}$ and $(c,v)_{p,k}$ are equivalent as bilinear forms over the residue field $\F[t]/(p)$.
\end{enumerate}
\end{theo}

Moreover, there is little limitation on the potential invariants $(b,u)_{p,k}$.
More precisely, let $(a_{p,k})_{p \in \Irr(\F), k \in \N^*}$ be a family in which,
for all $p \in \Irr(\F)$ and all $k \in \N^*$, the object $a_{p,k}$ is a
non-degenerate bilinear form on a (finite-dimensional) vector space over $\F[t]/(p)$.
Then the following conditions are equivalent:
\begin{itemize}
\item There exists an $(\varepsilon,1)$-pair $P$ such that
$P_{p,k} \simeq a_{p,k}$ for all $p,k$;
\item All the $a_{p,k}$'s are symmetric if $\varepsilon=1$, alternating if $\varepsilon=-1$,
and only finitely many of them are nonzero.
\end{itemize}

Note that the Jordan numbers of $u$ are easily retrieved from the forms $(b,u)_{p,k}$
since $n_{p,k}(u)$ is simply the dimension over the residue field $\F[t]/(p)$ of the underlying vector space of $(b,u)_{p,k}$.
For a $(-1,1)$-pair $(b,u)$, the form $(b,u)_{p,k}$ is symplectic and hence
its equivalence class is entirely determined by the dimension of the underlying space (which must be even). Hence, in that case the Jordan numbers of $u$ are sufficient to classify the pair $(b,u)$: this is expressed in the next theorem:

\begin{theo}[Scharlau \cite{Scharlaupairs}]
Let $(b,u)$ and $(c,v)$ be $(-1,1)$-pairs. The following conditions are equivalent:
\begin{enumerate}[(i)]
\item The pairs $(b,u)$ and $(c,v)$ are isometric.
\item The endomorphisms $u$ and $v$ are similar.
\item The endomorphisms $u$ and $v$ have the same Jordan numbers.
\end{enumerate}
\end{theo}

\begin{Rem}
The quadratic invariants of a $(\varepsilon,1)$-pair are compatible with orthogonal direct sums.
In other words, given such pairs $(b,u)$ and $(b',u')$, and given $p \in \Irr(\F)$ and $k \geq 1$,
one has $((b,u)\bot (b',u'))_{p,k} \simeq (b,u)_{p,k} \bot (b',u')_{p,k}$.
\end{Rem}

\subsubsection{Skew-selfadjoint endomorphisms}

Here, we assume that $\charac(\F) \neq 2$, we take an $(\varepsilon,-1)$-pair $(b,u)$
for some $\varepsilon \in \{-1,1\}$, and we denote by $V$ its underlying vector space. Note that $u$ is $b$-skew-selfadjoint.

For all $q \in \F[t]$, we have
$$\forall (x,y)\in V^2, \; b(x,q(u)[y])=b(q(-u)[x],y),$$
so that $q(u)^\star=q(-u)$.
In particular, $q(u)$ is $b$-selfadjoint if $q$ is even, and $b$-skew-selfadjoint if $q$ is odd.

Classically, the endomorphism $u^\star$ is similar to the (dual) transpose $u^t : \varphi \in V^\star \mapsto \varphi \circ u$ of $u$, and hence to $u$.
Hence, the Jordan numbers of $-u$ are the ones of $u$, leading to
$n_{p,k}(u)=n_{p^\op,k}(u)$ for every $p \in \Irr(\F)$, which is a non-trivial condition only when $p$ is non-even and different from $t$.

Now, let $p \in \Irr_0(\F)$. Consider the residue field $\L:=\F[t]/(p)$ and its non-identity involution $x \mapsto x^\bullet$
that takes the class of $t$ to the one of $-t$.

Let also $k \in \N^*$. Noting that $p(u)$ is $b$-selfadjoint, we obtain, like in the previous section, that the bilinear form
$(x,y) \mapsto b(x,p(u)^{k-1}(y))$ induces a non-degenerate bilinear form on the quotient space
$V_{p,k}:=\Ker p(u)^k/(\Ker p(u)^{k-1}+(\Ker p(u)^k \cap \im p(u)))$, which has a natural structure of vector space over $\L$.
This bilinear form $B$ satisfies
$$\forall (x,y,\lambda)\in (V_{p,k})^2\times \L, \; B(x,\lambda\,y)=B(\lambda^\bullet\,x,y).$$
Hence, we can consider its $\L$-extension $(b,u)_{p,k}:=B^{\L}$, which is
a Hermitian form over the $\L$-vector space $V_{p,k}$ if $b$ is symmetric, and a skew-Hermitian one if $b$ is skew-symmetric.
The equivalence class of this Hermitian (respectively, skew-Hermitian) form is then uniquely determined by the isometry class of $(b,u)$ and by the pair $(p,k)$.
The $(b,u)_{p,k}$ forms, when $p$ ranges over $\Irr_0(\F)$ and $k$ ranges over the positive integers, are called the \textbf{Hermitian invariants} of
$(b,u)$ (in fact, the Hermitian invariants of $(b,u)$ should be defined as the equivalence ``classes" of the previous Hermitian forms).

There is one additional family of invariants in the present situation. Consider the polynomial $p:=t$, and
let $k \in \N^*$. The bilinear form
$(x,y) \mapsto b(x,u^{k-1}(y))$ induces a non-degenerate bilinear form $(b,u)_{t,k}$ on
$V_{t,k}:=\Ker u^k/(\Ker u^{k-1}+(\Ker u^k \cap \im u))$.
Note that if $k$ is odd, the $(b,u)_{t,k}$ form has the same \emph{parity} as $b$
(i.e.\ it is symmetric if $b$ is symmetric, and alternating if $b$ is alternating);
if $k$ is even then $(b,u)_{t,k}$ has the inverse parity of $b$, i.e.\
it is symmetric if $b$ is alternating, and it is alternating if $b$ is symmetric.
As before, in case $(b,u)_{t,k}$ is alternating its isometry class is entirely determined
by the corresponding Jordan number $n_{t,k}(u)$ (which must be even).
In case $(b,u)_{t,k}$ is symmetric, we call it the \textbf{quadratic invariant} of $(b,u)$ attached to $(t,k)$.

Again, the following theorem can be derived, with some effort, from Theorem 4 of \cite{Sergeichuk}
or from \cite{Riehm}.

\begin{theo}[Classification of $(\varepsilon,-1)$-pairs]
Assume that $\charac(\F) \neq 2$. Let $\varepsilon \in \{-1,1\}$.
Let $(b,u)$ and $(c,v)$ be $(\varepsilon,-1)$-pairs.
For $(b,u)$ to be isometric to $(c,v)$, it is necessary and sufficient that all the following conditions hold:
\begin{enumerate}[(i)]
\item For each positive integer $k>0$, the bilinear forms $(b,u)_{t,k}$ and $(c,v)_{t,k}$ are equivalent.
\item For every $p \in \Irr_0(\F)$ and every integer $k \geq 1$, the Hermitian forms (respectively, the skew-Hermitian forms)
$(b,u)_{p,k}$ and $(c,v)_{p,k}$ over $\F[t]/(p)$ are equivalent if $\varepsilon=1$ (respectively, if $\varepsilon=-1$).
\item For every $p \in \Irr(\F) \setminus (\Irr_0(\F) \cup \{t\})$ and every positive integer $k>0$, one has $n_{p,k}(u)=n_{p,k}(v)$.
\end{enumerate}
\end{theo}

Moreover, if $\varepsilon=1$ (respectively, $\varepsilon=-1$) condition (i) is equivalent to the conjunction of the following two conditions:
\begin{itemize}
\item For every even (respectively, odd) integer $k \geq 1$, the Jordan numbers $n_{t,k}(u)$ and $n_{t,k}(v)$ are equal;
\item For every odd (respectively, even) integer $k \geq 1$, the bilinear forms $(b,u)_{t,k}$ and $(c,v)_{t,k}$ are equivalent.
\end{itemize}

\subsection{Main results}\label{section:results}

We are now ready to state our main results.

Unsurprisingly, the case of symplectic forms turns out to be the easier by far, and
the results obtained by de La Cruz in \cite{delaCruz} over the complex field hold true in an arbitrary field of characteristic different from $2$.

\begin{theo}\label{theo:altformantisym}
Let $\F$ be a field with $\chi(\F) \neq 2$, and $b$ be a symplectic form over $\F$.
Then every $b$-skew-selfadjoint endomorphism is the sum of two square-zero ones.
\end{theo}

In other words, over a field of characteristic different from $2$, every $(-1,-1)$-pair has the square-zero splitting property.

For the next theorem, in which there is no distinction of characteristic, remember that when $b$ is symplectic and $\chi(\F) \neq 2$
the $b$-selfadjoint endomorphisms are the $b$-alternating ones.

\begin{theo}\label{theo:altformalt}
Let $\F$ be a field, and $b$ be a symplectic form over $\F$.
Let $u$ be a $b$-alternating endomorphism.
The following conditions are equivalent:
\begin{enumerate}[(i)]
\item $u$ is the sum of two square-zero $b$-alternating endomorphisms;
\item $u$ is the sum of two square-zero endomorphisms;
\item Every invariant factor of $u$ is even or odd.
\end{enumerate}
Moreover, if $\chi(\F) \neq 2$, then these conditions are equivalent to each one of the following additional conditions:
\begin{itemize}
\item[(iv)] $u$ is similar to $-u$;
\item[(v)] $u$ is conjugated to $-u$ through an element of the symplectic group of $b$.
\end{itemize}
\end{theo}

For symmetric bilinear forms, the characterization is much more involving.
In that case it is necessary to separate the study into the one of pairs $(b,u)$ with $u$ nilpotent, and the one of pairs
$(b,u)$ with $u$ bijective.

To see that such a reduction is relevant, we use two very simple lemmas:

\begin{lemma}[Commutation lemma]\label{lemma:commutation}
Let $u$ be an endomorphism of a vector space $V$, and $a_1,a_2$ be square-zero endomorphisms of $V$ such that $u=a_1+a_2$.
Then $a_1$ and $a_2$ commute with $u^2$.
\end{lemma}

\begin{proof}
We expand $u^2=a_1a_2+a_2a_1$ and we compute $a_1u^2=a_1a_2a_1=u^2a_1$ and $a_2u^2=a_2a_1a_2=u^2a_2$.
\end{proof}

\begin{lemma}[Stabilization lemma]\label{lemma:stabilization}
Let $u$ be an endomorphism of a vector space $V$, and $a_1,a_2$ be square-zero endomorphisms of $V$ such that $u=a_1+a_2$.
Then $\Ker u^k$ and $\im u^k$ are stable under $a_1$ and $a_2$, for every integer $k \geq 0$.
\end{lemma}

\begin{proof}
The commutation lemma readily shows that $\Ker u^{2k}$ and $\im u^{2k}$ are stable under $a_1$ and $a_2$, for every integer $k \geq 0$.

Next, note that $a_1 u+u a_1=a_1a_2+a_2a_1=u^2$.
For every integer $k \geq 0$, using the commutation of $a_1$ with $u^2$ entails that
$a_1 u^{2k+1}+u^{2k+1} a_1=u^{2k+2}$, and one easily deduces that $\Ker u^{2k+1}$ et $\im u^{2k+1}$ are stable under $a_1$.
And symmetrically they are also stable under $a_2$.
\end{proof}

Now, consider an $(\varepsilon,\eta)$-pair $(b,u)$ that has the square-zero splitting property.
Consider an associated pair $(a_1,a_2)$ of square-zero endomorphisms. Denote by $V$ the underlying vector space.
We introduce the Fitting decomposition of $u$, consisting of the core
$\Co(u):=\underset{n \in \N}{\bigcap} \im u^n$
and of the nilspace
$\Nil(u):=\underset{n \in \N}{\bigcup} \Ker u^n$
of $u$. Remember that $V=\Co(u) \oplus \Nil(u)$, that there exists an integer $n \geq 0$ such that
$\Co(u)=\im u^{2n}$ and $\Nil(u)=\Ker u^{2n}$, and that $u$ induces an automorphism of $\Co(u)$ and a nilpotent endomorphism of $\Nil(u)$.
Since $u^{2n}$ is $b$-selfadjoint, we deduce that $\Nil(u)=\Co(u)^{\bot_b}$.
Using the Stabilization Lemma, we see that $a_1$ and $a_2$ stabilize both summands of the Fitting decomposition of $u$,
and by Remark \ref{remark:squarezerosplittinginduced} in Section \ref{section:pairs} this shows that the induced pairs $(b,u)^{\Nil(u)}$ and $(b,u)^{\Co(u)}$ have the square-zero splitting property. Of course, if we start from an arbitrary $(\varepsilon,\eta)$-pair $(b,u)$ such that both induced pairs $(b,u)^{\Nil(u)}$ and $(b,u)^{\Co(u)}$ have the square-zero splitting property,
then $(b,u) \simeq (b,u)^{\Nil(u)} \bot (b,u)^{\Co(u)}$ has the square-zero splitting property.

This leads to the following general result:

\begin{prop}
Let $(b,u)$ be an $(\varepsilon,\eta)$-pair. For $(b,u)$ to have the square-zero splitting property,
it is necessary and sufficient that both $(b,u)^{\Co(u)}$ and $(b,u)^{\Nil(u)}$ have the square-zero splitting property.
\end{prop}

Now, this leaves us with four subproblems to solve: for each $\eta \in \{1,-1\}$,
characterize the $(1,\eta)$-pairs $(b,u)$, with $u$ nilpotent or bijective,
that have the square-zero splitting property.

We start with the case of automorphisms, as it is substantially easier.

\begin{theo}\label{theo:symaltautomorphism}
Assume that $\charac(\F) \neq 2$.
Let $(b,u)$ be a $(1,-1)$-pair in which $u$ is an automorphism.
The following conditions are equivalent:
\begin{enumerate}[(a)]
\item $(b,u)$ has the square-zero splitting property;
\item All the Hermitian invariants of $u$ are hyperbolic, and all the Jordan numbers of $u$ are even.
\end{enumerate}
\end{theo}

For $b$-selfadjoint endomorphisms, the classification is more difficult. Some terminology on quadratic forms will help:

\begin{Def}
Let $B$ be a bilinear form on a vector space $V$. Let $\lambda \in \F$. We say that $B$ \textbf{represents}
$\lambda$ whenever there exists a vector $x$ of $V \setminus \{0\}$ such that $B(x,x)=\lambda$.
\end{Def}

First of all, let $p \in \Irr_0(\F)$. We can then write $p=q(t^2)$ for some $q \in \Irr(\F)$.
In $\L:=\F[t]/(p)$, we have the subfield $\K$ generated by the class of $t^2$,
and the automorphism $x \mapsto x^\bullet$ that takes the class of $r(t)$ to the one of $r(-t)$, for all $r\in \F[t]$.
It turns out that $\K=\{x \in \L : x^\bullet=x\}$ and that $e_p(x^\bullet)=e_p(x)$ for all $x \in \L$.

We say that a quadratic form $Q$ on a vector space $W$ over $\L$ is \textbf{skew-representable}
whenever there is a $Q$-orthogonal basis $(e_1,\dots,e_n)$ of $W$ such that $Q(e_i)^\bullet=-Q(e_i)$ for all $i \in \lcro 1,n\rcro$.

Next, let $p \in \Irr(\F) \setminus \Irr_0(\F)$.
The $\F$-algebras $\L:=\F[t]/(p)$ and $\L^{\op}:=\F[t]/(p^\op)$
are naturally isomorphic through the isomorphism $\varphi : \L^\op \rightarrow \L$ that takes the class of $t$ to the one of $-t$
(so that the class of $q$ is systematically mapped to the one of $q(-t)$).
Note that $e_{p^\op}(\alpha)=e_p(\varphi(\alpha))$ for all $\alpha \in \L^\op$.
Through this isomorphism, every bilinear form $B$ on an $\L^{\op}$-vector space $W$
yields a bilinear form $B^{\op}:=\varphi^{-1} \circ B$ on the $\L$-vector space $W^{\op}$ which has the same underlying structure of abelian group as
$W$ but in which the scalar multiplication is defined by $\lambda.x:=\varphi(\lambda).x$

\begin{theo}\label{theo:symsymautomorphism}
Assume that $\charac(\F) \neq 2$.
Let $(b,u)$ be a $(1,1)$-pair in which $u$ is bijective.
The following conditions are equivalent:
\begin{enumerate}[(a)]
\item $(b,u)$ has the square-zero splitting property;
\item For every $p \in \Irr_0(\F)$ and every integer $k \geq 1$, the quadratic form $(b,u)_{p,k}$
is skew-representable, and for every
$p \in \Irr(\F) \setminus \Irr_0(\F)$ and every integer $k \geq 1$, the quadratic invariants $(b,u)_{p,k}$ and
$(-1)^{1+(k-1)\deg p}(b,u)_{p^{\op},k}^\op$ are equivalent over the field $\F[t]/(p)$.
\end{enumerate}
\end{theo}

It remains to consider the nilpotent endomorphisms. There, we need more subtle considerations on symmetric bilinear forms
over fields of characteristic different from $2$.
Remember that every such form $B$ (on a finite-dimensional vector space) is the orthogonal sum of a hyperbolic symmetric bilinear form and of a
nonisotropic symmetric bilinear form; the rank of the hyperbolic form equals $2\nu(B)$, where $\nu(B)$
is the Witt index of $B$, i.e.\ the greatest dimension for a totally $B$-isotropic space; the
equivalence ``class" of the nonisotropic form depends only on that of $B$, and we say that this nonisotropic form is a
\textbf{nonisotropic part} of $B$.
Finally, a \textbf{subform} of $B$ is simply its restriction to $W^2$ for some linear subspace $W$ of $V$.

\begin{Def}
Assume that $\charac(\F) \neq 2$.
Let $B$ and $B'$ be non-degenerate symmetric bilinear forms (on finite-dimensional vector spaces over $\F$).
The following conditions are equivalent:
\begin{itemize}
\item Every nonisotropic part of $B'$ is isomorphic to a subform of $-B$;
\item One has $\nu(B')+\nu(B' \bot B) \geq \rk(B')$.
\end{itemize}
When they are satisfied, we say that $B$ \textbf{Witt-simplifies} $B'$.
\end{Def}

To see that the conditions from this definition are equivalent, remember
that, given a non-degenerate symmetric bilinear form $\varphi$,
the (non-degenerate) form $\varphi \bot B'$ has its Witt index greater than or equal to $\rk(\varphi)$
if and only if $\varphi$ is equivalent to a subform of $-B'$.
So, letting $\varphi'$ be a nonisotropic part of $B'$, and letting $r$ denote the Witt index of $B'$,
we have $B' \simeq \varphi' \bot H'$ for some hyperbolic bilinear form $H'$ of rank $2r$,
and hence
$B \bot B' \simeq H' \bot (\varphi' \bot B)$, to the effect
that $\nu(B \bot B')=r+\nu(\varphi' \bot B)$, whereas $\rk B'=2r+\rk \varphi'$.
Hence, $\nu(B')+\nu(B' \bot B) \geq \rk(B')$ if and only if $\nu(\varphi' \bot B) \geq \rk(\varphi')$,
which is equivalent to having $\varphi'$ equivalent to a subform of $-B$.

\begin{theo}\label{theo:symformnilpotent}
Assume that $\charac(\F) \neq 2$.
Let $\eta \in \{-1,1\}$, and let $(b,u)$ be a $(1,\eta)$-pair in which $u$ is nilpotent.
The following conditions are equivalent:
\begin{enumerate}[(i)]
\item $(b,u)$ has the square-zero splitting property.
\item For every integer $k \geq 0$, the non-degenerate symmetric bilinear form
$\eta^k (b,u)_{t,2k+1}$ Witt-simplifies $\underset{i >k}{\bot}\eta^i (b,u)_{t,2i+1}$.
\end{enumerate}
\end{theo}

Notice the lack of constraints on the $(b,u)_{t,2k}$ invariants.
From there, we can combine the results on nilpotent endomorphisms and on automorphisms to obtain the full characterization:

\begin{theo}\label{theo:symsym}
Assume that $\charac(\F) \neq 2$.
Let $(b,u)$ be a $(1,1)$-pair. For $u$ to be the sum of two $b$-selfadjoint square-zero endomorphisms, it is necessary and sufficient that all the following
conditions hold:
\begin{enumerate}[(i)]
\item For every $p \in \Irr_0(\F)$ and every integer $r \geq 1$, the quadratic form $(b,u)_{p,r}$
is skew-representable.
\item For every $p \in \Irr(\F) \setminus (\Irr_0(\F) \cup \{t\})$ and every integer $k \geq 1$,
the quadratic invariants $(b,u)_{p,k}$ and $(-1)^{1+(k-1)\deg p}(b,u)_{p^{\op},k}^\op$ are equivalent over the field $\F[t]/(p)$.
\item For every integer $k \geq 0$, the non-degenerate symmetric bilinear form
$(b,u)_{t,2k+1}$ Witt-simplifies $\underset{i >k}{\bot} (b,u)_{t,2i+1}$.
\end{enumerate}
\end{theo}

\begin{theo}\label{theo:symalt}
Assume that $\charac(\F) \neq 2$.
Let $(b,u)$ be a $(1,-1)$-pair. For $u$ to be the sum of two $b$-skew-selfadjoint square-zero endomorphisms, it is necessary and sufficient that all the following conditions hold:
\begin{enumerate}[(i)]
\item For every $p \in \Irr_0(\F)$ and every integer $r \geq 1$, the Hermitian form $(b,u)_{p,r}$
is hyperbolic.
\item For every $p \in \Irr(\F) \setminus (\Irr_0(\F) \cup \{t\})$ and every integer $k \geq 1$,
the Jordan number $n_{p,k}(u)$ is even.
\item For every integer $k \geq 0$, the non-degenerate symmetric bilinear form
$(-1)^k (b,u)_{t,2k+1}$ Witt-simplifies $\underset{i >k}{\bot}(-1)^i (b,u)_{t,2i+1}$.
\end{enumerate}
\end{theo}

\subsection{Examples}\label{section:examplesquadratic}

In our examples, all the fields are assumed to be of characteristic different from $2$.
We will not comment much on the situation of $(-1,\eta)$-pairs because, as seen in Theorems
\ref{theo:altformantisym} and \ref{theo:altformalt}, spotting which ones have the square-zero splitting property is quite easy.
We will only focus on the other two types of pairs.

\subsubsection{Algebraically closed fields}

Assume that $\F$ is algebraically closed. Then the irreducible polynomials over $\F$
are of degree $1$, and none of them is even. Moreover, non-degenerate symmetric bilinear forms over $\F$
are classified up to equivalence by their rank. Hence, $(1,\eta)$-pairs $(b,u)$
are classified up to isometry by the Jordan numbers of $u$. It follows that
a $(1,-1)$-pair $(b,u)$, with $u$ bijective, has the square-zero splitting property if and only
if all its Jordan numbers are even, whereas a $(1,1)$-pair $(b,u)$ with $u$ bijective has the square-zero splitting property if and only if $u$ is similar to its opposite.

For $(1,\eta)$-pairs $(b,u)$ with $u$ nilpotent, the characterization of Theorem \ref{theo:symformnilpotent}
takes a simplified form. Indeed, over $\F$ a non-degenerate symmetric bilinear form has
its nonisotropic parts of dimension $0$ if its rank is even, and of dimension $1$ otherwise,
and if it is $1$-dimensional then it is equivalent to a subform of every non-degenerate symmetric bilinear form of rank at least $1$. Hence, condition (ii) can be replaced by:
for all $k \geq 0$, if $\sum_{i>k} n_{t,2i+1}(u)$ is odd then $n_{t,2k+1}(u)>0$.
This yields the following characterization over algebraically closed fields:

\begin{theo}\label{theo:symsymalgclosed}
Assume that $\F$ is algebraically closed with $\chi(\F) \neq 2$.
Let $(b,u)$ be a $(1,1)$-pair. Then $(b,u)$ has the square-zero splitting property
if and only if the following conditions hold:
\begin{enumerate}[(i)]
\item $u$ is similar to $-u$, i.e.\ $n_{t-\lambda,k}(u)=n_{t+\lambda,k}(u)$ for all $k \geq 1$
and all $\lambda \in \F \setminus \{0\}$.
\item For all $k \geq 0$, if $\sum_{i>k} n_{t,2i+1}(u)$ is odd then $n_{t,2k+1}(u)>0$.
\end{enumerate}
\end{theo}

\begin{theo}\label{theo:symaltalgclosed}
Assume that $\F$ is algebraically closed with $\chi(\F) \neq 2$.
Let $(b,u)$ be a $(1,-1)$-pair. Then $(b,u)$ has the square-zero splitting property
if and only the following conditions hold:
\begin{enumerate}[(i)]
\item For all $\lambda \in \F \setminus \{0\}$ and all $k \geq 1$, the Jordan number $n_{t-\lambda,k}(u)$ is even.
\item For all $k \geq 0$, if $\sum_{i>k} n_{t,2i+1}(u)$ is odd then $n_{t,2k+1}(u)>0$.
\end{enumerate}
\end{theo}

For example, for $n \geq 1$ denote by $B_n$ the pair consisting of the bilinear form on $(\F^n)^2$
and of the endomorphism represented, in the canonical basis,
respectively by
$$K_n:=\begin{bmatrix}
0 & & 1 \\
 & \ddots & \\
1 & & 0
\end{bmatrix} \in \Mat_n(\F) \quad \text{and} \quad J_n:=\begin{bmatrix}
0 & 1 & &  (0) \\
0 & \ddots & \ddots & \\
 & & \ddots & 1\\
(0) & & 0 & 0
\end{bmatrix} \in \Mat_n(\F),$$
so that $(K_n)_{i,j}=1$ if $i+j=n+1$, and $(K_n)_{i,j}=0$ otherwise.
Then $B_n$ is a $(1,1)$-pair with exactly one Jordan cell of order $n$.

Theorem \ref{theo:symsymalgclosed} shows that $B_1 \bot B_3 \bot B_5$ and $B_5 \bot B_7$ have the square-zero splitting property as $(1,1)$-pairs, but $B_3 \bot B_7$ does not because the second condition fails for $k=2$ (the pair $B_3 \bot B_7$ should have at least one Jordan cell for the pair $(t,5)$).

\subsubsection{Real numbers}

Here, the situation is already more complicated. For a non-degenerate symmetric bilinear form
$B$ over the reals, the equivalence class of $B$ is determined by the inertia $(s_+(B),s_-(B))$,
where $s_+(B)$ (respectively, $s_-(B)$) is the greatest dimension for a subspace on which $B$ is positive definite (respectively, negative definite). In that case the Witt index of $B$ equals
$\min(s_+(B),s_-(B))$, and its nonisotropic parts have inertia $(|s_+(B)-s_-(B)|,0)$ if $s_+(B) \geq 0$, and
$(0,|s_+(B)-s_-(B)|)$ otherwise.

Now, for any irreducible polynomial $p$ of $\R[t]$ with degree more than $1$,
we actually have $\deg(p)=2$ and the field $\R[t]/(p)$ is isomorphic to $\C$.
Let $p \in \Irr_0(\R)$, i.e.\ $p(t)=t^2+a$ for some positive real number $a$.
Since every non-zero complex number has a square root, it follows that every non-degenerate symmetric bilinear form over $\R[t]/(p)$ is skew-representable (the associated quadratic form takes all possible nonzero values in $\R[t]/(p)$ if its rank is nonzero).
Hence in Theorem \ref{theo:symsym}, condition (i) is automatically satisfied.
For condition (ii), we need to split the discussion, whether the irreducible $p \in \Irr(\R) \setminus \Irr_0(\R)$
has degree $1$ or $2$:
\begin{itemize}
\item If $p$ has degree 2 then, given an integer $r \geq 1$,
the quadratic invariants $(b,u)_{p,r}$ and $(-1)^{1+(r-1)\deg p}(b,u)_{p^{\op},r}^\op$ are equivalent over the field $\R[t]/(p) \simeq \C$
if and only if they have the same rank, i.e.\ $n_{p,r}(u)=n_{p^\op,r}(u)$.
\item If $p$ has degree $1$ then the fields $\R[t]/(p)$ and $\R[t]/(p^\op)$ are naturally isomorphic to $\R$,
the invariants $(b,u)_{p,r}$ and $(b,u)_{p^\op,r}$ are then naturally seen as real (non-degenerate) symmetric bilinear forms, and the condition that
$(b,u)_{p,r}$ and $(-1)^{1+(r-1)\deg p}(b,u)_{p^{\op},r}^\op$ are equivalent over the field $\R[t]/(p)$ means that
$(b,u)_{p,r}$ has the same inertia as $(-1)^r (b,u)_{p^\op,r}$.
\end{itemize}

Then, in Theorem \ref{theo:symalt}, conditions (i) and (ii) can be summed up in the fact that the Jordan number
$n_{p,r}(u)$ is even for all $p \in \Irr(\R) \setminus \{t\}$ and all $r \geq 1$.

Finally, condition (ii) in Theorem \ref{theo:symformnilpotent} can be translated into conditions on inertia, as follows:
\begin{itemize}
\item The case where $\eta=1$. For all $k \geq 0$,
if
$$\sum_{i=k+1}^{+\infty} s_+((b,u)_{t,2i+1}) \geq \sum_{i=k+1}^{+\infty} s_-((b,u)_{t,2i+1}),$$
then
$$s_-((b,u)_{t,2k+1}) \geq \sum_{i=k+1}^{+\infty} \Bigl[s_+((b,u)_{t,2i+1})-s_-((b,u)_{t,2i+1})\Bigr],$$
otherwise
$$s_+((b,u)_{t,2k+1}) \geq \sum_{i=k+1}^{+\infty} \Bigl[s_-((b,u)_{t,2i+1})-s_+((b,u)_{t,2i+1})\Bigr].$$

\item The case where $\eta=-1$. For all $k \geq 0$,
if
$$\sum_{i=1}^{+\infty} (-1)^i s_+((b,u)_{t,2k+2i+1})\geq \sum_{i=1}^{+\infty} (-1)^i s_-((b,u)_{t,2k+2i+1})$$
then
$$s_-((b,u)_{t,2k+1}) \geq \sum_{i=1}^{+\infty} (-1)^i \Bigl[s_+((b,u)_{t,2k+2i+1}) - s_-((b,u)_{t,2k+2i+1})\Bigr],$$
otherwise
$$s_+((b,u)_{t,2k+1}) \geq \sum_{i=1}^{+\infty} (-1)^i \Bigl[s_-((b,u)_{t,2k+2i+1}) - s_+((b,u)_{t,2k+2i+1})\Bigr].$$
\end{itemize}

\subsubsection{Finite fields}

Here, we assume that $\F$ is finite with $\charac(\F) \neq 2$. Then $\F^*/(\F^*)^{[2]}$ is a group of order $2$
(where $(\F^*)^{[2]}$ denotes the set of all squares in $\F^*$).
Given a non-degenerate symmetric bilinear form $b$ on an $n$-dimensional vector space over a finite field $\L$ with characteristic different from $2$, we denote by $\det_\L(b)$ the class in $\L^*/(\L^*)^{[2]}$ of $\det M$, where $M$ is an arbitrary matrix that represents $b$,
and we denote by $\delta_\L(b):=(-1)^{n(n-1)/2} \det_\L(b)$ the \emph{discriminant} of $b$ (using discriminants rather than determinants is a good idea because
the discriminant is unchanged when taking the orthogonal sum of $b$ with a hyperbolic form).
It is known that two non-degenerate symmetric bilinear forms on $\L$ are equivalent if and only if they have the same discriminant
and their underlying spaces have the same dimension (i.e.\ they have the same rank).

In contrast, over a finite field equipped with a non-identity involution, even possibly of characteristic $2$
(i.e.\ a field of cardinality $p^{2d}$ for some prime number $p$ and some integer $d \geq 1$),
non-degenerate Hermitian forms are classified solely by their rank (because in such a field $\F$, when we set $\K:=\{x \in \F : \; x^\bullet=x\}$,
the norm $x \mapsto xx^\bullet$ induces a surjection from $\F^*$ to $\K^*$). From there, in Theorem \ref{theo:symalt}, the combination of conditions (i) and (ii) is equivalent to having $n_{p,k}(u)$ even for all $p \in \Irr(\F) \setminus \{t\}$ and all $k \geq 1$.

Now, let $p \in \Irr_0(\F)$. Set $\K:=\{q \in \F[t]/(p) : \; q^\bullet=q\}$, where $q \mapsto q^\bullet$ is the
involution that takes the class of $t$ to the one of $-t$.
Then every element of $\K$ has a square-root in the field $\L:=\F[t]/(p)$, and so either all the elements of
$\{q \in \F[t]/(p) : \; q^\bullet=-q\}$ are squares in $\L$, or none is a square in $\L$.
Set $s:=|\F|$ and $d:=\frac{1}{2}\,\deg p$. Let $x$ be a generator of the cyclic group $(\L^*,\times)$.
Then, for all $k \geq 0$, the condition $(x^k)^\bullet=-x^k$ is equivalent to $(x^k)^{s^d}=-x^k$ and hence to $x^{k(s^d-1)}=-1$
i.e.\ to $x^{k(s^d-1)}=x^{\frac{s^{2d}-1}{2}}$, and we deduce that for $k:=\frac{s^d+1}{2}$ the element $x^k$
is skew-Hermitian. Moreover, this very element has a square-root in $\L$ if and only if $s^d \equiv -1\; [4]$,
i.e.\ $s \equiv  -1\; [4]$ and $d$ is odd.
Hence:
\begin{itemize}
\item If $s \equiv 1 \; [4]$ or $\deg p$ is a multiple of $4$, then a non-degenerate quadratic form of rank $r$ over $\L$
is skew-representable if and only if its determinant equals $[\gamma]^r$, where $\gamma$ is an arbitrary non-square in $\L$.

\item If $s \equiv -1 \; [4]$ and $\deg p$ is not a multiple of $4$, then a quadratic form over $\L$
is skew-representable if and only if its determinant equals $1$.

\end{itemize}
Besides, condition (ii) in Theorem \ref{theo:symsym} can be rephrased as saying that $u$ is similar to $-u$ and that
$\delta_\L((b,u)_{p,k})=(-1)^{(1+(k-1)\deg p)n_{p,k}(u)}\bigl(\delta_{\L^\op}((b,u)_{p^{\op},k})\bigr)^\op$ for all $p \in \Irr(\F) \setminus \Irr_0(\F)$
and all $k \geq 1$ (where $\L=\F[t]/(p)$).

Finally, condition (ii) in Theorem \ref{theo:symformnilpotent} can be rephrased, for all $k \geq 1$, as follows:
\begin{itemize}
\item If $\underset{i=k+1}{\overset{+\infty}{\sum}} n_{t,2i+1}(u)$ is odd, then
either $n_{t,2k+1}(u)>1$, or $n_{t,2k+1}(u)=1$ and $\delta((b,u)_{t,2k+1}) =-\delta\left(\underset{i>k}{\bot}\eta^{i-k} (b,u)_{t,2i+1}\right)$;
\item If $\underset{i=k+1}{\overset{+\infty}{\sum}} n_{t,2i+1}(u)$ is even and
$\delta\left(\underset{i>k}{\bot}\eta^{i-k} (b,u)_{t,2i+1}\right) \neq 1$, then
either $n_{t,2k+1}(u)>2$, or $n_{t,2k+1}(u)=2$ and $\delta((b,u)_{t,2k+1}) \neq 1$.
\end{itemize}

\subsection{Structure of the article}

The remainder of the article is structured as follows.
In Section \ref{section:construction}, we develop two main ways to construct interesting selfadjoint/skew-selfadjoint endomorphisms.
The simpler is the first, which we call the hyperbolic/symplectic extensions of an endomorphism, and it turns out that this method, combined
with the characterization of sums of two square-zero endomorphisms, is just enough to obtain Theorem \ref{theo:altformalt}.
For the other theorems, a more subtle construction is required: the boxed sums of a pair of symmetric/alternating bilinear forms on the same vector space. It turns out (Proposition \ref{prop:caracboxedsum}) that pairs $(b,u)$ with the square-zero splitting property in which $u$ is bijective are systematically isometric to boxed sums. And conversely, boxed sums satisfy the square-zero splitting property. This pushes us to engage in a thorough study of boxed sums and their quadratic/Hermitian invariants (Section \ref{section:invariantsboxedsum}).

Using these two constructions, we complete the study for symplectic forms (i.e.\ $(-1,\eta)$-pairs) in Section \ref{section:symplectic}, and the study of automorphisms for symmetric bilinear forms
in Section \ref{section:symmetricautomorphism}.

However, for symmetric forms the case of nilpotent endomorphisms is difficult, and hyperbolic/symplectic extensions and boxed sums are not sufficient to completely account for that case.
The difficult case of $(1,\eta)$-pairs $(b,u)$ in which $u$ is nilpotent is dealt with in Section \ref{section:nilpotent}.

The last section is devoted to the solution of the Hermitian version of the problem, in which we consider pairs
$(h,u)$ where $h$ is a non-degenerate Hermitian form and $u$ is an $h$-selfadjoint endomorphism.

\section{Two key constructions}\label{section:construction}

\subsection{Review of duality theory}\label{section:duality}

Here we recall some basic notation and facts on dual spaces. Let $V$ be a finite-dimensional vector space
over $\F$. We denote by $V^\star:=\Hom(V,\K)$ its dual space. It has the same dimension as $V$.
Given a linear subspace $W$ of $V$, we denote by $W^o:=\{\varphi \in V^\star : \; \forall x \in W, \; \varphi(x)=0\}$
its orthogonal in $V^\star$. Given a linear subspace $H$ of $V^\star$, we denote by ${}^o H:=\{x \in V : \; \forall \varphi \in H^\star, \; \varphi(x)=0\}$
its pre-orthogonal in $V$. One finds that $\dim W+\dim W^o=\dim V=\dim H+\dim {}^o H$, as well as the double-dual relations
${}^o(W^o)=W$ and $({}^o H)^o=H$.

Given an endomorphism $u$ of $V$, the transpose of $u$ is defined as
$$u^t : \varphi \in V^\star \mapsto \varphi \circ u \in V^\star.$$
A classical consequence of the Frobenius normal form is that $u^t$ is similar to $u$.

\subsection{Hyperbolic and symplectic extensions of an endomorphism}\label{expansionsection}

Let $V$ be a finite-dimensional vector space, and let $\varepsilon \in \{-1,1\}$.
On the product space $V \times V^\star$, we consider the bilinear form
$$H_V^\varepsilon : \begin{cases}
(V \times V^\star)^2 & \longrightarrow \F \\
((x,\varphi),(y,\psi)) & \longmapsto \varphi(y)+\varepsilon \psi(x).
\end{cases}$$
Note that $H_V^1$ is a non-degenerate hyperbolic symmetric bilinear form, and it can be seen as the
polar form\footnote{The polar form of a quadratic form $Q$ is defined as $(x,y) \mapsto Q(x+y)-Q(x)-Q(y)$. We avoid dividing by $2$ to take
fields of characteristic $2$ into account.} of the canonical hyperbolic quadratic form.
In contrast, $H_V^{-1}$ is a symplectic form.

Let $u \in \End(V)$ and $\eta \in \{1,-1\}$.
We consider the endomorphism
$$h_\eta(u) : \begin{cases}
V \times V^\star & \longrightarrow V \times V^\star \\
(x,\varphi) & \longmapsto (u(x),\eta\, u^t(\varphi)).
\end{cases}$$
Note that for all $(x,\varphi)$ and $(y,\psi)$ in $V \times V^\star$,
\begin{align*}
H_V^\varepsilon\bigl((x,\varphi),h_\eta(u)[y,\psi]\bigr) & =\varphi(u(y))+\varepsilon\, \eta \,u^t(\psi)(x) \\
& =\varepsilon\, \eta\, \psi(u(x))+u^t(\varphi)[y] \\
& =\eta\, H_V^\varepsilon(h_\eta(u)[x,\varphi],(y,\psi)).
\end{align*}
In particular, if $\varepsilon\eta=-1$ then $h_\eta$ is $H_V^\varepsilon$-alternating.
We conclude that $(H_V^\varepsilon,h_\eta(u))$ is an $(\varepsilon,\eta)$-pair.

Moreover, the mapping $h_\eta : \End(V) \rightarrow \End(V \times V^\star)$ is linear and takes any square-zero
operator to a square-zero operator. Consequently, we have the following result:

\begin{prop}\label{remark:expansionsquarezero}
Let $u \in \End(V)$ be the sum of two square-zero endomorphisms of $V$.
Then $(H_V^\varepsilon,h_\eta(u))$ has the square-zero splitting property for all $(\varepsilon,\eta) \in \{-1,1\}^2$.
\end{prop}

The similarity class of the endomorphism $h_\eta(u)$ is easily deduced from the one of $u$ because $u^t$ is similar to $u$.
In particular, if the invariant factors of $u$ are $p_1,\dots,p_r$, then the invariants factors of $h_1(u)$ are $p_1,p_1,\dots,p_r,p_r$.
For $h_{-1}(u)$, the invariant factors are harder to compute and we prefer to think in terms of Jordan numbers:
simply $n_{p,k}(h_{-1}(u))=n_{p,k}(u) +n_{p^\op,k}(u)$ for all $p \in \Irr(\F)$ and all $k \geq 1$.

We shall further simplify the notation as follows:

\begin{Not}
For $(\varepsilon,\eta) \in \{-1,1\}^2$ and an endomorphism $u$ of a vector space $V$, we set
$$H_{\varepsilon,\eta}(u):=(H_V^\varepsilon,h_\eta(u)).$$
\end{Not}

Next, we examine the behavior of the pair $H_{\varepsilon,\eta}(u)$ with respect to similarity and direct sums.
Let $v \in \End(V)$ and $w \in \End(W)$ be similar endomorphisms, with an isomorphism $\varphi : V \overset{\simeq}{\rightarrow} W$
such that $w=\varphi \circ u \circ \varphi^{-1}$. Then, for the isomorphism $\Phi:=\varphi \oplus (\varphi^{-1})^t$
from $V \times V^\star$ to $W \times W^\star$, we have, for all $(x,f)$ and $(y,g)$ in $V \times V^\star$,
\begin{align*}
H_W^\varepsilon(\Phi(x,f),\Phi(y,g)) & =H_W^\varepsilon\bigl((\varphi(x),f \circ \varphi^{-1}),(\varphi(y),g \circ \varphi^{-1})\bigr) \\
& =(f \circ \varphi^{-1})(\varphi(y))+\varepsilon (g \circ \varphi^{-1})(\varphi(x)) \\
& =f(y)+\varepsilon g(x)=H_V^\varepsilon\bigl((x,f),(y,g)\bigr).
\end{align*}
Moreover, for all $(x,f)\in V \times V^\star$, we have
\begin{align*}
h_\eta(w)[\Phi(x,f)]
& =\bigl((\varphi \circ v \circ \varphi^{-1})[\varphi(x)],(\eta\,f \circ \varphi^{-1}) \circ (\varphi \circ v \circ \varphi^{-1})]\bigr) \\
& =(\varphi(v(x)),\eta\,f \circ v\circ \varphi^{-1})=\Phi(h_\eta(v)[x,f]),
\end{align*}
leading to $h_\eta(w)=\Phi \circ h_\eta(v) \circ \Phi^{-1}$.
Hence, we have proved that
$$H_{\varepsilon,\eta}(v) \simeq H_{\varepsilon,\eta}(w).$$
In other words, the isometry ``class" of the pair $H_{\varepsilon,\eta}(u)$ depends only on the pair
$(\varepsilon,\eta)$ and on the similarity ``class" of $u$.

Finally, we examine the behavior of $H_{\varepsilon,\eta}(u)$ with respect to direct sums.
So let $u_1 \in \End(V_1)$ and $u_2 \in \End(V_2)$ be endomorphisms.
We consider the direct sum $u:=u_1 \oplus u_2 \in \End(V_1 \times V_2)$.
We shall prove that
$$H_{\varepsilon,\eta}(u_1 \oplus u_2) \simeq H_{\varepsilon,\eta}(u_1) \bot H_{\varepsilon,\eta}(u_2).$$
We introduce the canonical injections $i_1 : V_1 \hookrightarrow V_1 \times V_2$ and $i_2 : V_2 \hookrightarrow V_1 \times V_2$,
and we consider the isomorphism
$$\Psi : \begin{cases}
(V_1 \times V_2) \times (V_1 \times V_2)^\star & \longrightarrow (V_1 \times V_1^\star) \times (V_2 \times V_2^\star) \\
((x_1,x_2),\varphi) & \longmapsto \bigl((x_1,\varphi \circ i_1),(x_2,\varphi \circ i_2)\bigr).
\end{cases}$$
Set $G:=H_{V_1}^{\varepsilon} \bot H_{V_2}^{\varepsilon}$.
For all $(x_1,x_2)$ and $(y_1,y_2)$ in $V_1 \times V_2$ and all $\varphi,\psi$ in $(V_1 \times V_2)^\star$, we see that
\begin{align*}
G(\Psi((x_1,x_2),\varphi),\Psi((y_1,y_2),\psi))
& =\varphi(y_1,0)+\varepsilon \psi(x_1,0)+\varphi(0,y_2)+\varepsilon \psi(0,x_2) \\
& =\varphi(y_1,y_2)+\varepsilon \psi(x_1,x_2)
\end{align*}
and hence $\Psi$ is an isometry from $H_{V_1 \times V_2}^{\varepsilon}$
to $H_{V_1}^{\varepsilon} \bot H_{V_2}^{\varepsilon}$.
Finally, for all $(x_1,x_2) \in V_1 \times V_2$ and all $\varphi \in (V_1 \times V_2)^\star$, we have
$$(h_\eta(u_1) \oplus h_\eta(u_2))\bigl(\Psi((x_1,x_2),\varphi)\bigr)
=\Bigl(\bigl(u_1(x_1),\eta\,\varphi \circ i_1 \circ u_1\bigr),\bigl(u_2(x_2),\eta\,\varphi \circ i_2 \circ u_2\bigr)\Bigr)$$
whereas
$$h_\eta(u_1\oplus u_2)[(x_1,x_2),\varphi]=
\bigl((u_1(x_1),u_2(x_2)),\eta\,\varphi \circ (u_1 \oplus u_2)\bigr)$$
so that
\begin{align*}
\Psi(h_\eta(u_1\oplus u_2)[(x_1,x_2),\varphi])
& =\bigl((u_1(x_1),\eta\,\varphi \circ (u_1 \oplus u_2) \circ i_1),(u_2(x_2),\eta\,\varphi \circ (u_1 \oplus u_2) \circ i_2)\bigr) \\
& =\bigl((u_1(x_1),\eta\,\varphi \circ i_1 \circ u_1),(u_2(x_2),\eta\,\varphi \circ i_2 \circ u_2)\bigr).
\end{align*}
Hence, we have proved that $h_\eta(u_1) \oplus h_\eta(u_2)=\Psi \circ h_\eta(u_1\oplus u_2) \circ \Psi^{-1}$,
which yields the claimed isometry.

We are now ready to investigate the invariants of the pair $H_{\varepsilon,\eta}(u)$.

\begin{prop}\label{prop:hyperbolicexpansion}
Let $u \in \End(V)$ and $(\varepsilon,\eta) \in \{-1,1\}^2$.
Then the quadratic/Hermitian invariants of $H_{\varepsilon,\eta}(u)$ are all hyperbolic.
\end{prop}

\begin{proof}
By the primary canonical form, we can split $u \simeq u_1 \oplus \cdots \oplus u_r$ where every $u_i$  is cyclic with
minimal polynomial a power of some monic irreducible polynomial.
Using the above two principles, we deduce that
$$H_{\varepsilon,\eta}(u) \simeq H_{\varepsilon,\eta}(u_1) \bot \cdots \bot H_{\varepsilon,\eta}(u_r).$$
Remember that the quadratic/Hermitian invariants are turned to equivalent invariants if we replace a pair with an isometric pair,
and that they are compatible with orthogonal sums. Remember finally that the orthogonal sum of two hyperbolic quadratic or Hermitian forms
is hyperbolic.

Hence, in order to prove the result it suffices to consider the case where $u$ is cyclic and its minimal polynomial equals $p^r$
for some $p \in \Irr(\F)$ and some integer $r \geq 1$.

We start by proving the result for the quadratic invariants of $f:=H_{1,1}(u)$.
First of all, since the invariant factors of $h_1(u)$ are $p^r,p^r$, all the quadratic invariants of $f$ vanish
with the exception of $f_{p,r}$, which is a $2$-dimensional non-degenerate symmetric bilinear form over the field $\L:=\F[t]/(p)$.
To see that $f_{p,r}$ is hyperbolic, it suffices to exhibit a non-zero isotropic vector for this form.
Note that $q(h_1(u))=h_1(q(u))$ for every $q \in \F[t]$.
For an element $q \in \F[t]$, we denote its class modulo $p$ by $\overline{q}$.
We equip $V \times V^\star$ with the structure of $\F[t]$-module induced by $h_1(u)$.

Here, we have $\Ker \bigl(p(h_1(u))^{r-1}\bigr)=\im p(h_1(u))$ and $V \times V^\star=\Ker p(h_1(u))^r=\Ker p(h_1(u))^{r+1}$, so the
starting space of $f_{p,r}$ is the quotient space $(V \times V^\star)/\im p(h_1(u))$.
Obviously $\im p(h_1(u))=\im p(u) \times \im p(u)^t$.
For all $x,y$ in $V$, the vectors $(x,0)$ and $(y,0)$ are orthogonal for the
bilinear form
$$B: (X,Y) \mapsto H_V^1\bigl(X,h_1(u)^{r-1}(Y)\bigr)$$
on $V \times V^\star$.
Let us choose $x \in V \setminus \im p(u)$. Then for all $q(t) \in \F[t]$ the vector
$q(t)\,(x,0)=(q(u)[x],0)$ is $B$-orthogonal to $(x,0)$.
Now, denote by $X$ the class of $(x,0)$ modulo $\im p(h_1(u))$. Denote
by $\overline{B}$ the non-degenerate bilinear form induced by $B$ on the quotient $\F[t]$-module $(V \times V^\star)/\im p(h_1(u))$.
Then $q(t).X$ is $\overline{B}$-orthogonal to $X$ for all $q(t)\in \F[t]$.
Considering $(V \times V^\star)/\im p(h_1(u))$ with its structure of vector space over $\L:=\F[t]/(p)$,
this means that $\overline{B}(X,\lambda\,X)=0$ for all $\lambda \in \F$.
In other words, for all $\lambda \in \F$,
$$0=\overline{B}(X,\lambda\, X)=e_p\bigl(\lambda f_{p,r}(X,X)\bigr),$$
which yields $f_{p,r}(X,X)=0$. Hence $f_{p,r}$ is isotropic, and we conclude that it is hyperbolic.

There are no quadratic invariants for $H_{-1,1}(u)$, so we turn to $g:=H_{\varepsilon,-1}(u)$ for a given $\varepsilon \in \{-1,1\}$.
For the quadratic invariants, we note that if $p=t$ then $h_{-1}(u)$ is nilpotent with exactly two Jordan cells,
both of size $r$. Thus, all the quadratic invariants of $g$ vanish if $r = \frac{1-\varepsilon}{2}$ mod 2,
otherwise the sole quadratic invariant of $g$ is $g_{t,r}$, which is a $2$-dimensional non-degenerate symmetric bilinear form:
the proof that $g_{t,r}$ is hyperbolic is then entirely similar to the previous proof.

Finally, if $p$ is an even polynomial then one checks that $h_{-1}(u)$ has exactly two invariant factors, namely $p^r$ and $p^r$. Again, from there the proof that the $g_{p,r}$ invariants are all hyperbolic is similar to the previous proof.
\end{proof}

We can immediately give two applications of that construction for our problem:

\begin{cor}\label{extensionhyperboliccor1}
Let $\eta \in \{-1,1\}$.
Let $(b,u)$ be a $(1,\eta)$-pair in which $u$ is nilpotent.
Assume that all the quadratic invariants of $(b,u)$ are hyperbolic.
Then $(b,u)$ has the square-zero splitting property.
\end{cor}

\begin{proof}
Since all the quadratic invariants of $(b,u)$ are hyperbolic, all the corresponding Jordan numbers are even.
Moreover, if $\eta=-1$ it is known that all the remaining Jordan numbers are even (because they are dimensions of spaces equipped with symplectic forms). Hence, all the Jordan numbers of $u$ are even.

Thus, we can consider a vector space $V$ together with a nilpotent endomorphism $v \in \End(V)$
such that $n_{t,k}(v)=\frac{1}{2} n_{t,k}(u)$ for every integer $k \geq 1$.
Then $h_\eta(v)$ is nilpotent and has the same Jordan numbers as $u$.
Moreover, all the quadratic invariants of $H_{1,\eta}(v)$ are hyperbolic.
Besides, two hyperbolic quadratic forms are equivalent whenever they have the same rank.
Hence, $(b,u)$ and $H_{1,\eta}(v)$ have the same quadratic invariants,
and we conclude that $(b,u) \simeq H_{1,\eta}(v)$.

Finally, since $v$ is nilpotent Theorem \ref{theoBotha} shows that it is the sum of two square-zero endomorphisms of $V$.
By Proposition \ref{remark:expansionsquarezero}, the pair $H_{1,\eta}(v)$ has the square-zero splitting property, and hence so does
$(b,u)$.
\end{proof}

\begin{cor}\label{extensionhyperboliccor2}
Let $(b,u)$ be a $(-1,-1)$-pair in which $u$ is nilpotent.
Assume that $u$ has only Jordan cells of odd size. Then $u$ is the sum of two square-zero $b$-skew-selfadjoint endomorphisms.
\end{cor}

\begin{proof}
The proof is similar to the one of Corollary \ref{extensionhyperboliccor1}, but this time around we use the fact that
the Jordan number $n_{t,2k+1}(u)$ is even for every integer $k \geq 0$, because the quadratic invariant $(b,u)_{t,2k+1}$
is a symplectic form.
\end{proof}

\subsection{Boxed sums: an introduction}\label{section:boxedsum}

Throughout this section and in the remainder of Section \ref{section:construction}, we will systematically assume that
$\chi(\F) \neq 2$.

Here, we let $V$ be an arbitrary vector space, and we let $\varepsilon\in \{-1,1\}$.

Let $b$ be a bilinear form on $V$ (at this point we do not assume that it is symmetric or alternating, and $b$ can be degenerate as well).
We consider the associated linear map $L_b : x \in V \mapsto b(-,x) \in V^\star$, and we define the endomorphism
$$v_b : (x,\varphi) \in V \times V^\star \mapsto (0,L_b(x)) \in V \times V^\star,$$
which is obviously of square zero.
For all $(x,\varphi)$ and $(y,\psi)$ in $V \times V^\star$, we see that
$$H_V^\varepsilon \bigl(v_b(x,\varphi),(y,\psi)\bigr)=L_b(x)[y]=b(y,x).$$
Hence:
\begin{itemize}
\item If $\varepsilon=1$ then $v_b$ is $H_V^\varepsilon$-selfadjoint (respectively, skew-selfadjoint) if and only $b$ is symmetric (respectively,  skew-symmetric).
\item If $\varepsilon=-1$ then $v_b$ is $H_V^\varepsilon$-selfadjoint (respectively, skew-selfadjoint) if and only $b$ is skew-symmetric (respectively, symmetric).
\item In any case, $v_b$ is $H_V^\varepsilon$-alternating if and only if $b$ is alternating.
\end{itemize}
Assume furthermore that $b$ is non-degenerate, to the effect that $L_b$ is an isomorphism. In that case we consider the endomorphism
$$w_b : (x,\varphi) \in V \times V^\star \mapsto (L_b^{-1}(\varphi),0) \in V \times V^\star,$$
again obviously of square zero.
Let $(x,\varphi)$ and $(y,\psi)$ in $V \times V^\star$. Setting $x',y'$ in $V$ such that $L_b(x')=\varphi$
and $L_b(y')=\psi$, we find that
$$H_V^\varepsilon \bigl(w_b(x,\varphi),(y,\psi)\bigr)=\varepsilon\, \psi(x')=\varepsilon\, b(x',y').$$
Again, we have the following equivalences:
\begin{itemize}
\item If $\varepsilon=1$ then $w_b$ is $H_V^\varepsilon$-selfadjoint (respectively, skew-selfadjoint) if and only $b$ is symmetric (respectively,  skew-symmetric).
\item If $\varepsilon=-1$ then $w_b$ is $H_V^\varepsilon$-selfadjoint (respectively, skew-selfadjoint) if and only $b$ is skew-symmetric (respectively, symmetric).
\item In any case, $w_b$ is $H_V^\varepsilon$-alternating if and only if $b$ is alternating.
\end{itemize}

\begin{Def}
Let $b,c$ be two bilinear forms on a finite-dimensional vector space $V$, with $b$ non-degenerate.
We define their \textbf{boxed sum} as the endomorphism
$$b \boxplus c:=v_c+w_b \in \End(V \times V^\star).$$
\end{Def}

Note that $b \boxplus c$ is an automorphism of $V \times V^\star$ if and only if $c$ is non-degenerate (in which case its inverse is simply $c \boxplus b$).
Again, one checks that for all $(x,\varphi)$ and $(y,\psi)$ in $V \times V^\star$,
$$H_V^\varepsilon \bigl((b\boxplus c)(x,\varphi),(y,\psi)\bigr)=c(y,x)+\varepsilon\, b(x',y')$$
where $x':=L_b^{-1}(\varphi)$ and $y':=L_b^{-1}(\psi)$.
The following equivalences are then easily obtained:
\begin{itemize}
\item If $\varepsilon=1$ then $b \boxplus c$ is $H_V^\varepsilon$-selfadjoint (respectively, skew-selfadjoint) if and only both $b$ and $c$ are symmetric (respectively, skew-symmetric).
\item If $\varepsilon=-1$ then $b \boxplus c$ is $H_V^\varepsilon$-selfadjoint (respectively, skew-selfadjoint) if and only both $b$ and $c$ are skew-symmetric (respectively, symmetric).
\item In any case, $b \boxplus c$ is $H_V^\varepsilon$-alternating if and only if both $b$ and $c$ are alternating.
\end{itemize}

Moreover, from the above properties of $w_b$ and $v_c$, we obtain:
\begin{itemize}
\item If $b$ and $c$ are symmetric then $(H_V^\varepsilon,b \boxplus c)$ has the square-zero splitting property as a $(\varepsilon,\varepsilon)$-pair.
\item If $b$ and $c$ are alternating then $(H_V^\varepsilon,b \boxplus c)$ has the square-zero splitting property as a $(\varepsilon,-\varepsilon)$-pair.
\end{itemize}

Better still, there is almost a converse result:

\begin{prop}\label{prop:caracboxedsum}
Let $(\varepsilon,\eta) \in \{1,-1\}^2$, and let $(b,u)$ be an $(\varepsilon,\eta)$-pair with the square-zero splitting property in which $u$ is an automorphism.
\begin{itemize}
\item If $\varepsilon=\eta$ then $(b,u) \simeq (H_W^{\varepsilon},B \boxplus C)$ for some vector space $W$
and some pair $(B,C)$ of non-degenerate symmetric bilinear forms on $W$.
\item If $\varepsilon=-\eta$ then $(b,u) \simeq (H_W^{\varepsilon},B \boxplus C)$ for some vector space $W$
and some pair $(B,C)$ of symplectic forms on $W$.
\end{itemize}
\end{prop}

\begin{proof}
Denote by $V$ the underlying vector space of the pair $(b,u)$.
There exist square-zero endomorphisms $u_1$ and $u_2$ of $V$ such that each $(b,u_i)$ is an
$(\varepsilon,\eta)$-pair and $u=u_1 +u_2$.

Let us consider the subspaces $V_1:=\im u_1$ and $V_2:=\im u_2$. We have $V_1 \subset \Ker u_1=V_1^{\bot_b}$
and likewise $V_2 \subset V_2^{\bot_b}$, i.e.\ $V_1$ and $V_2$ are totally $b$-singular, leading to
$\dim V_1 \leq \frac{1}{2} \dim V$ and  $\dim V_2 \leq \frac{1}{2} \dim V$.
Moreover, $V_1+V_2=V$ because $u$ is surjective, whence $V=V_1 \oplus V_2$,
$V_1=V_1^{\bot_b}$ and $V_2=V_2^{\bot_b}$. It follows that $u_1$ maps $V_2$ bijectively onto $V_1$,
and $u_2$ maps $V_1$ bijectively onto $V_2$.

Next, as $V_2$ is totally $b$-singular, the isomorphism $L_b: x \mapsto b(-,x)$ induces an
injective linear mapping $V_2 \rightarrow (V_2)^\circ$, which turns out to be an isomorphism because $2\dim V_2=\dim V$.
Composing it with the natural (restriction) isomorphism from $(V_2)^\circ$ to $V_1^\star$,
we obtain an isomorphism $f : V_2 \rightarrow V_1^\star$ that takes every $x \in V_2$ to the linear form $y \in V_1 \mapsto b(y,x)$.
Then, we consider the isomorphism
$$\Phi : \begin{cases}
V & \longrightarrow V_1 \times V_1^\star \\
x_1+x_2 & \longmapsto \bigl(\varepsilon x_1,f(x_2)\bigr) \quad \text{with $x_1 \in V_1$ and $x_2 \in V_2$.}
\end{cases}$$
Next, for all $x,y$ in $V$, we split $x=x_1+x_2$ and $y=y_1+y_2$ with $x_1,y_1$ in $V_1$ and $x_2,y_2$ in $V_2$,
and we see that
\begin{align*}
b(x,y) & =b(x_1,y_2)+b(x_2,y_1) \\
       & =b(x_1,y_2)+\varepsilon\,b(y_1,x_2) \\
       & =f(y_2)[x_1]+\varepsilon\, f(x_2)[y_1] \\
       & =H_{V_1}^\varepsilon\bigl(\Phi(x),\Phi(y)\bigr).
\end{align*}
Now, set $u'_1:=\Phi \circ u_1 \circ \Phi^{-1}$ and $u'_2 := \Phi \circ u_2 \circ \Phi^{-1}$.
As $u_1$ vanishes everywhere on $V_1$ and maps $V_2$ bijectively onto $V_1$,
we find that $u'_1$ vanishes everywhere on $V_1 \times \{0\}$ and maps $\{0\} \times V_1^\star$ bijectively onto $V_1 \times \{0\}$.
The associated isomorphism from $V_1^\star$ to $V_1$ then reads $L_B^{-1}$ for a unique non-degenerate bilinear form $B$ on $V_1$,
and hence $u'_1=w_B$.
Likewise, $u'_2$ vanishes everywhere on $\{0\} \times V_1^\star$ and maps $V_1 \times \{0\}$ bijectively onto $\{0\} \times V_1^\star$: the associated linear mapping from $V_1$ to $V_1^\star$ reads $L_C$ for some bilinear form $C$ on $V_1$, and hence $u'_2=v_C$.

Finally, $\Phi \circ u \circ \Phi^{-1}=B \boxplus C$, and we conclude that
$(b,u) \simeq (H_V^\varepsilon,B \boxplus C)$. It follows that $(H_V^\varepsilon,B \boxplus C)$ is an $(\varepsilon,\eta)$-pair.
From there, we deduce that if $\varepsilon=\eta$ then $B$ and $C$ are symmetric, otherwise $B$ and $C$ are alternating.
\end{proof}

Hence, in order to solve our problem for automorphisms, it essentially
remains to compute the quadratic/Hermitian invariants of a boxed sum!

Before we can do that, additional general considerations are required on the dependency of
$(H_V^\varepsilon,b \boxplus c)$ on the pair $(b,c)$.

\subsection{Boxed sums as functions of pairs of forms}

Let $(b,c)$ and $(b',c')$ be two pairs of bilinear forms, on respective vector spaces $V,V,V',V'$.

We say that $(b,c)$ is isometric to $(b',c')$ when there exists a vector space isomorphism
$\varphi : V \overset{\simeq}{\rightarrow} V'$ such that
$$\forall (x,y) \in V^2, \; b'\bigl(\varphi(x),\varphi(y)\bigr)=b(x,y) \; \text{and} \;
c'\bigl(\varphi(x),\varphi(y)\bigr)=c(x,y).$$
This defines an equivalence relation on the collection of pairs of bilinear forms with the same underlying vector space.

Next, the orthogonal direct sum of two such pairs $(b_1,c_1)$ and $(b_2,c_2)$ is defined as
$(b_1,c_1) \bot (b_2,c_2) :=(b_1 \bot b_2,c_1 \bot c_2)$, and one checks that it is compatible with isometry
(if we replace a summand with an isometric one, then the result is unchanged up to an isometry).

We shall say that a pair is indecomposable if it is non-trivial (i.e.\ defined on a non-zero vector space) and
not isometric to the orthogonal direct sum of two non-trivial pairs.

We shall now examine the effect, in boxed sums, of replacing a pair $(b,c)$ with an isometric one.

\begin{lemma}\label{cpequivalencelemma}
Let $\varepsilon \in \{-1,1\}$. Let $(b,c)$ be a pair of bilinear forms on a vector space $V$, and
$(b',c')$ be a pair of bilinear forms on a vector space $V'$. Assume that $b$ and $b'$ are non-degenerate and that
$(b,c) \simeq (b',c')$. Then
$$(H_V^{\varepsilon},b \boxplus c) \simeq (H_{V'}^{\varepsilon},b' \boxplus c').$$
\end{lemma}

\begin{proof}
Let us consider an isometry $\varphi : V \overset{\simeq}{\rightarrow} V'$ from $(b,c)$ to $(b',c')$.
Note in particular that $L_{c}=\varphi^t \circ L_{c'} \circ \varphi$ and $L_{b}=\varphi^t \circ L_{b'} \circ \varphi$.
We have already seen in Section \ref{expansionsection} that the vector space isomorphism
$\Phi:=\varphi \oplus (\varphi^{-1})^t$ yields an isometry from $H_V^{\varepsilon}$ to $H_{V'}^{\varepsilon}$.
Moreover, for all $(y,g)$ in $V' \times (V')^\star$,
\begin{multline*}
v_c(\Phi^{-1}(y,g))=v_c(\varphi^{-1}(y),g \circ \varphi) =\bigl(0,L_c(\varphi^{-1}(y))\bigr) \\
=(0,(\varphi^t \circ L_{c'})(y))=\Phi^{-1}(0,L_{c'}(y))=\Phi^{-1}(v_{c'}(y,g)),
\end{multline*}
that is
$$v_c=\Phi^{-1} \circ v_{c'} \circ \Phi.$$
Likewise, one checks that $w_b=\Phi^{-1} \circ w_{b'} \circ \Phi$.
Hence, by summing we obtain $b \boxplus c=\Phi^{-1} \circ (b' \boxplus c') \circ \Phi$, and the expected conclusion follows.
\end{proof}

\begin{lemma}\label{cpdirectsumlemma}
Let $\varepsilon \in \{-1,1\}$. Let $(b_1,c_1)$ and $(b_2,c_2)$ be pairs of bilinear forms on respective vectors spaces $V_1$ and $V_2$. Then,
$$\bigl(H_{V_1 \times V_2}^{\varepsilon},(b_1\bot b_2) \boxplus (c_1 \bot c_2)\bigr) \simeq
(H_{V_1}^\varepsilon,b_1 \boxplus c_1) \bot (H_{V_2}^\varepsilon,b_2 \boxplus c_2).$$
\end{lemma}

\begin{proof}
We introduce the canonical injections $i_1 : V_1 \hookrightarrow V_1 \times V_2$ and $i_2 : V_2 \hookrightarrow V_1 \times V_2$,
and we consider the isomorphism
$$\Phi : \begin{cases}
(V_1 \times V_2) \times (V_1 \times V_2)^\star & \longrightarrow (V_1 \times V_1^\star) \times (V_2 \times V_2^\star) \\
((x_1,x_2),\varphi) & \longmapsto \bigl((x_1,\varphi \circ i_1),(x_2,\varphi \circ i_2)\bigr).
\end{cases}$$
It was already proved in Section \ref{expansionsection} that
$\Phi$ yields an isometry from $H_{V_1 \times V_2}^{\varepsilon}$ to $H_{V_1}^\varepsilon \bot H_{V_2}^\varepsilon$.
It remains to check that $(b_1 \boxplus c_1) \oplus (b_2 \boxplus c_2) = \Phi \circ ((b_1 \bot b_2)\boxplus (c_1 \bot c_2)) \circ \Phi^{-1}$.
To see this, it suffices to check that $(v_{c_1}\oplus v_{c_2})  \circ \Phi=\Phi \circ v_{(c_1 \bot c_2)}$
and $(w_{b_1} \oplus  w_{b_2}) \circ \Phi=\Phi \circ w_{(b_1 \bot b_2)}$. We will only check the former relation, the proof of the latter being similar.
So, let $(x_1,x_2,\varphi) \in V_1 \times V_2 \times (V_1 \times V_2)^\star$.
Then,
$$(v_{c_1} \oplus  v_{c_2})(\Phi((x_1,x_2),\varphi))=\bigl(v_{c_1}(x_1,\varphi \circ i_1),v_{c_2}(x_2,\varphi \circ i_2)\bigr)
=\bigl((0,L_{c_1}(x_1)),(0,L_{c_2}(x_2))\bigr).$$
Besides,
$$\Phi\bigl(v_{c_1\bot c_2}((x_1,x_2),\varphi)\bigr)=\bigl((0,L_{c_1\bot c_2}(x_1,x_2) \circ i_1),(0,L_{c_1\bot c_2}(x_1,x_2) \circ i_2)\bigr)$$
and we conclude by noting that $L_{c_1\bot c_2}(x_1,x_2) \circ i_1=L_{c_1}(x_1)$ and
$L_{c_1\bot c_2}(x_1,x_2) \circ i_2=L_{c_2}(x_2)$.
\end{proof}

Now, we need to better understand the isometry type of pairs $(b,c)$ of bilinear forms with the same parity (both symmetric or both alternating),
with $b$ non-degenerate. It turns out that this question is a reformulation of the classification of $(\varepsilon,1)$-pairs!
Indeed, let $\varepsilon \in \{1,-1\}$, and let $(b,c)$ be a pair of bilinear forms on a vector space $V$, both symmetric if $\varepsilon=1$, and
both alternating if $\varepsilon=-1$. Then $u:=L_b^{-1} \circ L_c$ is an endomorphism of $V$, and one checks that
$(b,u)$ is a $(\varepsilon,1)$-pair. Conversely, given an $(\varepsilon,1)$-pair $(b,u)$, one sees that
$c : (x,y) \mapsto b(x,u(y))$ is a bilinear form that is symmetric if $\varepsilon=1$, and alternating if $\varepsilon=-1$.
Moreover, one checks that, given pairs $(b,c)$ and $(b',c')$ of bilinear forms, any isometry from $(b,c)$ to $(b',c')$
turns out to be an isometry from $(b,L_b^{-1}\circ L_c)$ to $(b',L_{b'}^{-1}\circ L_{c'})$. Conversely, given
$(\varepsilon,1)$-pairs $(b,u)$ and $(b',u')$, any isometry from $(b,u)$ to $(b',u')$ turns out to be an isometry from
$(b,c)$ to $(b',c')$ where $c : (x,y) \mapsto b(x,u(y))$ and $c' : (x,y) \mapsto b'(x,u'(y))$.

Finally, let $(b,c)$ and $(b',c')$ be two pairs of bilinear forms on respective vector spaces $V$ and $V'$, and assume that both $b$
and $b'$ are non-degenerate. Then, one checks that $L_{b \bot b'}^{-1}\circ L_{c \bot c'} =(L_b^{-1}\circ L_c) \oplus (L_{b'}^{-1}\circ L_{c'})$.
Hence, the previous correspondence literally takes orthogonal direct sums to orthogonal direct sums.

Hence, the classification of pairs of symmetric (respectively, alternating) bilinear forms (with the first form non-degenerate)
comes entirely down to the one of $(1,1)$-pairs (respectively, of $(-1,1)$-pairs).
In particular, pairs of symmetric bilinear forms with the first component non-degenerate are classified by quadratic invariants.

\begin{Rem}\label{squareboxed}
One checks that $(b \boxplus c)^2=h_1(L_b^{-1} \circ L_c)$.
\end{Rem}

\subsection{The invariants of a boxed sum}\label{section:invariantsboxedsum}

We start with an interesting observation which is especially relevant to the special case of $(1,1)$-pairs:

\begin{prop}\label{prop:isomopposite}
Let $(b,c)$ be a pair of bilinear forms on a vector space $V$, with $b$ non-degenerate.
Then the pairs $(H_V^1, b \boxplus c)$ and $(-H_V^1, -b \boxplus c)$ are isometric.
\end{prop}

\begin{proof}
Consider the automorphism $\varphi:=h_{-1}(\id_V)$ of $V \times V^\star$.
It is skew-selfadjoint for $H_V^1$ and satisfies $\varphi^2=\id$, whence
$\varphi \varphi^\star=-\id$ (where the adjoint is taken with respect to $H_V^1$).
This leads to $H_V^1\bigl(\varphi(x,f),\varphi(y,g)\bigr)=-H_V^1\bigl((x,f),(y,g)\bigr)$
for all $x,y$ in $V$ and all $f,g$ in $V^\star$.
Next, one easily checks that $\varphi \circ w_b \circ \varphi^{-1}=-w_b=w_{-b}$ and that
$\varphi \circ v_c \circ \varphi^{-1}=v_{-c}$, which leads to $\varphi \circ (b \boxplus c) \circ \varphi^{-1}=-b \boxplus c$.
Hence, $\varphi$ is the required isometry.
\end{proof}

We are now ready to compute the invariants of $(H_V^\varepsilon,b \boxplus c)$
when $(b,c)$ is a pair of bilinear forms on the vector space $V$, both symmetric or both alternating, and with $b$ non-degenerate.
Using the compatibility of boxed sums with orthogonal direct sums and with isometry, we will limit the computation to the
case where $(b,c)$ is indecomposable.
There are only two special cases to consider:
\begin{itemize}
\item If $b$ and $c$ are symmetric, $b$ is non-degenerate and $(b,c)$ is indecomposable, then $u:=L_b^{-1} L_c$ is cyclic and its minimal
polynomial equals $p^r$ for some $p \in \Irr(\F)$ and some $r \geq 1$. In that case, exactly one quadratic invariant of $(b,c)$
is non-trivial, namely $(b,c)_{p,r}$, and it is a non-degenerate symmetric bilinear form on a $1$-dimensional vector space over the field $\F[t]/(p)$.

\item If $b$ and $c$ are symplectic, $b$ is non-degenerate and $(b,c)$ is indecomposable, then $u:=L_b^{-1} L_c$
is the direct sum of two cyclic endomorphisms with minimal polynomial $p^r$ for some $p \in \Irr(\F)$ and some $r \geq 1$.
\end{itemize}

\begin{lemma}\label{lemma:invariantsboxedsum}
Let $b$ and $c$ be bilinear forms on a vector space $V$, with $b$ non-degenerate.
Denote by $p_1,\dots,p_r$ the invariant factors of $L_b^{-1} L_c$. Then
the invariant factors of $b \boxplus c$ are $p_1(t^2),\dots,p_r(t^2)$.
\end{lemma}

\begin{proof}
Consider a basis $(e_1,\dots,e_n)$ of $V$. Then,
$$\mathbf{B}:=((0,L_b(e_1)),\dots,(0,L_b(e_n)),(e_1,0),\dots,(e_n,0))$$
is a basis of $V \times V^\star$ and one checks that the matrix of $b \boxplus c$ in $\mathbf{B}$ equals
$$M:=\begin{bmatrix}
[0]_{n \times n} & A \\
I_n & [0]_{n \times n}
\end{bmatrix},$$
where $A$ stands for the matrix of $L_b^{-1} L_c$ in $(e_1,\dots,e_n)$. From there, it is known (see the proof of corollary
A.3 in \cite{dSPsum3})
that the invariants factors of $M$ are $p_1(t^2),\dots,p_r(t^2)$.
\end{proof}

Now, we finally compute the relevant quadratic/Hermitian invariants of $(H_V^\varepsilon,b \boxplus c)$ when
$(b,c)$ is indecomposable. We start with the case where $b,c$ are symplectic forms.

\begin{lemma}\label{lemma:hyperbolicboxedsum}
Let $(b,c)$ be an indecomposable pair of symplectic forms on a vector space $V$.
Then all the Hermitian invariants of $(H_V^1 ,b \boxplus c)$
are hyperbolic. Moreover, all the Jordan numbers of $b \boxplus c$ are even.
\end{lemma}

\begin{proof}
Here, there exists $p_0 \in \Irr(\F) \setminus \{t\}$ and an integer $r \geq 1$ such that
 $u:=L_b^{-1} L_c$ has exactly two primary invariants: $p_0^r$ and $p_0^r$.
 Lemma \ref{lemma:invariantsboxedsum} shows that the invariant factors of $b \boxplus c$ are $p_0(t^2)^r$ and $p_0(t^2)^r$.
 It easily follows that the Jordan numbers of $b \boxplus c$ are all even.

Now, we have two cases: either $p:=p_0(t^2)$ is irreducible or $p_0(t^2)=q q^{\op}$ for some $q \in \Irr(\F)$
such that $q\neq q^\op$. In the second case, all the Hermitian invariants of $(H_V^1 ,b \boxplus c)$ vanish.
So, we assume that the first case holds. Then $p^r$ and $p^r$ are the primary invariants of $(H_V^1 ,b \boxplus c)$,
and exactly one Hermitian invariant of $(H_V^1 ,b \boxplus c)$ is non-trivial, namely $(H_V^1 ,b \boxplus c)_{p,r}$:
it is a Hermitian form defined on a $2$-dimensional vector space over the field $\L:=\F[t]/(p)$ equipped with the involution that takes the class $\overline{t}$ to its opposite.

In the remainder of the proof, we equip $V$ with the $\F[t]$-module structure induced by $u$, and $V \times V^\star$
with the one induced by $b \boxplus c$.

In order to show that $(H_V^1 ,b \boxplus c)_{p,r}$ is hyperbolic, it suffices to exhibit a nonzero isotropic vector for it.
In our situation, the relevant quadratic form that is used to build $(H_V^1 ,b \boxplus c)_{p,r}$ is the bilinear form
$\overline{B}$ induced by
$$B : (X,Y) \mapsto H_V^1 (X,p^{r-1}\,Y)$$
on the quotient vector space  $(V \times V^\star) / \im p(b \boxplus c)$
(the details are the same as in the proof of Proposition \ref{prop:hyperbolicexpansion}).
Remembering (see Remark \ref{squareboxed}) that $(b \boxplus c)^2=h_1(u)$,
we find that $p(b \boxplus c)=h_1(p_0(u))$ and hence
$\im p(b \boxplus c)=\im p_0(u) \times \im p_0(u)^t$.
Let us choose $x \in V \setminus \im p_0(u)$, set $X:=(x,0)$ and denote by $\overline{X}$ the class of $X$ modulo $\im p(b \boxplus c)$.
Note that $\overline{X} \neq 0$. We shall prove that $X$ is isotropic for $(H_V^1 ,b \boxplus c)_{p,r}$, which will complete the proof.

Noting that $(b \boxplus c)^2(x,0)=(u(x),0)$, we obtain
$\forall q \in \F[t], \; q(t^2)\,(x,0)=(q\,x,0)$ and it follows that
$$\forall q \in \F[t], \; H_V^1\bigl((x,0),q(t^2)\,(x,0)\bigr)=0.$$
Let $q \in \F[t]$. Then, $tq(t^2).(x,0)=\bigl(0,L_c(q\,x)\bigr)$.
Hence $H_V^1\bigl((x,0),tq(t^2).(x,0)\bigr)=L_c(q(u)[x])[x]=c(x,q(u)[x])=b(x, u q(u)[x])$.
Since $b$ and $c$ are alternating and $u$ is $b$-selfadjoint, we see that $b(x,u^k(x))=0$ for every integer
$k \geq 0$: indeed, if $k=2l$ one writes $b(x,u^k(x))=b(u^l(x),u^l(x))=0$, if $k=2l+1$ one writes $b(x,u^k(x))=b(u^{l}(x),u^{l+1}(x))=c(u^l(x),u^{l}(x))=0$.
Hence,
$$H_V^1\bigl((x,0),tq(t^2).(x,0)\bigr)=b(x,uq(u)[x])=0.$$
Combining the above two points yields $H_V^1((x,0),\lambda\, (x,0))=0$ for all $\lambda \in \F[t]$.

It follows that, in the quotient $\F[t]$-module $V/\im p(b \boxplus c)$,
the class $\overline{X}$ generates a submodule that is orthogonal to the submodule $\F[t]\,\overline{X}$
for the bilinear form induced by $(Y,Z) \mapsto H_V^1(Y,p^{r-1}\,Z)$.

By coming back to the definition of the Hermitian invariant of order $r$ with respect to $p$ (see Section \ref{section:invariants}),
we get that $(H_V^1,b \boxplus c)_{p,r}(\overline{X},\overline{X})=0$, which completes the proof.
\end{proof}

We turn to boxed sums of symmetric forms, for which the situation is more complicated.

\begin{lemma}\label{lemma:skewrepresentabledim1}
Let $(b,c)$ be a pair of symmetric bilinear forms on a vector space $V$, with $b$ non-degenerate.
Assume that $u:=L_b^{-1}L_c$ is cyclic with minimal polynomial $p_0^r$ for some $p_0 \in \Irr(\F)$ and some $r \geq 1$, and assume furthermore
that $p:=p_0(t^2)$ is irreducible. Assume finally that the quadratic invariant $(b,u)_{p_0,r}$ represents, for some polynomial $s(t)\in \F[t]$ (not divisible by $p_0$), the class of $s(t)$ modulo $p_0$.
Then, for all $\varepsilon \in \{-1,1\}$, the pair $(H_V^\varepsilon,b \boxplus c)$ has exactly one quadratic/Hermitian invariant, namely $(H_V^\varepsilon,b \boxplus c)_{p,r}$, which is defined on a
$1$-dimensional vector space, and it represents the class of $\varepsilon t s(t^2)$ in the field $\F[t]/(p)$.
\end{lemma}

\begin{proof}
Again, the sole invariant factor of $b \boxplus c$ is $p^r$ because the sole one of $u$ is $p_0^r$.

We consider the fields $\mathbb{M}:=\F[t]/(p_0)$ and $\mathbb{L}:=\F[t]/(p)$.
For a polynomial $a(t) \in \F[t]$, we shall denote by $\overline{a(t)}^{\mathbb{L}}$ its class modulo $p$, and by
$\overline{a(t)}^{\mathbb{M}}$ its class modulo $p_0$.

Let $\varepsilon \in \{-1,1\}$.
Like in the proofs of previous lemmas, the form $\overline{B}$ that is used to construct the quadratic/Hermitian invariant
$(H_V^\varepsilon,b \boxplus c)_{p,r}$ is defined on the quotient space $(V \times V^\star)/(\im p_0(u) \times \im p_0(u)^t)$,
whereas the one that is used to construct the quadratic invariant $(b,u)_{p_0,r}$ is defined
on the quotient space $V /\im p_0(u)$. The assumption on $s(t)$ shows that we can choose
a vector $x$ of $V \setminus \im p_0(u)$ such that
$$\forall q \in \F[t], \; b(x,(qp_0^{r-1})(u)[x])=e_{p_0}(\overline{s(t)q(t)}^{\mathbb{M}}).$$
To simplify the notation, we now endow $V \times V^\star$ with the structure of $\F[t]$-module induced by $b \boxplus c$,
whereas $V$ is endowed with the structure of $\F[t]$-module associated with $u$.
Let us consider the vector $X:=(x,0) \in V \times V^\star$.
Note that $p^{r-1}\,X=(p_0^{r-1}\,x,0)$

Let $q \in \F[t]$, and split $q=q_1(t^2)+t q_2(t^2)$ for $q_1(t),q_2(t)$ in $\F[t]$.
Remembering that $(b \boxplus c)^2=h_1(u)$, we see that
$j(t^2)\,X=(j\,x,0)$ for all $j \in \F[t]$, and it follows that
$$(p^{r-1}q)\,X=q\,(p^{r-1}\,X)=\bigl(q_1 p_0^{r-1}\,x,L_c(q_2p_0^{r-1}\,x)\bigr).$$
Hence,
$$H_V^\varepsilon(X,qp^{r-1}\,X)=\varepsilon\,c(x,q_2p_0^{r-1}\,x)=\varepsilon\,b(x,tq_2 p_0^{r-1}\,x)=\varepsilon\, e_{p_0}(\overline{tq_2 s}^{\mathbb{M}}).$$
Now, if we perform the Euclidean division $t q_2  s=Q p_0+R$ with $\deg R<\deg p_0$, then
$t^2 q_2(t^2) s(t^2)=Q(t^2) p+R(t^2)$ and $\deg R(t^2)<\deg p$, which leads to
$$e_{p}(\overline{t^2 q_2(t^2) s(t^2)}^{\L})=R(0^2)=R(0)=e_{p_0}\bigl(\overline{t q_2 s}^{\M}\bigr).$$
Noting that $t \, q_1(t^2) s(t^2)$ is an odd polynomial, we see that its remainder mod $p$ is also odd (because $p$ is even), and hence
$$e_{p}\bigl(\overline{t q_1(t^2) s(t^2)}^{\L}\bigr)=0.$$
By summing, we end up with
$$H_V^\varepsilon(X,p^{r-1}\,(q\,X))=\varepsilon\,e_{p_0}(\overline{t q_2 s}^{\M})
=\varepsilon\,e_{p}(\overline{t q(t) s(t^2)}^{\L})
=e_p\bigl(\overline{q}^\L \varepsilon\,\overline{t s(t^2)}^\L\bigr).$$
Hence, by denoting by $\overline{X}$ the class of $X$ modulo $\im p(b \boxplus c)$,
we have found that
$$(H_V^\varepsilon,b \boxplus c)_{p,r}(\overline{X},\overline{X})=\varepsilon\,\overline{t s(t^2)}^{\L},$$
which yields the claimed result because $\overline{X} \neq 0$.
\end{proof}

\begin{cor}\label{cor:skewrepresentable}
Let $(b,c)$ be a pair of non-degenerate symmetric bilinear forms on a vector space $V$.
Let $p \in \Irr_0(\F)$ and $r \geq 1$ be an integer. Then the quadratic invariant $(H_V^1,b \boxplus c)_{p,r}$
is skew-representable over the field $\L:=\F[t]/(p)$ equipped with the involution that takes the class of $t$ to its opposite.
\end{cor}

\begin{proof}
Let us split the pair $(b,c) \simeq (b_1,c_1) \bot \cdots \bot (b_n,c_n)$
into an orthogonal sum of irreducible pairs of symmetric bilinear forms (defined on respective vector spaces $V_1,\dots,V_n$). Then we know that
$(H_V^1,b \boxplus c) \simeq (H_{V_1}^1,b_1 \boxplus c_1) \bot \cdots \bot (H_{V_n}^1,b_n \boxplus c_n)$, to the effect that
$(H_V^1,b \boxplus c)_{p,r} \simeq (H_{V_1}^1,b_1 \boxplus c_1)_{p,r} \bot \cdots \bot (H_{V_n}^1,b_n \boxplus c_n)_{p,r}$.
Hence, it suffices to prove the result when $(b,c)$ is indecomposable as a pair of symmetric bilinear forms.

So, assume now that $(b,c)$ is indecomposable. Then it is known that $u:=L_b^{-1} L_c$ is cyclic with minimal polynomial $q^s$
for some $q \in \Irr(\F)$ and some $s>0$. If $p=q(t^2)$, then Lemma \ref{lemma:skewrepresentabledim1} shows that
$(H_V^1,b \boxplus c)_{p,r}$ is a skew-representable symmetric bilinear form (of rank $1$).
Otherwise $p$ does not divide $q(t^2)$, and $(H_V^1,b \boxplus c)_{p,r}$ has rank $0$ (and is therefore skew-representable).
\end{proof}

Here is now a sort of converse statement:

\begin{cor}\label{prop:skewrepresentablereciproc}
Let $p \in \Irr_0(\F)$ and $r>0$. Let $B$ be a skew-representable symmetric bilinear form over the field $\L:=\F[t]/(p)$
(for the involution $x \mapsto x^\bullet$ that takes the class of $t$ to its opposite).
Then there exists a vector space $W$ (over $\F$) and a pair $(b,c)$ of non-degenerate symmetric bilinear forms on $W$
such that:
\begin{enumerate}[(i)]
\item The invariant factors of $b \boxplus c$ all equal $p^r$;
\item $(H_W^1,b \boxplus c)_{p,r}$ is equivalent to $B$.
\end{enumerate}
\end{cor}

\begin{proof}
We choose an orthogonal basis $(e_1,\dots,e_n)$ for $B$ such that $B(e_i,e_i)^\bullet=-B(e_i,e_i)$ for all $i \in \lcro 1,n\rcro$.
In other words, for all $i \in \lcro 1,n\rcro$, there is a polynomial $s_i(t) \in \F[t]$ such that $B(e_i,e_i)$ is the class of
$t s_i(t^2)$ modulo $p$, and $p$ does not divide $s_i(t^2)$.

Let $i \in \lcro 1,n\rcro$. We can choose an indecomposable pair $(b_i,c_i)$ of non-degenerate symmetric bilinear forms (with underlying space denoted by
$V_i$) such that $u_i:=L_{b_i}^{-1} L_{c_i}$ is cyclic with minimal polynomial $p_0^r$ (where $p=p_0(t^2)$)
and the quadratic form associated with $(b_i,u_i)_{p_0,r}$ takes the value $s_i(t)$.
By Lemma \ref{lemma:invariantsboxedsum}, the boxed sum $(H_{V_i},b_i \boxplus c_i)$ is such that $b_i \boxplus c_i$ is cyclic with minimal polynomial
$p^r$, and the (rank $1$) symmetric bilinear form $(H_{V_i},b_i \boxplus c_i)_{p,r}$ represents the value $\overline{t s_i(t^2)}$,
yielding that this form is equivalent to the form $B_i$ induced by $B$ on $\L e_i$.

Setting $(b,c):=(b_1,c_1) \bot \cdots \bot (b_n,c_n)$ (with underlying space denoted by $W$), we deduce that
$(H_W^1,b \boxplus c) \simeq (H_{V_1}^1,b_1 \boxplus c_1) \bot \cdots \bot (H_{V_n}^1,b_n \boxplus c_n)$
has a sole non-zero quadratic invariant, attached to the pair $(p,r)$ and equivalent to $B_1 \bot \cdots \bot B_n$
i.e.\ to $B$.
\end{proof}

\begin{lemma}\label{lemma:representabledim1}
Let $p \in \Irr(\F) \setminus (\Irr_0(\F) \cup \{t\})$ and $r>0$.
Set $p_0 \in \Irr(\F)$ such that $p_0(t^2)=p(t)\,p^\op(t)$.
Consider the field $\L:=\F[t]/(p)$.
Let $\beta \in \L \setminus \{0\}$.
Then there exists a vector space $V$ (over $\F$) equipped with a pair $(b,c)$ of non-degenerate symmetric bilinear forms such that:
\begin{enumerate}[(i)]
\item $b \boxplus c$ is cyclic with minimal polynomial $p_0(t^2)^r=p(t)^r p^\op(t)^r$;
\item The quadratic invariant $(H_V^1, b \boxplus c)_{p,r}$ has dimension $1$ and represents the value $\beta$.
\end{enumerate}
\end{lemma}

\begin{proof}
We consider the field $\M:=\F[t]/(p_0)$.
Let us start from an arbitrary $\alpha \in \M \setminus \{0\}$.
By the classification of pairs of symmetric bilinear forms, there exists a vector space $V$ equipped with a pair of non-degenerate
symmetric bilinear forms $(b,c)$ such that $u:=L_b^{-1} L_c$ is cyclic with minimal polynomial $p_0^r$
and the quadratic invariant $(b,u)_{p_0,r}$ represents the value $\alpha$.
By Lemma \ref{lemma:invariantsboxedsum} the endomorphism $b \boxplus c$ has $p_0(t^2)^r$ as its sole invariant factor.
As $p \neq p^{\op}$, it follows that $b \boxplus c$ has exactly two Jordan cells, both of size $r$, associated respectively with $p$ and $p^\op$,
and in particular the bilinear form $(H_V^1, b \boxplus c)_{p,r}$ has rank $1$. We shall compute a specific nonzero value taken by the quadratic form attached to
$(H_V^1, b \boxplus c)_{p,r}$, a value that will happen to be an $\F$-linear function of $\alpha$.
Then, we will see that this function maps $\M \setminus \{0\}$ onto $\L \setminus \{0\}$.

We set $W:=V \times V^\star$ and $v:=b \boxplus c$.
Throughout, we equip $W$ with the $\F[t]$-module structure induced by the endomorphism $v$, and
$V$ with the $\F[t]$-module structure induced by $u$.
We choose $x$ in $V \setminus \im p_0(u)$ such that $\alpha=(b,u)_{p_0,r}(\overline{x},\overline{x})$,
where $\overline{x}$ denotes the class of $x$ in the quotient space $V/\im p_0(u)$, and we set $X:=(p^\op)^r\,(x,0) \in W$.
Hence $p^r\, X=0$, that is $X \in \Ker p^r(v)$.

We will prove first that $X \not\in \Ker p^{r-1}(v)$.
Let us write $(p^\op)^r=r_1(t^2)+t r_2(t^2)$ for polynomials $r_1,r_2$ in $\F[t]$.
Then $X=(r_1\,x,L_c(r_2\,x))$. Now, assume on the contrary that $p^{r-1}\,X=0$.
Noting that $\im (pp^\op)(v)=\Ker p^{r-1}(v) \oplus \Ker (p^\op)^{r-1}(v)$
(because of the Jordan cells structure of $v$), we would deduce that $X \in \im (pp^\op)(v)$. Yet $(pp^\op)(v)=p_0(v^2)=h_1(p_0(u))$, and hence
we would find $X \in \im p_0(u) \times V^\star$.
Therefore $r_1 x \in \im p_0(u)$. Then, $p_0$ must divide $r_1$ because $x \not\in \im p_0(u)$ and $u$ is cyclic with minimal polynomial $p_0^r$.
Hence $pp^\op=p_0(t^2)$ must divide $r_1(t^2)$, and we successively deduce that $p^\op$ divides $t r_2(t^2)$ and that $p^\op$ divides $r_2(t^2)$.
Hence $p$ divides $r_2(t^2)^\op=r_2(t^2)$, and finally $p$ divides $r_1(t^2)+t\, r_2(t^2)=(p^\op)^r$, which is absurd.
We conclude that $X \in \Ker p^r(v) \setminus \Ker p^{r-1}(v)$.

In the next step, we obtain precious information on the non-zero scalar $\gamma:=(H_V^1, b \boxplus c)_{p,r}(X,X)$ of $\L$.
Let us start from an arbitrary polynomial $q \in \F[t]$.
We compute
\begin{align*}
e_p(\overline{q}^{\mathbb{L}} \gamma)& =H_V^1(X,p^{r-1}\,(q\,X)) \\
& =H_V^1\bigl((p^{\op})^rq\,(x,0),p^\op(p_0^{r-1}(t^2)\,(x,0))\bigr) \\
& =H_V^1\bigl((p^{\op})^rq\,(x,0),p^\op(p_0^{r-1}\,x,0)\bigr) \\
& =H_V^1\bigl((p^{\op})^{r+1}q\,(x,0),(p_0^{r-1}\,x,0)\bigr).
\end{align*}
Now, split $(p^{\op})^{r+1}q=s_1(t^2)+t s_2(t^2)$ with $s_1,s_2$ in $\F[t]$.
Then $(p^{\op})^{r+1}q\,(x,0)=(s_1\,x,L_c(s_2\,x))$ and hence
\begin{multline*}
e_p(\overline{q}^{\mathbb{L}} \gamma)=c(p_0^{r-1}\,x,s_2\,x)=c(s_2\,x,p_0^{r-1}\,x) \\
= b(s_2\,x, tp_0^{r-1}\,x)=b\bigl(x,p_0^{r-1}\,(ts_2\,x)\bigr)=e_{p_0}(\overline{ts_2}^{\mathbb{M}} \alpha).
\end{multline*}
In order to conclude, it suffices to prove that $\alpha$ can be chosen so that $\gamma=\beta$.
To see this, we shall prove that $\gamma$ can be obtained as an $\F$-linear function of $\alpha$ (from $\mathbb{M}$ to $\mathbb{L}$).

We consider the linear mapping $q \in \F[t] \mapsto \overline{t s_2}^{\mathbb{M}} \in \F[t]/(p_0)$ (with the above construction).
Let $q \in \F[t]$, and write $q(p^{\op})^{r+1}=s_1(t^2)+t s_2(t^2)$ with $s_1$ and $s_2$ in $\F[t]$.

Assume that $p$ divides $q$. Then $p_0(t^2)$ divides $q(p^{\op})^{r+1}$,
and one deduces that $p_0$ divides $s_1$ and $s_2$ (to see this, we split the quotient $z$ of $q(p^{\op})^{r+1}$ by $p_0(t^2)$ into
$z(t)=r_1(t^2)+t r_2(t^2)$ for $r_1$ and $r_2$ in $\F[t]$, we recover that $q(p^{\op})^{r+1}=p_0(t^2) r_1(t^2)+t p_0(t^2) r_2(t^2)$
and hence $s_1(t^2)=p_0(t^2) r_1(t^2)$ and $s_2(t^2)=p_0(t^2) r_2(t^2)$, so that $s_1=p_0 r_1$ and $s_2=p_0 r_2$).
Hence $\overline{t s_2}^{\mathbb{M}}=0$.

Conversely, assume that $p_0$ divides $t s_2$. Then $p_0$ divides $s_2$ and hence $pp^\op=p_0(t^2)$ divides $s_2(t^2)$.
It follows that $p^\op$ divides $s_1(t^2)=q(p^{\op})^{r+1}- t s_2(t^2)$, and since $s_1(t^2)$ is even it follows that
$p_0(t^2)=pp^\op$ divides $s_1(t^2)$. Finally $p_0(t^2)$ divides $q(p^{\op})^{r+1}$ and hence $p$ divides $q$.
We conclude that the previous map induces an injective $\F$-linear mapping
$$\varphi : \F[t]/(p) \hookrightarrow \F[t]/(p_0).$$
The $\F$-vector spaces $\F[t]/(p)$ and $\F[t]/(p_0)$ have the same (finite) dimension over $\F$, and we deduce that $\varphi$ is a vector space isomorphism.
Now, consider the transposed isomorphism
$$\varphi^t :  f \in \Hom_\F(\mathbb{M},\F) \overset{\simeq}{\longmapsto} (f \circ \varphi) \in \Hom_\F(\mathbb{L},\F).$$
We know that the mappings
$$\Phi_p : \begin{cases}
\L & \longrightarrow \Hom_\F(\L,\F) \\
a & \longmapsto [x \mapsto e_p(ax)]
\end{cases} \quad \text{and} \quad
\Psi_{p_0} : \begin{cases}
\mathbb{M} & \longrightarrow \Hom_\F(\mathbb{M},\F) \\
b & \longmapsto [x \mapsto e_{p_0}(bx)]
\end{cases}$$
are isomorphisms of $\F$-vector spaces.
Hence, the composite $\Lambda:=\Phi_p^{-1} \circ \varphi^t \circ \Psi_{p_0}$ is an isomorphism of $\F$-vector spaces from
$\mathbb{M}$ to $\L$, and from the previous computation we have found that this isomorphism maps $\alpha$ to $\gamma$.

The conclusion is finally obtained by taking $\alpha:=\Lambda^{-1}(\beta)$.
\end{proof}

With the same method as in the proof of Corollary \ref{cor:skewrepresentable}, we obtain the following corollary:

\begin{cor}\label{cor:partcorquadratic}
Let $p \in \Irr(\F) \setminus (\Irr_0(\F) \cup \{t\})$ and $r\geq 1$.
Set $p_0 \in \Irr(\F)$ such that $p_0(t^2)=p p^\op$.
Consider the field $\L:=\F[t]/(p)$. Let $B$ be a non-degenerate symmetric bilinear form on a finite-dimensional vector space over $\L$.
Then there exists a triple $(V,b,c)$ consisting of a finite-dimensional vector space $V$ over $\F$ and of a pair $(b,c)$
of non-degenerate symmetric bilinear forms on $V$ such that:
\begin{enumerate}[(i)]
\item All the invariant factors of $b \boxplus c$ equal $p_0(t^2)^r$;
\item The quadratic invariant $(H_V^1,b\boxplus c)_{p,r}$ is equivalent to $B$.
\end{enumerate}
\end{cor}

We finish with an easier result.

\begin{lemma}\label{lemma:evensizenilpotent}
Let $r$ be an even positive integer, and let $\varepsilon \in \{-1,1\}$ and $\alpha \in \F \setminus \{0\}$.
Then there exist a triple $(V,b,c)$ consisting of a finite-dimensional vector space $V$ over $\F$ and of a pair $(b,c)$
of symmetric bilinear forms on $V$, with $b$ non-degenerate, such that:
\begin{enumerate}[(i)]
\item $b \boxplus c$ is nilpotent with a sole Jordan cell, of size $r$.
\item The quadratic invariant $(H_V^\varepsilon,b \boxplus c)_{t,r}$ represents the value $\alpha$.
\end{enumerate}
\end{lemma}

\begin{proof}
By the classification of pairs of symmetric bilinear forms, we can find a pair $(b,c)$ of symmetric bilinear forms on a vector space $V$
such that $b$ is non-degenerate and $u:=L_b^{-1} L_c$ is cyclic with minimal polynomial $t^{r/2}$.
By Lemma \ref{lemma:invariantsboxedsum}, $b \boxplus c$ is cyclic with minimal polynomial $t^r$, i.e.\ it is nilpotent with a sole Jordan cell, of size $r$.
Note that for all $\lambda \in \F \setminus \{0\}$, the endomorphism $b \boxplus (\lambda c)$ has the same properties
since $L_b^{-1} L_{\lambda c}=\lambda u$.

Let $\beta$ be a non-zero value represented by the quadratic invariant $(H_V^\varepsilon,b \boxplus c)_{t,r}$.
Then, for all $\lambda \in \F \setminus \{0\}$, the quadratic invariant $(H_V^\varepsilon,b \boxplus (\lambda c))_{t,r}$
equals $\lambda^{r-1} (H_V^\varepsilon,b \boxplus c)_{t,r}$, and hence it represents $\lambda^{r-1} \beta$.
Since $r$ is even, this quadratic invariant represents $\lambda^{-r} \lambda^{r-1} \beta$. Hence it suffices to take $\lambda=\beta \alpha^{-1}$,
and the pair $(b,\lambda c)$ then satisfies the expected conclusion.
\end{proof}

\section{The case of a symplectic form}\label{section:symplectic}

\subsection{Alternating endomorphisms}

Here, we will prove Theorem \ref{theo:altformalt} by taking advantage of the symplectic extensions technique.

Throughout this part, we let $(b,u)$ be a $(-1,1)$-pair.
In Theorem \ref{theo:altformalt}, the implication (i) $\Rightarrow$ (ii) is obvious.
Moreover, the implication (ii) $\Rightarrow$ (iii) is known by Theorem \ref{theoBotha}.

Let us prove that condition (iii) implies condition (i).
Assume that all the invariant factors of $u$ are even or odd. The classification of $(-1,1)$-pairs (which holds regardlessly of the characteristic of $\F$)
shows that the invariant factors of $u$ can be written $p_1,p_1,p_2,p_2,\dots,p_r,p_r,\dots$,
hence with all the $p_i$'s even or odd.
Now, let us choose an endomorphism $w$ of a vector space $W$ whose invariant factors are $p_1,p_2,\dots,p_r,\dots$.
As $h_1(w)=w \oplus w^t$ and $w^t$ is similar to $w$, we gather that
$h_1(w)$ and $u$ have the same invariant factors.
Noting that $H_{-1,1}(w)$ is a $(-1,1)$-pair, we gather from the characterization of $(-1,1)$-pairs
that $(b,u) \simeq H_{-1,1}(w)$.
And since every invariant factor of $w$ is even or odd, we deduce from Theorem \ref{theoBotha} that
$w$ is the sum of two square-zero endomorphisms of $W$, and we deduce
from Proposition \ref{remark:expansionsquarezero} that $H_{-1,1}(w)$ has the square-zero splitting property.
Hence $(b,u) \simeq H_{-1,1}(w)$ also has the square-zero splitting property.

Assume furthermore that $\charac(\F) \neq 2$. Then, by Theorem \ref{theoBotha}, there exists an automorphism $\varphi$ of $W$
such that $\varphi  w  \varphi^{-1}=-w$.
Set $\Phi:=\varphi \oplus (\varphi^{-1})^t$, which is an automorphism of $W \times W^\star$.
Note that
$$\Phi \, h_1(w) \, \Phi^{-1}=(\varphi \, w \, \varphi^{-1}, (\varphi^{-1})^t \, w^t\, \varphi^t)
=h_1(\varphi \, w \, \varphi^{-1})=h_1(-w)=-h_1(w).$$
Finally, for all $(x,f) \in W \times W^\star$ and all $(y,g) \in W \times W^\star$, we find
\begin{multline*}
H_W^{-1}\bigl(\Phi(x,f),\Phi(y,g)\bigr)
=H_W^{-1}\bigl((\varphi(x),f \circ \varphi^{-1}),(\varphi(y),g \circ \varphi^{-1})\bigr) \\
=f(y)- g(x)=H_W^{-1}\bigl((x,f),(y,g)\bigr).
\end{multline*}
Hence, $\Phi$ belongs to the symplectic group of $H_W^{-1}$.
By the isometry $(b,u) \simeq H_{-1,1}(w)$, we deduce that $u$ is conjugated to $-u$ through an element of the symplectic group of $b$.

Hence, we have proved that, in Theorem \ref{theo:altformalt}, condition (iii) implies condition (i) and that it implies condition (v) if $\charac(\F) \neq 2$.
Finally, condition (v) obviously implies condition (iv), and it known by Theorem \ref{theoBotha} that
condition (iv) implies condition (iii) if $\charac(\F) \neq 2$.
This completes the proof of Theorem \ref{theo:altformalt}.

\subsection{Skew-selfadjoint endomorphisms}

Here, we prove Theorem \ref{theo:altformantisym}.
So, we assume that $\charac(\F) \neq 2$ and
we take a pair $(b,u)$ consisting of a symplectic form $b$ and of a $b$-skew-selfadjoint endomorphism $u$.
We shall prove that $u$ is the sum of two square-zero $b$-skew-selfadjoint operators.

To do so, we can limit our study to the case where $(b,u)$ is indecomposable as a $(-1,-1)$-pair.
Then there are four possibilities:
\begin{itemize}
\item Case 1. The invariant factors of $u$ read $t^r,t^r$ for some odd integer $r \geq 1$.
\item Case 2. The primary invariants of $u$ read $p^r,(p^{\op})^r$ for some $p \in \Irr(\F) \setminus (\Irr_0(\F) \cup \{t\})$ and some
integer $r \geq 1$.
\item Case 3. $u$ is cyclic with minimal polynomial $t^r$ for some even integer $r \geq 2$.
\item Case 4. $u$ is cyclic with minimal polynomial $p^r$ for some $p \in \Irr_0(\F)$ and some integer $r \geq 1$.
\end{itemize}

\noindent \textbf{Case 1: $u$ is nilpotent with exactly two Jordan cells, both of size $r$ for some odd integer $r \geq 1$.} \\
Let us choose a cyclic endomorphism $v$ of a vector space $W$, with minimal polynomial $t^r$.
Then, $H_{-1,-1}(v)$ is a $(-1,-1)$-pair and $h_{-1}(v)$ is nilpotent with exactly two Jordan cells, both of size $r$.
By the classification of $(-1,-1)$-pairs, we deduce that $(b,u) \simeq H_{-1,-1}(v)$.
As $v$ is nilpotent, Theorem \ref{theoBotha} shows that we can split it into the sum of two square-zero endomorphisms of $W$,
and hence Proposition \ref{remark:expansionsquarezero} shows that $H_{-1,-1}(v)$ has the square-zero splitting property.
Using the isometry $(b,u) \simeq H_{-1,-1}(v)$, we deduce that $(b,u)$ has the square-zero splitting property.

\vskip 3mm
\noindent \textbf{Case 2: The invariant factors of $u$ read $p^r,(p^{\op})^r$ for some $p \in \Irr(\F) \setminus (\Irr_0(\F) \cup \{t\})$ and some
integer $r \geq 1$.} \\
Then $p p^\op=p_0(t^2)$ for some $p_0 \in \Irr(\F)$.
By the classification of pairs of symmetric bilinear forms, we can choose a pair $(B,C)$ of symmetric bilinear forms
on a vector space $W$ such that $B$ is symmetric and non-degenerate, and $L_B^{-1} L_C$ is cyclic with minimal polynomial $p_0^r$.
By Lemma \ref{lemma:invariantsboxedsum}, the boxed sum $B \boxplus C$ has $p_0(t^2)^r$ as its sole invariant factor, i.e.\ its
primary invariants are $p^r$ and $(p^\op)^r$.
Hence, by the classification of $(-1,-1)$-pairs, we find the isometry $(b,u) \simeq (H_W^{-1},B \boxplus C)$, and hence
$(b,u)$ has the square-zero splitting property.

\vskip 3mm
\noindent \textbf{Case 3: $u$ is nilpotent with a sole Jordan cell, of even size $r$.} \\
Let $\alpha \in \F \setminus \{0\}$ be represented by the quadratic invariant $(b,u)_{t,r}$.
By Lemma \ref{lemma:evensizenilpotent}, there exists a triple $(W,B,C)$ in which $W$ is a vector space and
$B,C$ are symmetric bilinear forms, with $B$ non-degenerate, such that $B \boxplus C$ is nilpotent with a single Jordan cell, of size $r$, and
the quadratic invariant $(H_W^{-1},B \boxplus C)_{t,r}$ represents the value $\alpha$.
Hence, the quadratic invariants $(H_W^{-1},B \boxplus C)_{t,r}$ and $(b,u)_{t,r}$ are equivalent. By the classification of
$(-1,-1)$-pairs, we conclude that $(b,u) \simeq (H_W^{-1},B \boxplus C)$.
From there, we conclude as in Case 2.

\vskip 3mm
\noindent \textbf{Case 4: $u$ is cyclic with minimal polynomial $p^r$ for some $p \in \Irr_0(\F)$ and some integer $r \geq 1$.} \\
Let us write $p=p_0(t^2)$ for some $p_0 \in \Irr(\F)$.
Set $\mathbb{M}:=\F[t]/(p_0)$ and $\mathbb{L}:=\F[t]/(p)$.
We equip $\mathbb{L}$ with the non-identity involution $x \mapsto x^\bullet$ that takes the class of $t$ to its opposite.

The Hermitian invariant $R:=(b,u)_{p,r}$ is a skew-Hermitian form (on a $1$-dimensional vector space),
so $R(x,x)^\bullet=-R(x,x)$ for all relevant $x$.
In particular, the mapping $x \mapsto R(x,x)$ takes the value $\overline{t s(t^2)}^{\mathbb{L}}$
for some $s \in \F[t]$ that is not divisible by $p_0$.

Next, by the classification of pairs of symmetric bilinear forms, we can choose such a pair $(B,C)$
on a vector space $W$ such that $B$ is non-degenerate, $v:=L_B^{-1}  L_C$ is cyclic with minimal polynomial
$p_0^r$, and the quadratic invariant $(B,v)_{p_0,r}$ represents the value $-\overline{s(t)}^{\mathbb{M}}$.

By Lemma \ref{lemma:skewrepresentabledim1}, $B \boxplus C$ is cyclic with minimal polynomial $p^r$,
and the Hermitian invariant $(H_W^{-1}, B \boxplus C)_{p,r}$ represents $\overline{t s(t^2)}^{\mathbb{L}}$
(and is defined on a $1$-dimensional vector space over $\L$).
By the classification of $(-1,-1)$-pairs, it follows that
$(b,u) \simeq (H_W^{-1}, B \boxplus C)$, and once more we conclude that
$(b,u)$ has the square-zero splitting property.

\section{Symmetric forms with automorphisms}\label{section:symmetricautomorphism}

Throughout this section, we assume that $\charac(\F) \neq 2$.

\subsection{Skew-selfadjoint automorphisms}

Here, we shall prove Theorem \ref{theo:symaltautomorphism}.
So, let $(b,u)$ be a $(1,-1)$-pair in which $u$ is bijective.

Assume first that $(b,u)$ has the square-zero splitting property.
By Proposition \ref{prop:caracboxedsum}, it is isometric to $(H^1_W,B \boxplus C)$ for some
vector space $W$ and some pair $(B,C)$ of symplectic forms on $W$.
We set $v:=L_B^{-1} L_C$. We know from the classification of pairs of alternating bilinear forms, with the first one non-degenerate, that
the invariant factors of $v$ are of the form $p_1,p_1,p_2,p_2,\dots$
Hence, by Lemma \ref{lemma:invariantsboxedsum} the ones of $B \boxplus C$ are $p_1(t^2),p_1(t^2),p_2(t^2),p_2(t^2),\dots$
Hence, $B \boxplus C \simeq w \oplus w$ for some endomorphism $w$ with invariant factors $p_1(t^2),p_2(t^2),\dots$,
and we gather that the Jordan numbers of $B \boxplus C$ are all even. Hence, so are the ones of $u$.

Finally, by Proposition \ref{lemma:hyperbolicboxedsum} the Hermitian invariants of $(H^1_W,B \boxplus C)$ are all hyperbolic, and hence
so are the ones of $(b,u)$.

\vskip 3mm
Conversely, assume that all the Jordan numbers of $(b,u)$ are even and that all its Hermitian invariants are hyperbolic.
Remember that $n_{p,k}(u)=n_{p^\op,k}(u)$ for all $k \geq 1$ and all $p \in \Irr(\F)$, because $u$ is $b$-skew-selfadjoint (and hence similar to its opposite).
Now, we can consider a vector space $W$ equipped with an automorphism $w$ such that:
\begin{itemize}
\item For all $p \in \Irr(\F) \setminus \{t\}$ such that $p(t^2)$ splits into $qq^\op$ for some $q \in \Irr(\F) \setminus \Irr_0(\F)$, and for all $k \geq 1$, one has $n_{p,k}(w)=n_{q,k}(u)$.
\item For all $p \in \Irr(\F) \setminus \{t\}$ such that $p(t^2)$ is irreducible, one has $n_{p,k}(w)=n_{p(t^2),k}(u)$.
\end{itemize}
Since all the Jordan numbers of $u$ are even, all the ones of $w$ are even and it follows from the classification of pairs of symplectic forms that there exists a pair
$(B,C)$ of symplectic forms on $W$ such that $L_B^{-1} L_C=w$.
It follows from the definition of $w$ that $B \boxplus C$ has the same Jordan numbers as $u$.
Finally, for every $p \in \Irr_0(\F)$ and every $k \geq 1$, the Hermitian forms $(H_W^1,B \boxplus C)_{p,k}$
and $(b,u)_{p,k}$ are hyperbolic (by Proposition \ref{lemma:hyperbolicboxedsum} for the former, by assumption for the latter), and they have the same rank. Hence $(b,u)_{p,k} \simeq (H_W^1,B \boxplus C)_{p,k}$.
It follows that $(b,u) \simeq (H_W^1,B \boxplus C)$, and we conclude that
$(b,u)$ has the square-zero splitting property.

The proof of Theorem \ref{theo:symaltautomorphism} is now complete.

\subsection{Selfadjoint automorphisms}

In that case, we need an additional result to deal with the $(b,u)_{p,r}$ invariant when $p$ is neither even nor odd.

\begin{prop}\label{prop:isomopposite2}
Let $(b,u)$ be a $(1,1)$-pair.
Let $p \in \Irr(\F) \setminus (\Irr_0(\F) \cup \{t\})$ and $r \geq 1$. Denote by $d$ the degree of $p$.
Then $(-b,-u)_{p,r}=(-1)^{1+d(r-1)} \bigl((b,u)_{p^\op,r}\bigr)^{\op}$.
\end{prop}

\begin{proof}
We consider the fields $\L:=\F[t]/(p)$ and $\L':=\F[t]/(p^\op)$ and
the isomorphism $\varphi : \L \overset{\simeq}{\longrightarrow} \L'$ that takes the class of $t$ mod $p$ to the one of $-t$ mod $p^\op$.
For a polynomial $\lambda \in \F[t]$, we denote its class in $\F[t]/(p)$ by $\overline{\lambda}^\L$, and its class in
$\F[t]/(p^\op)$ by $\overline{\lambda}^{\L'}$.

Note that $p(-u)=(-1)^d p^\op(u)$. It follows that
$\Ker p(-u)^k=\Ker p^\op(u)^k$ for all $k \geq 1$, to the effect that the quotients
$\Ker p(-u)^{k+1}/\Ker p(-u)^{k}$ and $\Ker p^{\op}(u)^{k+1}/\Ker p^\op(u)^{k}$ are equal for all $k \geq 1$.
Moreover, the cokernel of the linear map $\Ker p(-u)^{r+1}/\Ker p(-u)^{r} \rightarrow \Ker p(-u)^{r}/\Ker p(-u)^{r-1}$
induced by $p(-u)$ equals the cokernel of the linear map $\Ker p^\op(u)^{r+1}/\Ker p^\op(u)^{r} \rightarrow \Ker p^\op(u)^{r}/\Ker p^\op(u)^{r-1}$ induced by $p^\op(u)$.

Hence, the quadratic invariants $(b,u)_{p^\op,r}$ and $(-b,-u)_{p,r}$ are defined on the same $\F$-vector space $W$. We consider the induced bilinear forms $$B : (\overline{x},\overline{y})\in W^2 \mapsto b(x,p^\op(u)^{r-1}(y))$$
and
$$B' : (\overline{x},\overline{y})\in W^2 \mapsto (-b)(x,p(-u)^{r-1}(y)).$$
Note that $B'=(-1)^{1+d(r-1)}B$.
Let $x,y$ in $W$. Let $\lambda(t) \in \F[t]$.
Then
\begin{align*}
B(x,\lambda(u)[y]) & =(-1)^{1+d(r-1)} B'(x, \lambda(u)[y]) \\
& =(-1)^{1+d(r-1)} B'(x,\lambda(-t)(-u)[y]) \\
& =(-1)^{1+d(r-1)} e_{p}\bigl(\overline{\lambda(-t)}^{\L} (-b,-u)_{p,r}(x,y)\bigr) \\
& =(-1)^{1+d(r-1)} e_{p^\op}\bigl(\varphi\bigl(\overline{\lambda(-t)}^{\L} (-b,-u)_{p,r}(x,y)\bigr)\bigr) \\
& =(-1)^{1+d(r-1)} e_{p^\op}\Bigl(\overline{\lambda}^{\L'} \varphi\bigl((-b,-u)_{p,r}(x,y)\bigr)\Bigr) \\
& =e_{p^\op}\Bigl(\overline{\lambda}^{\L'}\,(-1)^{1+d(r-1)}  \varphi\bigl((-b,-u)_{p,r}(x,y)\bigr)\Bigr).
\end{align*}
It follows that
$$(b,u)_{p^\op,r}(x,y)=(-1)^{1+d(r-1)}  \varphi\bigl((-b,-u)_{p,r}(x,y)\bigr)=(-1)^{1+d(r-1)} \bigl((-b,-u)_{p,r}\bigr)^{\op}(x,y).$$
\end{proof}

We are now ready to prove Theorem \ref{theo:symsymautomorphism}.
Let $(b,u)$ be a $(1,1)$-pair in which $u$ is an automorphism.

Assume first that $(b,u)$ has the square-zero splitting property. Then, by Proposition \ref{prop:caracboxedsum} it is isometric to
$(H^1_W,B \boxplus C)$ for some vector space $W$ and some pair $(B,C)$ of non-degenerate symmetric bilinear forms on $W$.
By Proposition \ref{prop:isomopposite}, the pairs $(H^1_W,B \boxplus C)$ and $(-H^1_W,-B \boxplus C)$ are isometric, and hence
$(b,u)$ is isometric to $(-b,-u)$. It follows from Proposition \ref{prop:isomopposite2} that
$(b,u)_{p,r}\simeq (-1)^{1+d(r-1)} (b,u)_{p^\op,r}^{\op}$ for every $r \geq 1$ and every $p \in \Irr(\F) \setminus (\Irr_0(\F) \cup \{t\})$
of degree $d$.

Besides, Corollary \ref{cor:skewrepresentable} directly shows that, for all $p \in \Irr_0(\F)$, the quadratic invariant
$(H^1_W,B \boxplus C)_{p,r}$ is skew-representable over $\F[t]/(p)$ equipped with the involution that takes the class of $t$ to its opposite.

Conversely, assume that $(b,u)$ satisfies condition (ii) from Theorem \ref{theo:symsymautomorphism}.
First of all, we can split $(b,u) \simeq (b_1,u_1) \bot \cdots \bot (b_n,u_n)$
so that, for each $i \in \lcro 1,n\rcro$, one of the following conditions holds:
\begin{itemize}
\item The invariant factors of $u_i$ all equal $p^r$ for some $p \in \Irr_0(\F)$ and some $r \geq 1$, and further still
$(b_i,u_i)_{p,r} \simeq (b,u)_{p,r}$; in that case this quadratic invariant is skew-representable.
\item The invariant factors of $u_i$ all equal $p^r$ or $p^{\op}$ for some $p \in \Irr(\F) \setminus \Irr_0(\F)$, and further still
$(b_i,u_i)_{p,r} \simeq (b,u)_{p,r}$ and $(b_i,u_i)_{p^\op,r} \simeq (b,u)_{p^\op,r}$.
\end{itemize}
It suffices to prove that each $(b_i,u_i)$ has the square-zero splitting property.
Hence, in the remainder of the proof, we can reduce the situation to only two cases.

\vskip 3mm
\noindent \textbf{Case 1.} The pair $(b,u)$ has a sole quadratic invariant, attached to $(p,r)$ for some $p \in \Irr_0(\F)$ and some $r \geq 1$,
and this invariant is skew-representable over $\F[t]/(p)$. \\
Then, we know from Proposition \ref{prop:skewrepresentablereciproc} that $(b,u)$ has the same quadratic invariants as some $(1,1)$-boxed sum, the latter having the square-zero splitting property. Hence, $(b,u)$ has the square-zero splitting property.

\vskip 3mm
\noindent \textbf{Case 2.} The pair $(b,u)$ has exactly two quadratic invariants, attached to $(p,r)$ and $(p^\op,r)$ for some $p \in \Irr(\F) \setminus (\Irr_0(\F) \cup \{t\})$ and some $r \geq 1$,
and these invariants satisfy $(b,u)_{p,r} \simeq (-1)^{1+d(r-1)} (b,u)_{p^\op,r}^{\op}$ where $d:=\deg(p)$. \\
Then, we apply Corollary \ref{cor:partcorquadratic} to the bilinear form $(b,u)_{p,r}$.
We can find a triple $(W,B,C)$ in which $W$ is a vector space and $(B,C)$ is a pair of symmetric bilinear forms on $W$, with $B$ non-degenerate,
such that $B \boxplus C$ has only two non-zero quadratic invariants, attached to $(p,r)$ and $(p^\op,r)$, and
$(H_W^1,B \boxplus C)_{p,r} \simeq (b,u)_{p,r}$. Yet, by Proposition \ref{prop:isomopposite}, we know that
$(H_W^1,B \boxplus C) \simeq (-H_W^1, - B \boxplus C)$, and hence Proposition \ref{prop:isomopposite2} yields
$$(H_W^1,B \boxplus C)_{p^\op,r} \simeq (-1)^{1+(r-1) d} ((H_W^1,B \boxplus C)_{p,r})^{\op}
\simeq (-1)^{1+(r-1) d} ((b,u)_{p,r})^{\op} \simeq (b,u)_{p^\op,r}.$$
Hence, $(b,u)$ and $(H_W^1,B \boxplus C)$ have equivalent quadratic invariants, and we conclude that they are isometric.
Finally, $(b,u)$ has the square-zero splitting property.
This completes the proof of Theorem \ref{theo:symsymautomorphism}.

\section{Symmetric forms with nilpotent endomorphisms}\label{section:nilpotent}

This section is devoted to the proof of Theorem \ref{theo:symformnilpotent}. Throughout, we assume that $\charac(\F) \neq 2$.
The most difficult part is to prove the implication (i) $\Rightarrow$ (ii), so we will start with the implication (ii) $\Rightarrow$ (i).
One of the main keys for this implication is the ability to construct pairs $(b,u)$
with the square-zero splitting property in which $u$ has exactly two Jordan cells, one of size $2k+1$ and one of size $2k-1$:
we shall say that such pairs are \emph{twisted} (and it is a consequence of Theorem \ref{theo:symformnilpotent} that, in such a pair,
the bilinear form $b$ must be hyperbolic). Before we study twisted pairs, we will start with a classical result that relates the
Witt type of $b$ to the ones of the $(b,u)_{t,2k+1}$ invariants.

\subsection{On quadratic invariants and the type of $b$}

Let $(b,u)$ be a $(1,\eta)$-pair in which $u$ is nilpotent. Here, we discuss the relationship between the Witt type of $b$
and the quadratic invariants of $(b,u)$.

We start with the case where all the quadratic invariants of $(b,u)$ but one vanish.
So, assume that, for some $k \geq 1$, the endomorphism $u$ is nilpotent with only Jordan cells of size $k$.

\noindent \textbf{Case 1:} $k=2l$ for some $l \geq 1$. Then $\Ker u^i=\im u^{k-i}$ for all $i \in \lcro 0,k\rcro$ (this is obvious on a Jordan cell of size $k$).
In particular $\Ker u^l=\im u^l=(\Ker u^l)^{\bot_b}$, and it follows that $b$ is hyperbolic.

\noindent \textbf{Case 2:} $k=2l+1$ for some $l \geq 0$. This time around, we find $(\Ker u^l)^{\bot_b}=\im u^l=\Ker u^{l+1}$.
Classically, this yields that $b$ is Witt-equivalent to the non-degenerate bilinear form it induces on the quotient space
$\Ker u^{l+1}/\Ker u^l$. Now, the bilinear form $B : (x,y)\in V^2 \mapsto b(x,u^{k-1}(y))$ reads
$(x,y) \in V^2 \mapsto \eta^l b(u^l(x),u^l(y))$
Hence, the form induced by $B$ on $V/\im u=\Ker u^{2l+1}/\Ker u^{2l}$
is equivalent to the form induced by $\eta^l b$ on the quotient space $\Ker u^{l+1}/\Ker u^{l}$ through the isomorphism
$\Ker u^{2l+1}/\Ker u^{2l} \overset{\simeq}{\longrightarrow} \Ker u^{l+1}/\Ker u^{l}$ induced by $u^l$.
Hence, $b$ is Witt-equivalent to $\eta^l (b,u)_{t,k}$.

\vskip 3mm
In the general case of a $(1,\eta)$-pair $(b,u)$ in which $u$ is nilpotent, the classification theorem shows that for every integer $k \geq 1$,
there exists a $(1,\eta)$-pair $(b,u)^{(k)}$ in which $u$ is nilpotent, the quadratic invariant $(b,u)^{(k)}_{t,i}$ vanishes
for every positive integer $i \neq k$, and $(b,u)^{(k)}_{t,k}$ is equivalent to $(b,u)_{t,k}$.
It follows that $(b,u) \simeq \underset{k \geq 1}{\bot} (b,u)^{(k)}$.
Using the previous special cases, we conclude:

\begin{prop}
Let $(b,u)$ be a $(1,\eta)$-pair in which $u$ is nilpotent.
Then $b$ is Witt-equivalent to $\underset{l \geq 0}{\bot} \eta^l (b,u)_{t,2l+1}$.
\end{prop}

\subsection{Twisted pairs have the square-zero splitting property}

\begin{lemma}\label{lemma:twisted}
Let $k \in \N^*$. Let $V$ be a $4k$-dimensional vector space
together with a hyperbolic symmetric bilinear form $b : V^2 \rightarrow \F$.
Let $\eta \in \{-1,1\}$. There exists $u \in \End(V)$ such that $(b,u)$ is a $(1,\eta)$-pair with the square-zero splitting property, and
$u$ is nilpotent with exactly two Jordan cells, one of size $2k+1$ and one of size $2k-1$.
\end{lemma}

\begin{proof}
Let us choose a $b$-hyperbolic basis $\bfB=(e_1,\dots,e_{2k},f_1,\dots,f_{2k})$ of $V$.
The matrix of $b$ in that basis is
$S=\begin{bmatrix}
0 & I_{2k} \\
I_{2k} & 0
\end{bmatrix}$.

Next, we define $a_1 \in \End(V)$ by $a_1(e_{2i})=e_{2i-1}$ and
$a_1(f_{2i-1})=\eta f_{2i}$ for all $i \in \lcro 1,k\rcro$, and $a_1$ maps all the other vectors of $\bfB$ to $0$.
We define $a_2 \in \End(V)$ by $a_2(e_{2i+1})=e_{2i}$ and
$a_2(f_{2i})=\eta f_{2i+1}$ for all $i \in \lcro 1,k-1\rcro$,
$a_2(f_1)=e_{2k}$ and $a_2(f_{2k})=\eta e_1$, and $a_2$ maps all the other vectors of $\bfB$ to $0$.
It is easily checked on the vectors of $\bfB$ that $a_1^2=0$ and $a_2^2=0$.

The respective matrices $M_1$ and $M_2$ of $a_1$ and $a_2$ in $\bfB$ are of the form
$$M_1=\begin{bmatrix}
A_1 & 0 \\
0 & \eta A_1^T
\end{bmatrix} \quad \text{and} \quad
M_2=\begin{bmatrix}
A_2 & C \\
0 & \eta A_2^T
\end{bmatrix}$$
where $C^T=\eta C$.
One checks that both $SM_1$ and $SM_2$ are symmetric if $\eta=1$, and skew-symmetric otherwise.
Hence, both $a_1$ and $a_2$ are $b$-symmetric if $\eta=1$, and $b$-skew-symmetric otherwise.

To conclude, it suffices to check that $u:=a_1+a_2$ is nilpotent and that it has exactly two Jordan cells, one of size $2k+1$
and one of size $2k-1$.
To see this, note first that $u(e_i)=e_{i-1}$ for all $i \in \lcro 2,2k\rcro$, whereas $u(e_1)=0$,
$u(f_1)=\eta f_2+e_{2k}$, $u(f_i)=\eta f_{i+1}$ for all $i \in \lcro 2,2k-1\rcro$, and
$u(f_{2k})=\eta e_1$.

From $u(f_1)=\eta f_2+e_{2k}$, we gather that
$u^i(f_1)=\eta^i f_{i+1}+e_{2k-i+1}$ for all $i \in \lcro 1,2k-1\rcro$,
in particular $u^{2k-1}(f_1)=\eta f_{2k}+e_2$, hence $u^{2k}(f_1)=2e_1$ and finally $u^{2k+1}(f_1)=0$.
Hence, the subspace $W_1:=\Vect(f_1,e_1) \oplus \Vect(\eta^i f_{i+1}+e_{2k-i+1})_{1 \leq i \leq 2k-1}$
is stable under $u$ and the resulting endomorphism is a Jordan cell of size $2k+1$.

Besides, one checks that $u^i(\eta f_{2}-e_{2k})=\eta^{i+1} f_{i+2}-e_{2k-i}$ for all $i \in \lcro 1,2k-2\rcro$,
and in particular $u^{2k-2}(\eta f_{2}-e_{2k})=\eta f_{2k}-e_2$ and then $u^{2k-1}(\eta f_2-e_{2k})=\eta^2 e_1-e_1=0$.
Obviously, $W_2:=\Vect(\eta^i f_{i+1}-e_{2k-i+1})_{1 \leq i \leq 2k-1}$ is a $(2k-1)$-dimensional subspace, it is stable
under $u$ and the resulting endomorphism is a Jordan cell of size $2k-1$.
Finally, one checks that $V=W_1 \oplus W_2$ (note that here the assumption that $\charac(\F) \neq 2$ is critical), and the claimed result follows.
\end{proof}

\begin{cor}\label{cor:twistedcor}
Let $(b,u)$ be a $(1,\eta)$-pair in which $u$ is nilpotent with exactly two Jordan cells, of respective sizes $2k+1$ and
$2k-1$. Assume that $(b,u)_{t,2k+1} \simeq -\eta (b,u)_{t,2k-1}$. Then
$(b,u)$ has the square-zero splitting property.
\end{cor}

\begin{proof}
By the previous lemma, there is a $(1,\eta)$-pair $(b',u')$ in which $u$ is nilpotent with exactly two Jordan cells, of respective sizes $2k+1$ and $2k-1$, and $b'$ is hyperbolic.
Since $b'$ is Witt-equivalent to $\eta^k (b',u')_{t,2k+1} \bot \eta^{k-1}  (b',u')_{t,2k-1}$,
it follows that $\eta^k (b',u')_{t,2k+1} \simeq -\eta^{k-1}  (b',u')_{t,2k-1}$,
that is $(b',u')_{t,2k+1} \simeq -\eta (b',u')_{t,2k-1}$.
Now, choose a nonzero value $\alpha$ (respectively, $\beta$) represented by $(b,u)_{t,2k-1}$ (respectively, by $(b',u')_{t,2k-1}$).
Then, for $b'':=\alpha \beta^{-1}b'$, we see that $(b'',u')$ is a $(1,\eta)$-pair, with sole non-trivial invariants $(b'',u')_{t,2k-1}= \alpha \beta^{-1} (b',u')_{t,2k-1}$ and
$(b'',u')_{t,2k+1}= \alpha \beta^{-1} (b',u')_{t,2k+1}$. The first one represents $\alpha$, and the second one $-\eta \alpha$.
It follows from our assumptions that $(b'',u')_{t,2k-1} \simeq (b,u)_{t,2k-1}$ and $(b'',u')_{t,2k+1} \simeq (b,u)_{t,2k+1}$,
and we deduce from the classification of $(1,\eta)$-pairs that $(b'',u')$ is isometric to $(b,u)$.
Of course, the $b'$-selfadjoint (respectively, $b'$-alternating) endomorphisms are also $b''$-selfadjoint (respectively, $b''$-alternating) and hence $u'$ is the sum of two such square-zero endomorphisms. The conclusion then follows from the previous isometry.
\end{proof}

\begin{cor}\label{cor:twistedpairs}
Let $(b,u)$ be a $(1,\eta)$-pair in which $u$ is nilpotent with only Jordan cells of sizes $2k+1$ and
$2k-1$. Assume that $(b,u)_{t,2k+1} \simeq -\eta (b,u)_{t,2k-1}$. Then
$(b,u)$ has the square-zero splitting property.
\end{cor}

\begin{proof}
Let us take an orthogonal basis $(e_1,\dots,e_n)$ for the form $(b,u)_{t,2k+1}$, and denote by $\alpha_1,\dots,\alpha_n$
the values of the associated quadratic form at the vectors $e_1,\dots,e_n$.
By assumption, we can find an orthogonal basis $(f_1,\dots,f_n)$ for the form $(b,u)_{t,2k-1}$
such that the corresponding values of the associated quadratic form are $-\eta \alpha_1,\dots,-\eta \alpha_n$.
By the classification of $(1,\eta)$-pairs, we can find, for all $i \in \lcro 1,n\rcro$,
such a pair $(b_i,u_i)$ in which $u_i$ is nilpotent with exactly two Jordan cells, one of size $2k+1$ and one of size $2k-1$,
and $(b^{(i)},u^{(i)})_{t,2k+1}$ represents $\alpha_i$ and $(b_i,u_i)_{t,2k-1}$ represents $-\eta \alpha_i$.
Then, $(b,u) \simeq (b^{(1)},u^{(1)}) \bot \cdots \bot (b^{(n)},u^{(n)})$.
By Corollary \ref{cor:twistedcor}, each pair $(b^{(i)},u^{(i)})$ has the square-zero splitting property, and we conclude
that so does $(b,u)$.
\end{proof}

\subsection{Reconstructing pairs having the square-zero splitting property}

We now have all the tools to prove the implication (i) $\Rightarrow$ (ii) in Theorem \ref{theo:symformnilpotent}.

We start with simple cases:

\begin{prop}\label{prop:Jordanevensize}
Let $\eta \in \{-1,1\}$.
Let $(b,u)$ be a $(1,\eta)$-pair in which $u$ is nilpotent with only Jordan cells of even size.
Then $(b,u)$ has the square-zero splitting property.
\end{prop}

\begin{proof}
Assume first that $\eta=-1$. The assumptions yield that all the Jordan numbers of $u$ are even.
We can consider a vector space $W$ equipped with a nilpotent endomorphism $v$ such that $n_{t,k}(v)=\frac{1}{2} n_{t,k}(u)$
for all $k \geq 1$. Then the pair $H_{1,\eta}(v)$ has the same Jordan numbers as $u$, with only Jordan cells of even size, all attached to the irreducible polynomial $t$. By the classification of $(1,-1)$-pairs, we deduce that $H_{1,\eta}(v) \simeq (b,u)$. Besides, since $v$ is nilpotent it is the sum of two square-zero
endomorphisms of $W$, and hence $H_{1,\eta}(v)$ has the square-zero splitting property. We conclude that so does $(b,u)$.

Assume now that $\eta=1$. Then, we can split $(b,u) \simeq (b_1,u_1) \bot \cdots \bot (b_n,u_n)$
so that each $u_i$ is a Jordan cell of even size $k_i$.
Let $i \in \lcro 1,n\rcro$, and choose a non-zero value $\alpha_i \in \F \setminus \{0\}$
that is represented by the quadratic invariant $(b_i,u_i)_{t,k_i}$.
By Lemma \ref{lemma:evensizenilpotent}, there exists a triple $(W,B,C)$ consisting of a vector space $W$ and of a pair
$(B,C)$ of symmetric bilinear forms on $W$, with $B$ non-degenerate, such that $B \boxplus C$
is cyclic with minimal polynomial $t^{k_i}$, and the quadratic invariant $(H_W^1,B \boxplus C)_{t,k_i}$ represents $\alpha_i$;
as this quadratic invariant is the sole non-trivial invariant of $(H_W^1,B \boxplus C)$, it is equivalent to
$(b_i,u_i)_{t,k_i}$ and we conclude that $(b_i,u_i) \simeq (H_W^1,B \boxplus C)$.
Hence, $(b_i,u_i)$ has the square-zero splitting property.
We conclude that so does $(b,u)$.
\end{proof}

In the next step, we tackle the pairs $(b,u)$, where $u$ is nilpotent with only Jordan cells of odd size,
and whose quadratic invariants are all hyperbolic.

\begin{prop}\label{prop:nilpotenthyperbolicodd}
Let $\eta \in \{-1,1\}$.
Let $(b,u)$ be a $(1,\eta)$-pair in which $u$ is nilpotent with only Jordan cells of odd size.
Assume that $(b,u)_{t,2k+1}$ is hyperbolic for all $k \geq 0$. Then $(b,u)$ has the square-zero splitting property.
\end{prop}

\begin{proof}
The assumptions show that all the Jordan numbers of $u$ are even.
Let us choose a nilpotent endomorphism $v$ of a vector space such that $n_{t,k}(v)=\frac{1}{2} n_{t,k}(u)$
for all $k \geq 1$ (and in particular $v$ has only Jordan cells of odd size).
Then the quadratic invariants of $H_{1,\eta}(v)$ attached to the pairs $(t,2k+1)$ are all hyperbolic, and all the other ones vanish.
Moreover $n_{t,k}(h_\eta(v))=2\, n_{t,k}(v)=n_{t,k}(u)$ for all $k \geq 1$.
Since hyperbolic symmetric bilinear forms are characterized by their rank up to equivalence, it follows from the classification
of $(1,\eta)$-pairs that $(b,u) \simeq H_{1,\eta}(v)$. Finally, $H_{1,\eta}(v)$ has the square-zero splitting property
because so does $v$ by Botha's theorem (see Proposition \ref{remark:expansionsquarezero}).
We conclude that $(b,u)$ has the square-zero splitting property.
\end{proof}

Now, we extend our study to the situation where $u$ has Jordan cells of odd size only.

\begin{prop}\label{prop:Jordanoddsize}
Let $\eta \in \{-1,1\}$.
Let $(b,u)$ be a $(1,\eta)$-pair in which $u$ is nilpotent with only Jordan cells of odd size,
and $\eta^k (b,u)_{t,2k+1}$ Witt-simplifies $\underset{i > k}{\bot}\eta^i (b,u)_{t,2i+1}$ for every integer $k \geq 0$.
Then $(b,u)$ has the square-zero splitting property.
\end{prop}

\begin{proof}
We shall say that a $(1,\eta)$-pair $(b',u')$ satisfies condition $(\calC)$ whenever
$\eta^k (b',u')_{t,2k+1}$ Witt-simplifies $\underset{i > k}{\bot} \eta^i\,(b',u')_{t,2i+1}$ for every integer $k \geq 0$.

We prove the result by induction on the dimension of the underlying space. The result is obvious for $0$-dimensional spaces, and more generally
when $u=0$ (in that case it suffices to take the summands of $u$ equal to zero).

So now we assume that $u \neq 0$. The nilindex of $u$ reads $2\ell+1$ for some $\ell \geq 1$.
Write $B_i:=(b,u)_{t,2i+1}$ for all $i \geq 0$  and note that $B_{\ell}$ is non-zero, whereas $B_k=0$ for all $k>\ell$.

\noindent \textbf{Case 1: $B_\ell$ is isotropic.} \\
Then we can split $B_\ell=H \bot B'$ where $H$ is a hyperbolic form on a $2$-dimensional space, and $B'$ is a non-degenerate symmetric bilinear form. Using the classification of $(1,\eta)$-pairs, this helps us split $(b,u)\simeq (b_1,u_1) \bot (b_2,u_2)$
in which:
\begin{itemize}
\item $(b_1,u_1)$ is a $(1,\eta)$-pair with exactly two Jordan cells, both of size $2\ell+1$, and
$(b_1,u_1)_{t,2\ell+1}$ is hyperbolic;
\item $(b_2,u_2)$ is a $(1,\eta)$-pair with only Jordan cells of odd size,
$(b_2,u_2)_{t,2\ell+1} \simeq B'$ and $(b_2,u_2)_{t,2k+1} \simeq (b,u)_{t,2k+1}$ for all $k \neq \ell$.
\end{itemize}

By Proposition \ref{prop:nilpotenthyperbolicodd}, $(b_1,u_1)$ has the square-zero splitting property.

Next, we show that $(b_2,u_2)$ satisfies condition $(\calC)$: indeed, for all $k\neq \ell$,
the forms $\underset{i>k}{\bot} \eta^i (b_2,u_2)_{t,2i+1}$ and $\underset{i>k}{\bot} \eta^i B_i$
have the same non-isotropic parts up to equivalence, whereas $\eta^k (b_2,u_2)_{t,2k+1} \simeq \eta^k B_k$;
moreover $\underset{i>\ell}{\bot} \eta^i (b_2,u_2)_{t,2i+1}$ equals zero, and hence its non-isotropic part is a subform of
$\eta^\ell (b_2,u_2)_{t,2\ell+1}$. Using the fact that $(b,u)$  satisfies condition $(\calC)$, we deduce that so does
$(b_2,u_2)$.

Hence, by induction $(b_2,u_2)$ has the  square-zero splitting property, and we conclude that so does $(b,u)$.

\vskip 3mm
\noindent \textbf{Case 2: $B_\ell$ is nonisotropic.} \\
By condition $(\calC)$, we can split $B_{\ell-1} \simeq (-\eta B_\ell) \bot \varphi$ for some non-degenerate symmetric bilinear form $\varphi$.
This helps us split $(b,u)\simeq (b_1,u_1) \bot (b_2,u_2)$
in which:
\begin{itemize}
\item $(b_1,u_1)$ is a $(1,\eta)$-pair with only Jordan cells of size $2\ell-1$ or $2\ell+1$
$(b_1,u_1)_{t,2\ell+1} \simeq B_\ell$ and $(b_1,u_1)_{t,2\ell-1} \simeq -\eta B_\ell$;
\item $(b_2,u_2)$ is a $(1,\eta)$-pair with only Jordan cells of odd size at most $2\ell-1$,
$(b_2,u_2)_{t,2\ell-1} \simeq \varphi$ and
$(b_2,u_2)_{t,2k+1} \simeq B_k$ for every $k \in \lcro 0,\ell-2\rcro$.
\end{itemize}
By Corollary \ref{cor:twistedpairs}, $(b_1,u_1)$ has the square-zero splitting property.

To check that $(b_2,u_2)$ satisfies condition $(\calC)$, we note first that, for all
$k \leq \ell-2$, we have
$\underset{i >k}{\bot}\eta^i B_i \simeq \bigl((-\eta^\ell B_\ell) \bot (\eta^\ell B_\ell)\bigr)
\bot \underset{i >k}{\bot}\eta^i (b_2,u_2)_{t,2i+1}$, and since the form
$(-\eta^\ell B_\ell) \bot (\eta^\ell B_\ell)$ is hyperbolic this shows that the nonisotropic parts of
$\underset{i >k}{\bot}\eta^i B_i$ are equivalent to those of $\underset{i >k}{\bot}\eta^i (b_2,u_2)_{t,2i+1}$.
This is enough to check condition $(\calC)$ for $(b_2,u_2)$ at each integer $k \leq \ell-2$.
Besides, for $k\geq \ell-1$ we have $\underset{i >k}{\bot}\eta^i (b_2,u_2)_{t,2i+1}=0$ and hence condition
$(\calC)$ is trivially satisfied at $k$ for $(b_2,u_2)$. We conclude that $(b_2,u_2)$ satisfies condition $(\calC)$.
Hence, by induction $(b_2,u_2)$ has the square-zero splitting property, and we conclude that so does $(b,u)$.

Hence, our inductive proof is complete.
\end{proof}

Now, we complete the proof of (i) $\Rightarrow$ (ii) in Theorem \ref{theo:symformnilpotent}.
So, let $(b,u)$ be a $(1,\eta)$-pair in which $u$ is nilpotent and, for all
$k \geq 0$ the bilinear form $\eta^k\,(b,u)_{t,2k+1}$ Witt-simplifies $\underset{i>k}{\bot} \eta^i\, (b,u)_{t,2i+1}$.
Using the classification of indecomposable $(1,\eta)$-pairs, we can split
$(b,u) \simeq (b_1,u_1) \bot (b_2,u_2)$ so that:
\begin{itemize}
\item $(b_1,u_1)$ is a $(1,\eta)$-pair in which $u_1$ is nilpotent with all its Jordan cells of odd size;
\item $(b_2,u_2)$ is a $(1,\eta)$-pair in which $u_2$ is nilpotent with all its Jordan cells of even size.
\end{itemize}
By Proposition \ref{prop:Jordanevensize}, the pair $(b_2,u_2)$ has the square-zero splitting property.
Next, we see that $(b_2,u_2)_{t,2k+1} \simeq (b,u)_{t,2k+1}$ for every integer $k \geq 0$, and it follows that
$(b_2,u_2)$ satisfies the assumptions of Proposition \ref{prop:Jordanoddsize}.
Hence, $(b_2,u_2)$ has the square-zero splitting property, and we conclude that so does $(b,u)$.

Thus, it is now proved that in Theorem \ref{theo:symformnilpotent} condition (i) implies condition (ii).
The proof of the converse implication is spread over the next four sections.

\subsection{Three reduction techniques to lower the nilindex}

Here, we shall examine three techniques to reduce the nilindex in the analysis of a $(1,\eta)$-pair $(b,u)$
in which $u$ is nilpotent.

\subsubsection{Descent}\label{section:descent}

Let $(b,u)$ be a $(1,\eta)$-pair in which $u$ is nilpotent with nilindex $\nu \geq 3$ (and underlying vector space $V$).

Here, we compute the invariants of the induced pair $(b,u)^{\im u}$, which we will locally denote by $(\overline{b},\overline{u})$.
Noting that $(\im u)^{\bot_b}=\Ker u$, we see that the underlying vector space of the pair $(b,u)^{\im u}$ is
$W:=\im u/(\Ker u \cap \im u)$.

Obviously $\overline{u}^{\nu-1}=0$. Let us look at the quadratic invariants of $(b,u)$ (for pairs of type $(t,r)$ with $r \geq 1$).
To this end, we start by examining the case where all the Jordan cells of $u$ have the same size.

So, assume that, for some integer $r \geq 1$, all the Jordan cells of $u$ have size $r$.
If $r=1$ then $\im u=\{0\}$ and $(\overline{b},\overline{u})$ is the zero pair.
If $r=2$ then $\im u=\Ker u$ and again $(\overline{b},\overline{u})$ is the zero pair.

Assume now that $r>2$.
We have $\Ker u^k=\im u^{r-k}$ for all $k \in \lcro 0,r\rcro$, and $\Ker u^k=V$ for all $k \geq r$.
Hence, the source space of $\overline{u}$ is $\Ker u^{r-1}/\Ker u$.
Then, $\Ker \overline{u}^i=\Ker u^{i+1}/\Ker u$ for all $i \in \lcro 0,r-2\rcro$, and
$\Ker \overline{u}^i=\Ker u^{r-1}/\Ker u$ for all $i>r-2$.
It follows that all the Jordan cells of $\overline{u}$ have size $r-2$ (and there are as many as there are Jordan cells of $u$).
Hence, $(\overline{b},\overline{u})_{t,r-2}$ is the sole non-zero quadratic invariant of $(\overline{b},\overline{u})$.
It is defined as the bilinear form induced by $(x,y) \mapsto \overline{b}(x,\overline{u}^{r-3}(y))$ on the quotient space of $\Ker u^{r-1}/\Ker u$ with $\Ker u^{r-2}/\Ker u$, and it is therefore naturally equivalent to the bilinear form
induced by $(x,y) \mapsto b(x,u^{r-3}(y))$ on the quotient space $\Ker u^{r-1}/\Ker u^{r-2}$.

The linear map $u$ induces an isomorphism from $V/\Ker u^{r-1}$ to $\Ker u^{r-1}/\Ker u^{r-2}$.
Noting that $b(u(x),u^{r-3}(u(y)))=\eta\, b(x,u^{r-1}(y))$ for all $x,y$ in $V$, we find that this isomorphism defines an isometry from
$V/\Ker u^{r-1}$ equipped with $(\overline{x},\overline{y})\mapsto b(x,u^{r-1}(y))$ to
$\Ker u^{r-1}/\Ker u^{r-2}$ equipped with $(\overline{x},\overline{y})\mapsto \eta\,\overline{b}(x,u^{r-3}(y))$.
Therefore, $(\overline{b},\overline{u})_{t,r-2} \simeq \eta\, (b,u)_{t,r}$.

Let us come back to the general case where $u$ is nilpotent.
The pair $(b,u)$ is known to split into $(b,u) \simeq \underset{i \geq 1}{\bot} (b_i,u_i)$
where each $(b_i,u_i)$ is a $(1,\eta)$-pair, $(b_i,u_i)_{t,k}$ vanishes for all $k \neq i$, and $(b_i,u_i)_{t,i} \simeq (b,u)_{t,i}$.
Clearly, the previous reduction is compatible with isomorphisms and orthogonal direct sums. Piecing together the information we have gathered on
the above special case, we obtain the following conclusion.

\begin{prop}\label{proposition:descentenilpotente1}
Let $(b,u)$ be a $(1,\eta)$-pair in which $u$ is nilpotent with nilindex $\nu \geq 1$.
Then:
\begin{enumerate}[(i)]
\item For all $r \geq 1$, one has $((b,u)^{\im u})_{t,r} \simeq \eta\, (b,u)_{t,r+2}$.
\item The nilindex of the endomorphism component of $(b,u)^{\im u}$ equals $\min(1,\nu-2)$.
\end{enumerate}
\end{prop}

\subsubsection{Twisted descent}\label{section:twisteddescent}

Here, we use a slightly different process.
Again, we let $(b,u)$ be a $(1,\eta)$-pair in which $u$ is nilpotent with nilindex $\nu \geq 3$ (and underlying vector space $V$).
This time around, we consider the induced pair $(\overline{b},\overline{u}):=(b,u)^{\Ker u+\im u}$.

Noting that $(\Ker u+\im u)^{\bot_b}=\im u \cap \Ker u$ is included in $\Ker u+\im u$,
we see that the underlying vector space of the induced pair is the quotient space
$(\Ker u +\im u)/(\Ker u \cap \im u)$.
From there, the results are very similar to the one of the previous section, with the notable exception
of the case where $u$ has only Jordan cells of size $1$ (i.e.\ $u=0$).
In that case indeed, we have $(\overline{b},\overline{u})=(b,u)$, leading to
$(\overline{b},\overline{u})_{t,1}=(b,u)_{t,1}$ (instead of $(\overline{b},\overline{u})_{t,1}=0$).

Using the same method as in the previous paragraph, we obtain:

\begin{prop}\label{proposition:descentenilpotente2}
Let $(b,u)$ be a $(1,\eta)$-pair in which $u$ is nilpotent.
Then:
\begin{enumerate}[(i)]
\item For all $r \geq 2$, $((b,u)^{\Ker u+\im u})_{t,r} \simeq \eta\, (b,u)_{t,r+2}$.
\item $((b,u)^{\Ker u+\im u})_{t,1} \simeq (b,u)_{t,1}\, \bot \,\eta\, (b,u)_{t,3}$.
\end{enumerate}
\end{prop}

Note that if $u^3=0$ then the endomorphism component of $(b,u)^{\Ker u+\im u}$ vanishes,
and $(b,u)^{\Ker u+\im u}\simeq ((b,u)^{\Ker u+\im u})_{t,1} \simeq (b,u)_{t,1} \bot \eta (b,u)_{t,3}$.

\subsubsection{Folding}\label{section:symmetrization}

Here, we use yet another type of induced pairs.
Again, we let $(b,u)$ be a $(1,\eta)$-pair in which $u$ is nilpotent.
Let $k \geq 1$.

Here, we consider the induced pair $(b,u)^{\Ker u^k}$, which is defined on the quotient space
$W:=\Ker u^k/(\im u^k\cap \Ker u^k)$ because $(\Ker u^k)^{\bot_b}=\im u^k$.

Again, we compute the quadratic invariants of $(b,u)^{\Ker u^k}$ as functions of those
of $(b,u)$ by considering the special case where $u$ has only Jordan cells of size $r$ for some $r \geq 1$.
If $k\geq r$ then $\im u^k \cap \Ker u^k=\{0\}$ and $\Ker u^k=V$, and so
$(b,u)^{\Ker u^k}=(b,u)$.

Assume now that $k<r$. Then $\im u^k=\Ker u^{r-k}$ and hence
$\im u^k \cap \Ker u^k=\Ker u^{r-k} \cap \Ker u^k=\Ker u^{\min(r-k,k)}$.
\begin{itemize}
\item If $k \leq r-k$ then $W=\{0\}$ and $(b,u)^{\Ker u^k}$ vanishes.
\item Assume that $k>r-k$. Then $W=\Ker u^k/\Ker u^{r-k}$, and $\overline{u}$ has only Jordan cells of size $2k-r$. The $((b,u)^{\Ker u^k})_{t,2k-r}$ invariant is the bilinear form induced by
$(\overline{x},\overline{y}) \mapsto b(x,u^{2k-r-1}(y))$ on $\Ker u^k/\Ker u^{k-1}$.
Noting that $b(u^{r-k}(x),u^{2k-r-1}(u^{r-k}(y)))=\eta^{r-k} b(x,u^{r-1}(y))$ for all $x,y$ in $V$, we obtain that
the linear map $x \mapsto u^{r-k}(x)$ induces an isometry
from $(b,u)_{t,r}$ to $\eta^{r-k}((b,u)^{\Ker u^k})_{t,2k-r}$.
Hence, $((b,u)^{\Ker u^k})_{t,2k-r} \simeq \eta^{r-k} (b,u)_{t,r}$, and it is the sole possible non-vanishing quadratic invariant of $(b,u)^{\Ker u^k}$.
\end{itemize}

As in Section \ref{section:descent}, we piece the previous results together to obtain the general form of the quadratic invariants of $(\overline{b},\overline{u})$:

\begin{prop}\label{prop:symmetrization}
Let $(b,u)$ be a $(1,\eta)$-pair in which $u$ is nilpotent. Let $k \geq 1$.
Then:
\begin{enumerate}[(i)]
\item For all $r \in \lcro 1,k-1\rcro$, $\bigl((b,u)^{\Ker u^k}\bigr)_{t,r} \simeq
(b,u)_{t,r} \bot \eta^{k+r}(b,u)_{t,2k-r}$.
\item $\bigl((b,u)^{\Ker u^k}\bigr)_{t,k} \simeq (b,u)_{t,k}$.
\item For all $r >k$, $\bigl((b,u)^{\Ker u^k}\bigr)_{t,r}=0$.
\item The endomorphism component $\overline{u}$ of $(b,u)^{\Ker u^k}$ satisfies $\overline{u}^k=0$.
\end{enumerate}
\end{prop}

\subsection{The key lemma: statement and consequences}

We are ready for the main key in our study of nilpotent endomorphisms.
Let us first recall some notation and terminology. The rank of a non-degenerate symmetric bilinear form is simply the
dimension of the underlying vector space. The Witt index of such a form $B$, denoted by $\nu(B)$,
is the greatest dimension for a totally $B$-isotropic subspace.
Remember that, given two such forms $B$ and $B'$, we say that $B$ Witt-simplifies $B'$
whenever every nonisotropic part of $B'$ is equivalent to a subform of $-B$.

Here is our key result.

\begin{lemma}[Key lemma]\label{lemma:keylemmanilpotent}
Let $\eta \in \{1,-1\}$.
Let $(b,u)$ be a $(1,\eta)$-pair with the square-zero splitting property, in which $u^3=0$.
Then $(b,u)_{t,1}$ Witt-simplifies $\eta\,(b,u)_{t,3}$.
\end{lemma}

The proof of this lemma is the most technical part of the present manuscript, so we wait until the next section to give it.
Here, we shall show how this lemma, combined with the processes described in the previous paragraphs,
helps one recover the necessary condition featured in Theorem \ref{theo:symformnilpotent}.

\begin{prop}\label{prop:CNnilpotentecomplete}
Let $(b,u)$ be a $(1,\eta)$-pair with the square-zero splitting property.
Then $\eta^k (b,u)_{t,2k+1}$ Witt-simplifies $\underset{i >k}{\bot} \eta^i (b,u)_{t,2i+1}$ for all $k \geq 0$.
\end{prop}

\begin{proof}[Proof of Proposition \ref{prop:CNnilpotentecomplete}, assuming the validity of Lemma
\ref{lemma:keylemmanilpotent}]
We shall prove the result by induction on the nilindex of $u$, by steps of two.
If $u^3=0$, the result follows directly from Lemma \ref{lemma:keylemmanilpotent}.

Next, assume that $u^{2k+1}=0$ for some integer $k \geq 2$.
Let us choose square-zero endomorphisms $a_1,a_2$ such that $u=a_1+a_2$ and
$(b,a_1)$ and $(b,a_2)$ are $(1,\eta)$-pairs.

First of all, we use the descent technique applied to $(b,u)$.
This yields the $(1,\eta)$-pair $(\overline{b},\overline{u}):=(b,u)^{\im u}$ with underlying space
$\im u/(\Ker u\cap \im u)$. Since $\im u$ is stable under $a_1$ and $a_2$ (see Lemma \ref{lemma:stabilization}),
we obtain (see Section \ref{section:pairs}) that $(\overline{b},\overline{u})$ has the square-zero splitting property.
Then, $(\overline{b},\overline{u})_{t,2\ell+1} \simeq \eta (b,u)_{t,2\ell+3}$ is trivial for every integer $\ell \geq k$.
By induction, we deduce that $\eta^\ell (\overline{b},\overline{u})_{t,2\ell+1}$ Witt-simplifies $\underset{i >\ell}{\bot} \eta^i (\overline{b},\overline{u})_{t,2i+1}$ for all $\ell \geq 0$. Using Proposition \ref{proposition:descentenilpotente1}, we deduce that $\eta^{\ell+1} (b,u)_{t,2\ell+1}$ Witt-simplifies $\underset{i >\ell}{\bot} \eta^{i+1} (b,u)_{t,2i+1}$ for all $\ell \geq 1$, and multiplying by $\eta$ shows that $\eta^\ell (b,u)_{t,2\ell+1}$ Witt-simplifies $\underset{i >\ell}{\bot} \eta^{i} (b,u)_{t,2i+1}$ for all $\ell \geq 1$.

It only remains to obtain the conclusion for $\ell=0$. To do this, we go back to
$(b,u)$ and apply the folding technique twice, first at the integer $2k$, and then at $2k-1$.
So, we successively take $(b',u'):=(b,u)^{\Ker u^{2k}}$ and $(b'',u''):=(b',u')^{\Ker (u')^{2k-1}}$.
By Lemma \ref{lemma:stabilization}, the subspace $\Ker u^{2k}$ is stable under $a_1$ and $a_2$, and hence
$(b',u')$ has the square-zero splitting property. Likewise, we apply Lemma \ref{lemma:stabilization}
to find that $(b'',u'')$ has the square-zero splitting property.
Since $(u'')^{2k-1}=0$, we find by induction that
$(b'',u'')_{t,1}$ Witt-simplifies $\underset{i >0}{\bot} \eta^{i} (b'',u'')_{t,2i+1}$.
Yet, since $(u')^{2k+1}=0$, Proposition \ref{prop:symmetrization} shows that
$(b'',u'')_{t,2i+1} \simeq (b',u')_{t,2i+1}$ for every integer $i \geq 0$, and hence
$(b',u')_{t,1}$ Witt-simplifies $\underset{i >0}{\bot} \eta^{i} (b',u')_{t,2i+1}$.

Finally, remembering that $(b,u)_{t,2l+1}=0$ for all $l > k$, we deduce from Proposition \ref{prop:symmetrization} that
$(b',u')_{t,2k-1} \simeq \eta (b,u)_{t,2k+1} \bot (b,u)_{t,2k-1}$, whereas
$(b',u')_{t,2i+1} \simeq (b,u)_{t,2i+1}$ for all $i \in \lcro 0,k-2\rcro$.
Hence $\underset{i >0}{\bot} \eta^{i} (b,u)_{t,2i+1} \simeq \underset{i >0}{\bot} \eta^{i} (b',u')_{t,2i+1}$,
whereas $(b,u)_{t,1} \simeq (b',u')_{t,1}$, and we conclude that
$(b,u)_{t,1}$ Witt-simplifies $\underset{i >0}{\bot} \eta^{i} (b,u)_{t,2i+1}$.

This completes our inductive proof.
\end{proof}

The proof of (ii) $\Rightarrow$ (i) in Theorem \ref{theo:symformnilpotent} is now entirely reduced to the proof of Lemma \ref{lemma:keylemmanilpotent}, which is given in the next section.

\subsection{Proof of the key lemma}

Here, we finally prove Lemma \ref{lemma:keylemmanilpotent}.
Throughout, we set $B_1=(b,u)_{t,1}$ and $B_3:=(b,u)_{t,3}$.

First of all, as seen in Section \ref{section:results}, the conclusion of the lemma means that
$\nu(\eta B_3)+\nu(B_1 \bot \eta B_3) \geq \rk(\eta B_3)$, and it is precisely this inequality we shall prove.
Note that it is obviously true if $B_3$ is hyperbolic.

\vskip 3mm
\noindent \textbf{Step 1: Reduction to the indecomposable case.}

Here, we will see that it suffices to consider the case where $(b,u)$ is \emph{indecomposable} as
a $(1,\eta)$-pair with the square-zero splitting property, meaning that
there do not exist non-trivial $(1,\eta)$-pairs $(b_1,u_1)$ and $(b_2,u_2)$
such that $(b,u) \simeq (b_1,u_1) \bot (b_2,u_2)$ and each $(b_i,u_i)$ has the square-zero splitting property.

So, assume that such a decomposition exists and that the conclusion of Lemma \ref{lemma:keylemmanilpotent} is valid for both $(b_1,u_1)$ and $(b_2,u_2)$.
Note that $(b,u)_{t,1} \simeq (b_1,u_1)_{t,1} \bot (b_2,u_2)_{t,1}$ and
$(b,u)_{t,3} \simeq (b_1,u_1)_{t,3} \bot (b_2,u_2)_{t,3}$.
Thus, we have
\begin{align*}
\nu(\eta (b,u)_{t,3})+\nu((b,u)_{t,1}\bot \eta (b,u)_{t,3})
& \geq \nu\bigl(\eta (b_1,u_1)_{t,3}\bigr)+\nu\bigl(\eta (b_2,u_2)_{t,3}\bigr)\\
& \hskip 10mm +\nu\bigl((b_1,u_1)_{t,1}\bot \eta (b_1,u_1)_{t,3}\bigr) \\
& \hskip 10mm +\nu\bigl((b_2,u_2)_{t,1}\bot \eta (b_2,u_2)_{t,3}\bigr)
\end{align*}
and hence
$$\nu\bigl(\eta (b,u)_{t,3}\bigr)+\nu\bigl((b,u)_{t,1}\bot \eta (b,u)_{t,3}\bigr)
\geq \rk\bigl(\eta\,(b_1,u_1)_{t,3}\bigr)+\rk\bigl(\eta\,(b_2,u_2)_{t,3}\bigr)=\rk\bigl(\eta\,(b,u)_{t,3}\bigr).$$
Therefore, by induction on the dimension of the underlying vector space, we see that it suffices to prove the conclusion
under the following additional conditions:
\begin{itemize}
\item[(I)] $(b,u)$ is indecomposable as a $(1,\eta)$-pair with the square-zero splitting property;
\item[(II)] The invariant $B_3$ is non-hyperbolic.
\end{itemize}
So, in the rest of the proof, we assume that conditions (I) and (II) are valid.

Before we can proceed, we will need basic results on $(1,\eta)$-pairs $(b',u')$ such that $(u')^2=0$.

\vskip 3mm
\noindent \textbf{Step 2: General considerations on $(\varepsilon',\eta')$-pairs $(b',u')$ in which $(u')^2=0$.}
Let $(\varepsilon',\eta')\in \{-1,1\}^2$.
\begin{itemize}
\item If $(\varepsilon',\eta')=(1,1)$ (respectively, if $(\varepsilon',\eta')=(-1,-1)$), for every indecomposable $(1,\eta')$-pair $(b',u')$ in which $(u')^2=0$ and $u'\neq 0$,
there exists a $b'$-hyperbolic (respectively, $b'$-symplectic) basis $(x,y)$ such that $u'(y)=\lambda x$ for some $\lambda \in \F \setminus \{0\}$ (to the effect that $u'(x)=0$).
\item If $(\varepsilon',\eta')=(1,-1)$ (respectively, if $(\varepsilon',\eta')=(-1,1)$), for every indecomposable $(1,\eta')$-pair $(b',u')$ in which $(u')^2=0$ and $u'\neq 0$,
there exists a $b'$-hyperbolic (respectively, $b'$-symplectic) basis $(x_1,x_2,y_1,y_2)$
(so that $b'(x_i,y_i)=1=b'(y_i,x_i)$ for all $i \in \{1,2\}$, and all other pairs in that basis are mapped to zero by $b'$)
such that $u'(x_2)=x_1$ and $u'(y_1)=\eta' y_2$. An example of such a pair is $H_{\varepsilon',\eta'}(v)$
where $v$ is a nilpotent Jordan cell of size $2$.
\end{itemize}
By using the decomposition of a pair into indecomposable ones, we recover that for every $(\varepsilon',\eta')$-pair $(b',u')$ in which
$(u')^2=0$ and $u' \neq 0$:
\begin{itemize}
\item If $(\varepsilon',\eta')=(1,1)$ (respectively, $(\varepsilon',\eta')=(-1,-1)$)
then there exists a hyperbolic (respectively, symplectic) family $(x,y)$ for $b'$
and a nonzero scalar $\lambda$ such that $u'(y)=\lambda x$.
\item If $(\varepsilon',\eta')=(1,-1)$ (respectively, $(\varepsilon',\eta')=(-1,1)$)
then there exists a hyperbolic (respectively, symplectic) family $(x_1,x_2,y_1,y_2)$ for $b'$
such that $u'(x_2)=x_1$ and $u'(y_1)=\eta' y_2$, leading to $u'(x_1)=0$ and $u'(y_2)=0$.
\end{itemize}

\vskip 3mm
We are now ready to analyze the situation.
So, we let $a$ be an endomorphism of $V$ that is $b$-selfadjoint if $\eta=1$ and $b$-alternating otherwise, and
we assume that both $a$ and $u-a$ have square zero. Expanding $(u-a)^2=0$, we find
$$au+ua=u^2.$$

\vskip 3mm
\noindent \textbf{Step 3: $a$ maps $V$ into $\Ker u^2$.}

Remember that $a$ stabilizes $\Ker u^2$ (Lemma \ref{lemma:stabilization})
and hence it induces an endomorphism $\overline{a}$ of the quotient space $V/\Ker u^2$.
It turns out $(B_3,\overline{a})$ is a $(1,\eta)$-pair: indeed,
as $a$ commutes with $u^2$ (Lemma \ref{lemma:commutation}), we have, for all $(x,y)\in V^2$,
$$b\bigl(x,u^2(a(y))\bigr)=b\bigl(x,a(u^2(y))\bigr)=\eta\, b(a(x),u^2(y)).$$
Now, we assume that $\overline{a} \neq 0$ and we seek to find a contradiction by applying the results of Step 2.

First of all, let $x \in V$ be an arbitrary vector. We shall see that
$$W_x:=\Vect(x,a(x),u(x),u(a(x)),u^2(x),u^2(a(x)))$$
is stable under both $a$ and $u$.
The stability under $u$ is obvious because $u^3=0$.
For the stability under $a$, simply note that $a(u(x))=-u(a(x))+u^2(x)$
and likewise $a(u(a(x))=-u(a^2(x))+u^2(a(x))=u^2(a(x))$, and finally use the fact that $a$ commutes with $u^2$
to obtain $a(u^2(x))=u^2(a(x))$ and $a(u^2(a(x)))=a^2(u^2(x))=0$.

Next, we will use the $W_x$ spaces to contradict the indecomposability of $(b,u)$, by taking advantage of the fact that
$(B_3,\overline{a})$ is a $(1,\eta)$-pair.

\textbf{Case 1: $\eta=1$.}

In that case, we recover a $B_3$-hyperbolic pair $(\overline{x},\overline{y})$ in $V/\Ker u^2$ together with a nonzero scalar $\lambda$ such that
$\overline{a}(\overline{y})=\lambda \overline{x}$.
We lift $\overline{y}$ to a vector $y$ of $V$, and then we set $x:=\lambda^{-1} a(y)$, so that
$a(y)=\lambda x$ and $a(x)=0$.
Then, we see that $W_y=\Vect(x,y,u(x),u(y),u^2(x),u^2(y))$.
Since $u^3=0$, we have $\im u^2 \subset \Ker u$ and hence $\im u^2 \bot_b \im u$.
Hence, the matrix of $b$ in the \emph{family} $(x,y,u(x),u(y),u^2(x),u^2(y))$ (at this point we still do not know whether that family is a basis of $W_y$)
reads
$$M=\begin{bmatrix}
? & ? & A \\
? & B & [0]_{2 \times 2} \\
A^T & [0]_{2 \times 2} & [0]_{2 \times 2}
\end{bmatrix}$$
where $A$ and $B$ are $2$-by-$2$ matrices. Moreover, since $u$ is $b$-selfadjoint, we find $B=A$.
Finally,
$$A=\begin{bmatrix}
b(x,u^2(x)) & b(x,u^2(y)) \\
b(y,u^2(x)) & b(y,u^2(y))
\end{bmatrix}=\begin{bmatrix}
B_3(\overline{x},\overline{x}) & B_3(\overline{x},\overline{y}) \\
B_3(\overline{y},\overline{x}) & B_3(\overline{y},\overline{y})
\end{bmatrix}=\begin{bmatrix}
0 & 1 \\
1 & 0
\end{bmatrix}.$$
It follows that $M$ is invertible, yielding that $\bigl(x,y,u(x),u(y),u^2(x),u^2(y)\bigr)$
is a basis of $W_y$ and that $W_y$ is $b$-regular.
Hence $V=W_y \overset{\bot_b}{\oplus} W_y^{\bot_b}$, and $W_y^{\bot_b}$
is stable under both $u$ and $a$ because they are $b$-selfadjoint. It follows that the induced pairs
$(b,u)^{W_y}$ and $(b,u)^{W_y^{\bot_b}}$ have the square-zero splitting property.
By assumption (I), we deduce that $W_y=V$ (because $W_y \neq \{0\}$).
Now, in that situation it is clear that $B_3$ is hyperbolic,
thereby contradicting assumption (II).

\vskip 3mm
\textbf{Case 2: $\eta=-1$.}

In that case, we recover a $B_3$-hyperbolic family $(\overline{x_1},\overline{x_2},\overline{y_1},\overline{y_2})$ in $V/\Ker u^2$
such that $\overline{a}(\overline{x_2})=\overline{x_1}$ and $\overline{a}(\overline{y_1})=-\overline{y_2}$.
We choose arbitrary representatives $x_2$ and $y_1$ of $\overline{x_2}$ and $\overline{y_1}$ in $V$, and then we set $x_1:=a(x_2)$ and $y_2:=-a(y_1)$.
Now, we consider the space $U:=W_{x_2}+W_{y_1}$, which is stable under $u$ and $a$.
Again, we will prove that $U$ is $b$-regular.
To see this, we consider the $12$-tuple
\begin{multline*}
(z_i)_{1 \leq i \leq 12}:=
\bigl(x_1,x_2,y_1,y_2,u(x_1),u(x_2),u(y_1),u(y_2),u^2(x_1),u^2(x_2),u^2(y_1),u^2(y_2)\bigr).
\end{multline*}
Just like in Case 1, we find that the matrix of $b$ in that family has the form
$$M=\begin{bmatrix}
? & ? & A \\
? & B & [0]_{4 \times 4} \\
A^T & [0]_{4 \times 4} & [0]_{4 \times 4}
\end{bmatrix}$$
where each block is a $4$-by-$4$ matrix, and $B=-A$ because $u$ is $b$-skew-selfadjoint.
Moreover, $A$ is precisely the matrix of $B_3$ in the hyperbolic family $(\overline{x_1},\overline{x_2},\overline{y_1},\overline{y_2})$, and
hence it is invertible. From this, we derive that $M$ is invertible.
Just like in Case 1, we deduce from assumption (I) that $U=V$, and it follows that $B_3$ is hyperbolic, in contradiction with assumption (II).

\vskip 2mm
We conclude that $\overline{a}=0$, which means that $a$ maps $V$ into $\Ker u^2$.
It follows that $u^2a=0$ and hence $au^2=0$. Thus, $a$ vanishes everywhere on $\im u^2$.

\vskip 3mm
\noindent \textbf{Step 4: $a$ maps $\Ker u^2$ into $\im u+\Ker u$.}

Remember that $a$ stabilizes $\Ker u^2$, $\im u$ and $\Ker u$ by Lemma \ref{lemma:stabilization}.
It follows that $a$ induces an endomorphism
$\overline{a}$ of the quotient space $V':=\Ker u^2/(\im u+\Ker u)$, the underlying space of the bilinear form $B_2:=(b,u)_{t,2}$.
Using $au+ua=u^2$, we obtain, for all $x,y$ in $\Ker u^2$,
$$b(x,u(a(y)))=-b(x,a(u(y)))=-\eta\,b(a(x),u(y)),$$
and hence $(B_2,\overline{a})$ turns out to be an $(\eta,-\eta)$-pair.
Note that $(\overline{a})^2=0$.

From there, we use a line of reasoning that is very similar to the one we have used in Step 3.
For $x \in \Ker u^2$, we set
$$U_x:=\Vect(x,a(x),u(x),u(a(x)))$$
and, like in Step 3, we check that $U_x$ is stable under both $a$ and $u$ (here, we use the fact that $u^2(x)=0$
in addition to the identities $a^2=0$ and $ua+au=u^2$).

Assume now that $\overline{a} \neq 0$.

Then, we adapt Case 2 in Step 3 to obtain a family $(x_1,x_2,y_1,y_2)$ of vectors of $\Ker u^2$
whose family of classes modulo $\Ker u+\im u$ is a $B_2$-symplectic family if $\eta=-1$, a $B_2$-hyperbolic one if $\eta=1$,
and for which $a(x_2)=x_1$ and $a(y_1)=-\eta y_2$.
Then the family $(x_1,x_2,y_1,y_2,u(x_1),u(x_2),u(y_1),u(y_2))$ generates $U_{y_1}+U_{x_2}$
and the matrix of $b$ in that family equals
$$M=\begin{bmatrix}
? & K' \\
(K')^T & [0]_{4 \times 4}
\end{bmatrix}$$ for $K':=\begin{bmatrix}
[0]_{2 \times 2} & I_2 \\
-\eta I_2 & [0]_{2 \times 2}
\end{bmatrix}$. Hence $M$ is invertible.
Just like in Step 3, this leads to $U_{y_1}+U_{x_2}=V$ and further to $u^2=0$. In particular $B_3=0$, contradicting assumption (II).

We conclude that $\overline{a}=0$, which means that
$a$ maps $\Ker u^2$ into $\Ker u+\im u$. Note that a similar line of reasoning would also show that $a$ must map $\Ker u$ into $\Ker u \cap \im u$, but
we will not use this result in the remainder of the proof.

\vskip 3mm
\noindent \textbf{Step 5: Introducing an important linear map}

Now, we look at the endomorphism $\overline{a}$ of $V/(\Ker u+\im u)$ induced by $a$.
We know that its range is included in $(\Ker u^2)/(\Ker u+\im u)$ and that
it vanishes everywhere on $(\Ker u^2)/(\Ker u+\im u)$.
We choose direct summands
$$\Ker \overline{a} \oplus D_1=V/(\Ker u+\im u) \quad \text{and} \quad
\Ker \overline{a}=(\Ker u^2)/(\Ker u+\im u) \oplus D_{2.}$$
Then, we lift $D_1$ and $D_2$ to subspaces $V_1$ and $V_2$ of $V$ (so that the canonical projection from $V$ onto $V/(\Ker u+\im u)$ maps bijectively $V_i$ onto $D_i$ for all $i \in \{1,2\}$). Hence
$$\Ker u^2 \oplus V_1 \oplus V_2=V, \quad a(V_2) \subset \Ker u+\im u,$$
$$(\Ker u+\im u)\oplus a(V_1) \subset \Ker u^2 \quad \text{and} \quad \dim a(V_1)=\dim V_1.$$
Next, we consider once more the non-degenerate bilinear form
$$B_2 : (\overline{x},\overline{y}) \mapsto b(x,u(y))$$ on
$(\Ker u^2)/(\Ker u+\im u)$ (which is symmetric if $\eta=1$, and alternating otherwise).
We consider the projection $\overline{a(V_1)}$ of $a(V_1)$ on $(\Ker u^2)/(\Ker u+\im u)$,
and then a complementary subspace $D_3$ of the $B_2$-orthogonal $\overline{a(V_1)}^{\bot_{B_2}}$,
and finally we lift $D_3$ to a subspace $V_3$ of $\Ker u^2$ such that
$\dim V_3=\dim D_3=\dim(a(V_1))=\dim V_1$. Hence, the bilinear form $(x,y) \in V_3 \times a(V_1) \mapsto b(x,u(y))$
is non-degenerate on both sides.

Next, we consider the quotient space $Z:=(\Ker u+\im u)/(\Ker u \cap \im u)$.
We have seen in Section \ref{section:twisteddescent} that $b$ induces a non-degenerate symmetric bilinear form $\overline{b}$ on $Z$
that is equivalent to $B_1 \bot \eta B_3$.
We have $a(V_2) \subset \Ker u+\im u$, whereas $a(V_3) \subset \Ker u+\im u$ (by Step 4).
Hence, we can consider the induced linear mapping
$$f : \begin{cases}
V_3\oplus V_2 & \longrightarrow Z \\
x & \longmapsto \overline{a(x)}.
\end{cases}$$
Since $a^2=0$, we have $\im a \subset \Ker a=(\im a)^{\bot_b}$, and hence
the range of $f$ is totally $\overline{b}$-isotropic.
It follows that
$$\nu(B_1 \bot \eta B_3) \geq \rk(f).$$
In order to conclude, it suffices to prove that $\rk(f) \geq \rk(B_3)-\nu(B_3)$, and it is exactly what we will do
in the next and final step.

\vskip 3mm
\noindent \textbf{Step 6: Proving that $\rk(f) \geq  \rk(B_3)-\nu(B_3)$.}

We start by checking that $\Ker f \cap V_3=\{0\}$. Let $x \in \Ker f \cap V_3$.
Then $a(x) \in \Ker u \cap \im u$ and in particular $u(a(x))=0$.
It follows that $a(u(x))=-u(a(x))+u^2(x)=0$.
As the range and kernel of $a$ are $b$-orthogonal, this yields
$b(z,u(x))=0$ for all $z \in a(V_1)$. From the choice of $V_3$ we deduce that $x=0$.

Now, denote by $\pi$ the projection from $V_3 \oplus V_2$ onto $V_2$ along $V_3$.
Let $x \in \Ker f$.
Note that $a(x) \in \Ker u$ and hence $b(a(x),u(x))=\eta\,b(u(a(x)),x)=0$.
Then,
$$b(x,u^2(x))=b(x,(ua+au)(x))=\eta\, b(a(x),u(x))=0.$$
Since $V_3 \subset \Ker u^2$, it follows that the projection $\overline{\pi(\Ker f)}$ in $V/(\Ker u^2)$
of the subspace $\pi(\Ker f)$ is totally $B_3$-isotropic, yielding
$$\dim \overline{\pi(\Ker f)} \leq \nu(B_3).$$
Yet $\dim \overline{\pi(\Ker f)}=\dim \pi(\Ker f)$ because $V_2 \cap \Ker u^2=\{0\}$, whereas
$\dim \pi(\Ker f)=\dim \Ker f$ because $\Ker f \cap V_3=\{0\}$.
We conclude by the rank theorem that
\begin{align*}
\rk f & =\dim V_2+\dim V_3-\dim \Ker f \\
& \geq \dim V_2+\dim V_3-\nu(B_3) \\
& \geq \dim V_2+\dim V_1-\nu(B_3)=\dim(V/\Ker u^2)-\nu(B_3)=\rk (B_3)-\nu(B_3).
\end{align*}
Hence, the proof of the key lemma is finally complete, as well as the proof of Theorem \ref{theo:symformnilpotent}.

\section{The Hermitian case}

\subsection{The problem in the Hermitian case}

In this final section, we consider the problem in the Hermitian case.
Throughout, the field $\F$ is equipped with a non-identity involution $x \mapsto x^\bullet$,
so that $\F$ is a separable extension of degree $2$ of the subfield $\K:=\{x \in \F : \; x^\bullet=x\}$.
The given involution is extended to an involution $p \mapsto p^\bullet$ of $\F[t]$ by applying it
coefficient-wise. Given a vector space $V$, we denote by $V^{\star 1/2}$ the vector space of all semi-linear forms on $V$.

Remember our convention that a Hermitian form on a vector space over $\F$ is right-linear and left-semi-linear.
Given such a Hermitian form $b$, on a vector space $V$, we obtain a linear mapping
$$L_b : x \in V \mapsto b(-,x) \in V^{\star 1/2},$$
and this mapping is an isomorphism if and only if $b$ is non-degenerate.

A $1/2$-pair is a pair $(b,u)$ consisting of a non-degenerate Hermitian form $b$ on an $\F$-vector space $V$
and of a $b$-selfadjoint endomorphism $u$ of $V$, so that
$$\forall (x,y)\in V^2, \; b(u(x),y)=b(x,u(y)).$$
In that case, $c : (x,y)\mapsto b(x,u(y))$ is another Hermitian form on $V$.

A Hermitian pair is a pair $(b,c)$ of Hermitian forms on a given vector space (over $\F$, of finite-dimension).
In such a pair, if $b$ is non-degenerate then $c$ reads $(x,y) \mapsto b(x,u(y))$ for a unique $b$-selfadjoint endomorphism
$u$. Hence, $1/2$-pairs naturally correspond to Hermitian pairs whose first component is non-degenerate,
and $1/2$-pairs whose second components are automorphisms naturally correspond to Hermitian pairs in which both components are non-degenerate.

\begin{Def}
Let $(b,u)$ be a $1/2$-pair with underlying space $V$. We say that it has the \textbf{square-zero splitting property} when there are square-zero
$b$-selfadjoint endomorphisms $a_1$ and $a_2$ such that $u=a_1+a_2$.
\end{Def}

For $1/2$-pairs, we have the following notions, defined in the same way as for $(\varepsilon,\eta)$-pairs:
\begin{itemize}
\item isometric pairs;
\item the induced pair $(b,u)^W$ for a linear subspace $W$ of the underlying space;
\item orthogonal sums of pairs.
\end{itemize}

Again, we have the following principles:
\begin{enumerate}[(i)]
\item If two $1/2$-pairs are isometric and one has the square-zero splitting property, then so does the other one.
\item If two $1/2$-pairs have the square-zero splitting property, then so does their orthogonal sum.
\item Let $(b,u)$ be a $1/2$-pair, with underlying space $V$. Let $a_1$ and $a_2$ be square-zero $b$-selfadjoint endomorphisms of $V$
such that $u=a_1+a_2$. Let finally $W$ be a linear subspace of $V$ that is stable under both $a_1$ and $a_2$.
Then the induced pair $(b,u)^W$ has the square-zero splitting property.
\end{enumerate}

Our aim is to classify, up to isometry, the $1/2$-pairs with the square-zero splitting property.
To do so, we need to recall first the relevant invariants of such a pair $(b,u)$ under isometry, beyond the Jordan numbers
of $u$ (which, in general, do not suffice to characterize the pair $(b,u)$ up to isometry).

So, let $p \in \Irr(\F)$ be such that $p^\bullet=p$, i.e.\ $p$ is irreducible over $\F$ and has all its coefficients in $\K$.
Let $r \geq 1$ be an integer. Since $p^{r-1}$ has its coefficients in $\K$,
$(x,y) \mapsto b(x,p(u)^{r-1}(y))$ is a Hermitian form on $V^2$.
It induces a non-degenerate Hermitian form $B$ on the cokernel $W_{p,r}$ of the injective linear map
$$\Ker p(u)^{r+1}/\Ker p(u)^r \longrightarrow \Ker p(u)^{r}/\Ker p(u)^{r-1}$$
induced by $p(u)$.
This cokernel has a naturally induced structure of vector space over $\L:=\F[t]/(p)$,
which is equipped with the involution induced by the one of $\F$ (this works because $p^\bullet=p$),
and for this structure one uses the fact that $u$ is $b$-selfadjoint to find that
$B(q^\bullet\,x,y)=B(x,q\,y)$ for all $q \in \F[t]$ and all $(x,y)\in (W_{p,r})^2$.

Now, we use the extension process of Section \ref{section:extensionbilinform}.
Let $(x,y)\in (W_{p,r})^2$.
The mapping
$$\lambda \in \L \mapsto B(x,\lambda\,y)\in \F$$
is $\F$-linear, and hence there is a unique scalar $B^\L(x,y) \in \L$ such that
$$\forall \lambda \in \L, \; B(x,\lambda\,y)=e_p\bigl(\lambda\, B^\L(x,y)\bigr).$$
Note that
$$\forall \lambda \in \L, \; B(y,\lambda x)=B(\lambda x,y)^\bullet=B(x,\lambda^\bullet y)^\bullet
=e_p\bigl(\lambda^\bullet B^\L(x,y)\bigr)^\bullet=e_p\bigl(\lambda B^\L(x,y)^\bullet\bigr)$$
and hence $B^\L(y,x)=B^\L(x,y)^\bullet$.
As $B$ is right-$\F$-linear, it is clear that $B^\L$ is right-$\L$-linear, and we conclude that $B^\L$
is Hermitian. Finally, since $b$ is non-degenerate it is clear that so is $B^\L$.
We set $(b,u)_{p,r}:=B^\L$ and call it the \textbf{Hermitian invariant} of $(b,u)$ attached to the pair $(p,r)$.

In contrast with the case of $(1,1)$-pairs, note that such invariants only appear with selfadjoint monic irreducible polynomials, not with all monic irreducible polynomials!

Now, we can state the classification theorem for $1/2$-pairs. It can be retrieved with some effort from Theorem 4 of \cite{Sergeichuk}
(where there is no need to assume that $\chi(\F) \neq 2$).

\begin{theo}
Let $(b,u)$ and $(b',u')$ be two $1/2$-pairs over $\F$.
These pairs are isometric if and only if both the following conditions hold:
\begin{enumerate}[(i)]
\item $n_{p,r}(u)=n_{p,r}(u')$ for every $p \in \Irr(\F)$ such that $p \neq p^\bullet$, and every integer $r \geq 1$.
\item The Hermitian forms $(b,u)_{p,r}$ and $(b',u')_{p,r}$ are equivalent for every $p \in \Irr(\F)$ such that $p=p^\bullet$, and every
integer $r \geq 1$.
\end{enumerate}
\end{theo}

Moreover, there is little limitation on the possible Jordan numbers of $u$ and the possible Hermitian invariants of $(b,u)$:
one limitation is that $n_{p,r}(u)=n_{p^\bullet,r}(u)$ for all $p \in \Irr(\F)$ and all $r \geq 1$.
In other words, if we take any family $(m_{p,r})_{p \in \Irr(\F) \setminus \K[t], r \geq 1}$ of natural numbers,
and any family $(H_{p,r})_{p \in \Irr(\F) \cap \K[t], r \geq 1}$ of non-degenerate Hermitian forms (where $H_{p,r}$ is a Hermitian form over $\F[t]/(p)$),
and both those families have only finitely many non-zero terms and $m_{p,r}=m_{p^\bullet,r}$ for all $p \in \Irr(\F) \setminus \K[t]$ and all $r \geq 1$,
then there exists a $1/2$-pair $(b,u)$ such that $n_{p,r}(u)=m_{p,r}$ for all $p \in \Irr(\F) \setminus \K[t]$ and all $r \geq 1$,
and $(b,u)_{p,r}$ is equivalent to $H_{p,r}$ for all $p \in \Irr(\F) \cap \K[t]$ and all $r \geq 1$.

Consequently, indecomposable $1/2$-pairs are the ones of the following two types:
\begin{itemize}
\item The pairs $(b,u)$ in which $u$ is cyclic with minimal polynomial $p^r$ for some $p \in \Irr(\F) \cap \K[t]$ and some $r \geq 1$;
\item The pairs $(b,u)$ in which $u$ is cyclic with minimal polynomial $p^r (p^\bullet)^r$ for some $p \in \Irr(\F) \setminus \K[t]$ and some $r \geq 1$.
\end{itemize}

\subsection{The solution}

As we shall see, our theorem for $1/2$-pairs is very reminiscient to the theorem on $(1,1)$-pairs.
First of all, we use the Fitting decomposition to split the problem into two subproblems.
Like in Section \ref{section:pairs}, one proves that if a $1/2$-pair $(b,u)$ has the square zero splitting property, then so
do $(b,u)^{\Co(u)}$ and $(b,u)^{\Nil(u)}$, while in general $(b,u) \simeq (b,u)^{\Co(u)} \bot (b,u)^{\Nil(u)}$.
This leads to the following principle:

\begin{prop}
Let $(b,u)$ be a $1/2$-pair. For $(b,u)$ to have the square-zero splitting property, it is necessary and sufficient
that both $(b,u)^{\Co(u)}$ and $(b,u)^{\Nil(u)}$ have the square-zero splitting property.
\end{prop}

For automorphisms, we need an additional notion:

\begin{Def}
Let $p \in \Irr_0(\F) \cap \K[t]$. A Hermitian form $\varphi$ over $\L:=\F[t]/(p)$
is called \emph{odd-representable} whenever there is a $\varphi$-orthogonal basis $(e_1,\dots,e_n)$ such that, for all $i \in \lcro 1,n\rcro$,  $\varphi(e_i,e_i)$ is the class of an odd polynomial with entries in $\K$ modulo $p$ (i.e.\ there exists $s_i \in \K[t]$ such that $\varphi(e_i,e_i)=\overline{ts_i(t^2)}^{\L}$).
\end{Def}

\begin{theo}\label{theo:hermitianautomorphismcharnot2}
Assume that $\charac(\F) \neq 2$.
Let $(b,u)$ be a $1/2$-pair where $u$ is an automorphism. For $(b,u)$ to have the square-zero splitting property, it is necessary and sufficient that
all the following conditions hold:
\begin{enumerate}[(i)]
\item For all $p \in \Irr_0(\F) \cap \K[t]$ and all $r \geq 1$, the Hermitian invariant $(b,u)_{p,r}$ is odd-representable.
\item For all $p \in \Irr(\F) \cap \K[t]$ such that $p^\op \neq p$, and for all $r \geq 1$, one has
$(b,u)_{p,r} \simeq (-1)^{1+d(r-1)} (b,u)_{p^\op,r}^{\op}$ where $d:=\deg p$.
\item For all $p \in \Irr(\F)\setminus \K[t]$ and all $r \geq 1$, the Jordan numbers $n_{p,r}(u)$ and $n_{p^{\op},r}(u)$ are equal.
\end{enumerate}
\end{theo}

\begin{theo}\label{theo:hermitianautomorphismchar2}
Assume that $\charac(\F)=2$.
Let $(b,u)$ be a $1/2$-pair where $u$ is an automorphism. For $(b,u)$ to have the square-zero splitting property, it is necessary and sufficient that
all the following conditions hold:
\begin{enumerate}[(i)]
\item For all $p \in \Irr(\F) \setminus \Irr_0(\F)$ and all odd $r \geq 1$, one has $n_{p,r}(u)=0$.
\item For all $p \in \Irr_0(\F) \cap \K[t]$ and all $r \geq 1$, the Hermitian invariant $(b,u)_{p,r}$ is odd-representable.
\end{enumerate}
\end{theo}

Now, we turn to nilpotent endomorphisms. Remember that Witt theory applies in a similar way to Hermitian forms as to symmetric bilinear forms.
In particular, we have the Witt index $\nu(H)$ of a non-degenerate Hermitian form $H$.

\begin{Def}
Let $H$ and $H'$ be two non-degenerate Hermitian forms (on finite-dimensional vector spaces over $\F$).
The following conditions are equivalent:
\begin{itemize}
\item Each nonisotropic part of $H'$ is equivalent to a subform of $-H$;
\item One has $\nu(H')+\nu(H' \bot H) \geq \rk(H')$.
\end{itemize}
When they are satisfied, we say that $H$ \textbf{Witt-simplifies} $H'$.
\end{Def}

\begin{theo}\label{theo:hermitiannilpotent}
Let $(b,u)$ be a $1/2$-pair in which $u$ is nilpotent.
The following conditions are equivalent:
\begin{enumerate}[(i)]
\item $u$ is the sum of two square-zero $b$-Hermitian endomorphisms.
\item For every integer $k \geq 0$, the Hermitian form
$(b,u)_{t,2k+1}$ Witt-simplifies $\underset{i >k}{\bot} (b,u)_{t,2i+1}$.
\end{enumerate}
\end{theo}

Note that here the result holds whatever the characteristic of $\F$.

Combining the above two theorems finally leads to the full classification:

\begin{theo}\label{theo:hermitiangeneralcharnot2}
Assume that $\charac(\F) \neq 2$.
Let $(b,u)$ be a $1/2$-pair.
For $(b,u)$ to have the square-zero splitting property, it is necessary and sufficient that all the following conditions hold:
\begin{enumerate}[(i)]
\item For all $p \in \Irr_0(\F) \cap \K[t]$ and all $r \geq 1$, the Hermitian invariant $(b,u)_{p,r}$ is odd-representable.
\item For all $p \in \Irr(\F) \cap \K[t]$ such that $p^\op \neq p$, and for all $r \geq 1$, one has
$(b,u)_{p,r} \simeq (-1)^{1+d(r-1)} (b,u)_{p^\op,r}^{\op}$ where $d:=\deg p$.
\item For all $p \in \Irr(\F) \setminus \K[t]$ and all $r \geq 1$, the Jordan numbers $n_{p,r}(u)$ and $n_{p^{\op},r}(u)$ are equal.
\item For every integer $k \geq 0$, the Hermitian form
$(b,u)_{t,2k+1}$ Witt-simplifies $\underset{i >k}{\bot} (b,u)_{t,2i+1}$.
\end{enumerate}
\end{theo}

\begin{theo}\label{theo:hermitiangeneralchar2}
Assume that $\charac(\F) = 2$.
Let $(b,u)$ be a $1/2$-pair.
For $(b,u)$ to have the square-zero splitting property, it is necessary and sufficient that all the following conditions hold:
\begin{enumerate}[(i)]
\item For all $p \in \Irr(\F) \setminus \Irr_0(\F)$ and all odd $r \geq 1$, one has $n_{p,r}(u)=0$.
\item For all $p \in \Irr_0(\F) \cap \K[t]$ and all $r \geq 1$, the Hermitian invariant $(b,u)_{p,r}$ is odd-representable.
\item For every integer $k \geq 0$, the Hermitian form
$(b,u)_{t,2k+1}$ Witt-simplifies $\underset{i >k}{\bot} (b,u)_{t,2i+1}$.
\end{enumerate}
\end{theo}

\subsection{Examples}

\subsubsection{Complex numbers}

We consider the complex field $\C$ with the standard involution $z \mapsto \overline{z}$.
The only irreducible polynomials over $\C$ are the ones of degree 1, so in Theorem \ref{theo:hermitiangeneralcharnot2}
condition (i) is void.

Condition (iii) can be reformulated as saying that for all $\lambda \in \C \setminus \R$ and all $r \geq 1$, the Jordan numbers
$n_{t-\lambda,r}(u)$ and $n_{t+\lambda,r}(u)$ are equal.

Now, remember that non-degenerate Hermitian forms $H$ over $\C$ (on finite-dimensional vector spaces)
are classified by their inertia $(s_+(H),s_-(H))$.

Hence, condition (ii) from Theorem \ref{theo:hermitiangeneralcharnot2} can be reformulated as follows: for all $\lambda \in \R \setminus \{0\}$ and all $r \geq 1$,
either $r$ is even and $s_+((b,u)_{t-\lambda,r})=s_+((b,u)_{t+\lambda,r})$ and $s_-((b,u)_{t-\lambda,r})=s_-((b,u)_{t+\lambda,r})$,
or $r$ is odd and $s_+((b,u)_{t-\lambda,r})=s_-((b,u)_{t+\lambda,r})$ and $s_-((b,u)_{t-\lambda,r})=s_+((b,u)_{t+\lambda,r})$.

Finally, condition (iv) from Theorem \ref{theo:hermitiangeneralcharnot2} can be rephrased in terms of inertia of Hermitian invariants, just like
condition (iii) from Theorem \ref{theo:symsym} was rephrased over the field of real numbers in Section \ref{section:examplesquadratic} (with $\eta=1$).

\subsubsection{Finite fields}

Over finite fields with a non-identity involution, non-degenerate Hermitian forms are classified by their rank.
Hence, in Theorem \ref{theo:hermitiangeneralcharnot2}, conditions (i) to (iii) can be simplified as
saying that $u$ is similar to $-u$, and condition (iv) is simplified as follows:
for all $k \geq 0$,  $n_{t,2k+1}(u)>0$ if $\sum_{i=k+1}^{+\infty} n_{t,2i+1}(u)$ is odd.

Likewise, in Theorem \ref{theo:hermitiangeneralchar2}, conditions (i) to (ii) can be simplified
as saying that all the invariant factors of $u$ are even or odd, and condition (iii) is simplified just like for fields of characteristic different from $2$.

\subsection{Hermitian extensions}

Let $V$ be a finite-dimensional vector space over $\F$.
We equip the space $W:=V \times V^{\star 1/2}$ with the hyperbolic Hermitian form
$$H_V^{1/2} : \bigl((x,\varphi),(y,\psi)\bigr) \longmapsto \varphi(y)^\bullet+\psi(x).$$
Let $u \in \End(V)$. We consider its half-transpose, denoted by $u^{t/2}$ and defined as follows:
$$\varphi \in V^{\star 1/2} \longmapsto \varphi \circ u \in V^{\star 1/2}.$$
We set
$$h_{1/2}(u):=u \oplus u^{t/2}.$$
One checks that $h_{1/2}(u)$ is $H_V^{1/2}$-selfadjoint, and we denote by
$H_{1/2}(u)$ the $1/2$-pair $(H_V^{1/2},h_{1/2}(u))$.

Now, assume that $u=u_1+u_2$ for square-zero endomorphisms $u_1$ and $u_2$ of $V$.
Then, one checks that $h_{1/2}(u)=h_{1/2}(u_1)+h_{1/2}(u_2)$ and that $\bigl(h_{1/2}(u_1)\bigr)^2=0$ and
$\bigl(h_{1/2}(u_2)\bigr)^2=0$. It follows that $H_{1/2}(u)$ has the square-zero splitting property.

Besides, one checks that, for endomorphisms $u_1$ and $u_2$ of respective vector spaces $V_1$ and $V_2$, we have
an isometry
$$H_{1/2}(u_1 \oplus u_2) \simeq H_{1/2}(u_1) \bot H_{1/2}(u_2).$$

We can now state an important result on the Hermitian invariants of $H_{1/2}(u)$.

\begin{prop}
Let $u$ be an endomorphism of a vector space.
Then all the Hermitian invariants of $H_{1/2}(u)$ are hyperbolic.
\end{prop}

The proof is essentially similar to the one of Proposition \ref{prop:hyperbolicexpansion}, and the adaptation is effortless.

\subsection{Hyperbolic boxed sums}

Now, we turn to the adaptation of boxed sums.
Let $(b,c)$ be a pair of sesquilinear forms on a vector space $V$, with $b$ non-degenerate.
We obtain \emph{linear} mappings
$$L_c : x \in V \mapsto c(-,x) \in V^{\star 1/2} \quad \text{and} \quad L_b : x \in V \mapsto b(-,x) \in V^{\star 1/2},$$
the latter of which is an isomorphism. This yields two endomorphisms
$$w_b : (x,\varphi) \in V \times V^{\star 1/2} \longmapsto (L_b^{-1}(\varphi),0) \in  V \times V^{\star 1/2}$$
and
$$v_c : (x,\varphi) \in V \times V^{\star 1/2} \longmapsto (0,L_c(x)) \in  V \times V^{\star 1/2}.$$
The Hermitian boxed sum of $b$ and $c$ is defined as the endomorphism
$$b \boxplus c:=w_b+v_c : (x,\varphi) \mapsto \bigl(L_b^{-1}(\varphi),L_c(x)\bigr),$$
of $V \times V^{\star 1/2}$, and one checks that $b \boxplus c$ is $H_V^{1/2}$-Hermitian if and only if both $b$ and $c$ are Hermitian.

One checks that, given two pairs $(b,c)$ and $(b',c')$ of Hermitian forms on respective vector spaces $V$ and $V'$, with the left
component non-degenerate, one has
$$(b,c) \simeq (b',c') \Rightarrow \bigl(H_V^{1/2},b \boxplus c\bigr) \simeq \bigl(H_{V'}^{1/2},b' \boxplus c'\bigr)$$
and
$$\bigl(H_{V \oplus V'}^{1/2},(b \bot b') \boxplus (c \bot c')\bigr) \simeq \bigl(H_V^{1/2},b \boxplus c\bigr) \bot \bigl(H_{V'}^{1/2},b' \boxplus c'\bigr).$$
Moreover, just like in Section \ref{section:boxedsum}, we obtain that Hermitian boxed sums entirely
account for $1/2$-pairs $(b,u)$ with the square-zero splitting property in which $u$ is an automorphism:

\begin{prop}\label{prop:caracboxedsumhermitian}
Let $(b,u)$ be a $1/2$-pair in which $u$ is an automorphism. Then the following conditions are equivalent:
\begin{enumerate}[(i)]
\item $(b,u)$ has the square-zero splitting property;
\item There exists a vector space $W$ and a pair $(B,C)$ of non-degenerate Hermitian forms on $W$
such that $(b,u) \simeq \bigl(H_W^{1/2}, B \boxplus C\bigr)$.
\end{enumerate}
\end{prop}

Now, we compute the invariants of boxed sums. Beware that in the non-degenerate case we need to discuss whether the characteristic of $\F$
is $2$ or not. We start with the case that most resembles the one of $(1,1)$-pairs:

\begin{lemma}
Let $(b,c)$ be a pair of Hermitian forms on a vector space $V$, with $b$ non-degenerate.
Assume that $u:=L_b^{-1}L_c$ is cyclic with minimal polynomial $p_0^r$ for some $p_0 \in \Irr(\F) \cap \K[t]$ and some $r \geq 1$, and assume furthermore
that $p:=p_0(t^2)$ is irreducible. Assume finally that the Hermitian invariant $(b,u)_{p,r}$ represents, for some polynomial $s(t) \in \K[t]$ (not divisible by $p$), the class of $s(t)$ modulo $p_0$.
Then $(H_V^{1/2},b \boxplus c)$ has exactly one Hermitian invariant, namely $(H_V^{1/2},b \boxplus c)_{p,r}$, which is defined on a
$1$-dimensional vector space, and it represents the class of $t s(t^2)$ in the field $\F[t]/(p)$.
\end{lemma}

The proof is a word for word adaptation of the one of Lemma \ref{lemma:skewrepresentabledim1} for $\varepsilon=1$ and we give no details to it.
And just in like in Section \ref{section:invariantsboxedsum} we deduce the following corollaries:

\begin{cor}\label{cor:oddrepresentabledirect}
Let $(b,c)$ be a pair of non-degenerate Hermitian forms on a vector space $V$.
Let $p \in \Irr_0(\F) \cap \K[t]$ and $r \geq 1$. Then $(H_V^{1/2},b \boxplus c)_{p,r}$
is odd-representable over the field $\L:=\F[t]/(p)$.
\end{cor}

\begin{cor}\label{cor:oddrepresentable}
Let $p \in \Irr_0(\F) \cap \K[t]$ and $r>0$. Let $H$ be an odd-representable non-degenerate Hermitian form over the field $\L:=\F[t]/(p)$.
Then there exists a vector space $V$ (over $\F$) and a pair $(b,c)$ of non-degenerate Hermitian forms on $V$
such that:
\begin{enumerate}[(i)]
\item The invariant factors of $b \boxplus c$ all equal $p^r$;
\item $\bigl(H_V^{1/2},b \boxplus c\bigr)_{p,r}$ is equivalent to $H$.
\end{enumerate}
\end{cor}

Next, we have the results on the invariants attached to the irreducible polynomials $p \in \Irr(\F) \cap \K[t]$ such that $p^\op \neq p$.

\begin{lemma}\label{lemma:hermitiansymmetric}
Assume that $\charac(\F) \neq 2$.
Let $p \in \Irr(\F) \cap \K[t]$ be such that $p^\op \neq p$, and let $r>0$.
Let $\beta$ be a non-zero element of $\L:=\F[t]/(p)$ such that $\beta^\bullet=\beta$.
Then there exists a vector space $V$ (over $\F$) equipped with a pair $(b,c)$ of non-degenerate Hermitian forms such that:
\begin{enumerate}[(i)]
\item $b \boxplus c$ is cyclic with minimal polynomial $p^r (p^\op)^r$;
\item The Hermitian invariant $\bigl(H_V^{1/2}, b \boxplus c\bigr)_{p,r}$ has rank $1$ and represents the value $\beta$.
\end{enumerate}
\end{lemma}

The proof is essentially a word-for-word adaptation of the one of Lemma \ref{lemma:representabledim1}, with a
slight twist in the end: in the situation considered here, we need to check that the isomorphism
$\Lambda$ maps selfadjoint elements of $\M$ onto self-adjoint elements of $\L$.
This can be done by noting first that both sets are $\K$-vector spaces, both with finite dimension $\deg p$,
and that $\Lambda$ maps $\{x \in \M : \; x^\bullet=x\}$ into $\{x \in \L : \; x^\bullet=x\}$, which
can be obtained either by direct computation or by noting that it can be recovered from the first part of the proof
(this style of reasoning will be featured again in the proof of an ulterior lemma, so we will say no more about it until then).

Using the usual orthogonal sums technique, Lemma \ref{lemma:hermitiansymmetric} leads to:

\begin{cor}\label{cor:hermitianrepresentable}
Assume that $\charac(\F) \neq 2$.
Let $p \in \Irr(\F) \cap \K[t]$ be such that $p^\op \neq p$, and let $r>0$.
Let $H$ be a non-degenerate Hermitian form over the field $\L:=\F[t]/(p)$.
Then there exists a vector space $V$ (over $\F$) equipped with a pair $(b,c)$ of non-degenerate Hermitian forms such that:
\begin{enumerate}[(i)]
\item The invariant factors of $b \boxplus c$ all equal $(pp^\op)^r$;
\item The Hermitian invariant $\bigl(H_V^{1/2}, b \boxplus c\bigr)_{p,r}$ is equivalent to $H$.
\end{enumerate}
\end{cor}

The next result is proved exactly like its counterpart in the quadratic setting:

\begin{lemma}
Let $r$ be an even positive integer, and let $\alpha \in \K \setminus \{0\}$.
Then there exists, on some vector space $V$, a pair $(b,c)$ of Hermitian forms such that $b$ non-degenerate and:
\begin{enumerate}[(i)]
\item $b \boxplus c$ is nilpotent with a single Jordan cell, of size $r$;
\item The Hermitian invariant $\bigl(H_V^{1/2},b \boxplus c\bigr)_{t,r}$ represents the value $\alpha$.
\end{enumerate}
\end{lemma}

Once more, by using orthogonal sums of pairs of Hermitian forms and an orthogonal decomposition
of a Hermitian form, this leads to:

\begin{cor}
Let $r$ be an even positive integer, and let $H$ be a non-degenerate Hermitian form over $\F$.
Then there exists, on some vector space $V$, a pair $(b,c)$ of Hermitian forms, with $b$ non-degenerate, such that:
\begin{enumerate}[(i)]
\item $b \boxplus c$ is nilpotent and all its Jordan cells have size $r$;
\item The Hermitian invariant $\bigl(H_V^{1/2},b \boxplus c\bigr)_{t,r}$ is equivalent to $H$.
\end{enumerate}
\end{cor}

It remains to consider a special situation in the characteristic $2$ case:

\begin{lemma}
Assume that $\charac(\F) = 2$.
Let $p \in \Irr(\F) \cap \K[t]$ be non-even and distinct from $t$.
Let $r \geq 1$.
Let $\alpha$ be a nonzero element of $\L:=\F[t]/(p)$ such that $\alpha^\bullet=\alpha$.
Then there exists a vector space $V$ and a pair $(b,c)$ of non-degenerate Hermitian forms on $V$
such that:
\begin{itemize}
\item $b \boxplus c$ is cyclic with minimal polynomial $p^{2r}$;
\item The Hermitian invariant $\bigl(H_V^{1/2},b \boxplus c\bigr)_{p,2r}$ has rank $1$ and represents $\alpha$.
\end{itemize}
\end{lemma}

\begin{proof}
We write $p^2=p_0(t^2)$ for some $p_0\in \Irr(\F) \cap \K[t]$. We also set $\mathbb{M}:=\F[t]/(p_0)$.

We can start from a non-zero element $\beta \in \M$ such that $\beta^\bullet=\beta$,
and find a pair $(b,c)$ of non-degenerate Hermitian forms on a vector space $V$ such that $u:=L_b^{-1} L_c$
is cyclic with minimal polynomial $p_0$ and the Hermitian invariant $(b,u)_{p_0,r}$ represents the value $\beta$.
Then, we know from Lemma \ref{lemma:invariantsboxedsum} that $b \boxplus c$
is cyclic with minimal polynomial $(p_0)^r(t^2)=p^{2r}$, which yields the first point and the first part of the second one.

Again, we see $V$ as a $\F[t]$-module for the structure associated with the endomorphism $u$, and
$W:=V \times V^{\star 1/2}$ as a $\F[t]$-module for the structure associated with the endomorphism $v:=b \boxplus c$.

We can choose an element $x \in V$ such that
$b\bigl(x,p_0^{r-1}\,(q\,x)\bigr)=e_{p_0}(\overline{q}^{\M} \beta)$ for all $q \in \F[t]$.
We set $X:=(x,0)$ and we denote by $\overline{X}$ its class in $V/\im p_0(v^2)$ (the space on which
the invariant $(H_V^{1/2},v)_{p,2r}$ is defined).
So let $q \in \F[t]$. We split $pq=q_1(t^2)+t q_2(t^2)$ with $q_1,q_2$ in $\F[t]$.
Then,
\begin{align*}
H_V^{1/2}\bigl(X,p^{2r-1}\,(q\,X)\bigr) & =H_V^{1/2}\bigl((x,0),(p_0^{r-1}(t^2)\,q_1(t^2)+t p_0^{r-1}(t^2)\,q_2(t^2)).(x,0)\bigr) \\
& = L_c((p_0^{r-1} q_2)\,x)[x] \\
& = c\bigl(x,(p_0^{r-1} q_2)\,x\bigr) \\
& = b\bigl(x,(p_0^{r-1} tq_2)\,x\bigr) \\
& = e_{p_0}\bigl(\overline{tq_2}^{\mathbb{M}} \beta\bigr)
\end{align*}

Now, we consider the $\F$-linear mapping that takes the polynomial
$q \in \F[t]$ to $\overline{tq_2}^{\mathbb{M}} \in \F[t]/(p_0)$ (where $q_2$ is defined above with respect to $q$).

If $q$ is a multiple of $p$, then we check that $tq_2$ is a multiple of $p_0$: indeed we can then write
$q=p s$ for some $s \in \F[t]$ which we split $s=s_1(t^2)+t s_2(t^2)$; then
$pq=p^2 s=s_1(t^2)p_0(t^2)+t s_2(t^2) p_0(t^2)$ and hence $q_2(t^2)= s_2(t^2) p_0(t^2)$, which leads to $q_2=s_2 p_0$.

Conversely, assume that $tq_2$ is a multiple of $p_0$. Then $q_2$ is a multiple of $p_0$ and hence $q_2(t^2)$ is a multiple of $p^2$,
and finally $q_1(t^2)$ is a multiple of $p$. Then $q_1$ must be a multiple of $p_0$
(otherwise it is coprime with $p_0$, and then applying the B\'ezout identity and specializing at $t^2$ leads to $q_1(t^2)$ being coprime
with $p_0(t^2)=p^2$ and hence with $p$); therefore $q_1(t^2)$ is actually a multiple of $p^2$. We conclude that $pq$ is a multiple of $p^2$ and hence $p$ divides $q$.

Hence, the above $\F$-linear mapping yields an injective homomorphism of $\F$-vector spaces
$$\varphi : \F[t]/(p) \hookrightarrow \F[t]/(p_0),$$
which turns out to be an isomorphism because $\F[t]/(p)$ and $\F[t]/(p_0)$ have the same dimension as vector spaces over $\F$.
Finally, we consider the isomorphism of $\F$-vector spaces
$$\Phi_p : \begin{cases}
\L & \longrightarrow \Hom_\F(\L,\F) \\
a & \longmapsto [x \mapsto e_p(ax)]
\end{cases} \quad \text{and} \quad
\Psi_{p_0} : \begin{cases}
\mathbb{M} & \longrightarrow \Hom_\F(\mathbb{M},\F) \\
b & \longmapsto [x \mapsto e_{p_0}(bx)].
\end{cases}$$
The composite map $\Lambda:=\Phi_p^{-1} \circ \varphi^t \circ \Psi_{p_0}$ is an isomorphism of $\F$-vector spaces. In the above computation, it precisely maps
$\beta$ to $(H_V^{1/2},b \boxplus c)_{p,2r}(\overline{X},\overline{X})$.
In order to conclude, it suffices to prove that every Hermitian element of $\mathbb{L}$
is the image under $\Lambda$ of some Hermitian element of $\mathbb{M}$.
Yet, by varying $\beta$ in the above construction, we see that $\Lambda$ maps any Hermitian element to a Hermitian element. Hence $\Lambda$ induces an injective $\K$-linear map from $\{q \in \M : q^\bullet=q\}$ to
 $\{s \in \mathbb{L} : s^\bullet=s\}$, both of which are $\K$-vector spaces of dimension $\deg p$;
 hence $\Lambda$ induces a surjection from the former to the latter. Finally we can take $\beta:=\Lambda^{-1}(\alpha)$, which completes the proof.
\end{proof}

Once more, this yields the following corollary:

\begin{cor}
Assume that $\charac(\F)=2$.
Let $p \in \Irr(\F) \cap \K[t]$ be non-even and distinct from $t$.
Let $r \geq 1$, and let $H$ be a non-degenerate Hermitian form over $\F[t]/(p)$.
Then there exists a vector space $V$ together with a pair $(b,c)$ of non-degenerate Hermitian forms on $V$
such that:
\begin{itemize}
\item All the invariant factors of $b \boxplus c$ equal $p^{2r}$;
\item The Hermitian invariant $\bigl(H_V^{1/2},b \boxplus c\bigr)_{p,2r}$ is equivalent to $H$.
\end{itemize}
\end{cor}

\subsection{Completing the case of automorphisms}

\subsubsection{An additional result}\label{section:hermitiansmallresult}

We are almost ready to complete the study of $1/2$-pairs $(b,u)$ in which $u$ is an automorphism.
We only need two additional considerations before we can conclude.

First of all, given a pair $(b,c)$ of non-degenerate Hermitian forms on a vector space $V$,
the $1/2$-pair $\bigl(H_V^{1/2}, b \boxplus c\bigr)$ is isometric to its opposite $\bigl(-H_V^{1/2}, -b \boxplus c\bigr)$:
one simply checks that $\id_V \oplus (-\id_{V^{\star 1/2}})$ is such an isometry.

Besides, just like in the case of $(1,1)$-pairs, we obtain the following result on Hermitian invariants:

\begin{prop}\label{prop:isomopposite3}
Let $(b,u)$ be a $1/2$-pair.
Let $p \in \Irr(\F) \cap \K[t]$ be non-even and different from $t$, and let $r \geq 1$.
Denote by $d$ the degree of $p$.
Then $(-b,-u)_{p,r}=(-1)^{1+d(r-1)} \bigl((b,u)_{p^\op,r}\bigr)^{\op}$.
\end{prop}

The proof is, again, a word for word adaptation of the one of Proposition \ref{prop:isomopposite2}.

Now, we are ready to prove Theorems \ref{theo:hermitianautomorphismcharnot2} and \ref{theo:hermitianautomorphismchar2}.

\subsubsection{Fields of characteristic different from $2$}

Here, we assume that $\charac(\F) \neq 2$.
Let $(b,u)$ be a $1/2$-pair in which $u$ is an automorphism.
Assume that this pair has the square-zero splitting property.
First of all, by Botha's theorem we directly recover condition (iii) in Theorem \ref{theo:hermitianautomorphismcharnot2}.

Next, $(b,u) \simeq (H_V^{1/2},B \boxplus C)$ for some vector space $V$ and some pair $(B,C)$ of non-degenerate Hermitian forms on $V$.
Hence, by a remark in the preceding paragraph we obtain that  $(b,u)$ is isometric to $(-b,-u)$,
and by Proposition \ref{prop:isomopposite3} we derive condition (ii) in Theorem \ref{theo:hermitianautomorphismcharnot2}.

It remains to check condition (i). Noting that this condition is preserved in taking orthogonal direct sums of $1/2$-pairs,
we can reduce the situation to the one where the pair $(B,C)$ is \emph{indecomposable}.
Let us now assume that such is the case, and let us set $u:=L_B^{-1} L_C$.
We know that only three main situations can occur:

\vskip 3mm
\noindent \textbf{Case 1:} $u$ is cyclic with minimal polynomial $(pp^\bullet)^r$ for some $p \in \Irr(\F) \setminus \K[t]$ and some $r \geq 1$. Then, $B \boxplus C$ is cyclic with minimal polynomial $(pp^\bullet)^r(t^2)=(p(t^2))^r(p^\bullet(t^2))^r$.
We note that $p(t^2) \not\in \K[t]$. Since $p$ is irreducible, the only possible even $q \in \Irr(\F)$ that divides $p(t^2)$
is $p(t^2)$ (the polynomial $p(t^2)$ might not be irreducible, by the way). Hence, condition (i) is trivially satisfied
(all the Hermitian invariants of $(b,u)$ are trivial).

\vskip 3mm
\noindent \textbf{Case 2:} $u$ is cyclic with minimal polynomial $p^r$ for some $p \in \Irr(\F) \cap \K[t]$
such that $p(t^2)$ is irreducible, and some integer $r \geq 1$.

Then, by Lemma \ref{lemma:invariantsboxedsum}, $B \boxplus C$ is cyclic with minimal polynomial $p(t^2)^r$, and the Hermitian invariant
$\bigl(H_V^{1/2},B \boxplus C\bigr)_{p(t^2),r}$ is odd-representable. Hence $(b,u)_{p(t^2),r}$ is the only non-trivial Hermitian invariant
of $(b,u)$, and it is odd-representable.

\vskip 3mm
\noindent
\textbf{Case 3:} $u$ is cyclic with minimal polynomial $p_0^r$ for some $p_0 \in \Irr(\F) \cap \K[t]$
such that $p_0(t^2)$ is not irreducible.
Then $p_0(t^2)=pp^\op$ for some $p \in \Irr(\F)$ such that $p \neq p^\op$, and like in Case 1 we obtain
that condition (i) is trivially satisfied.

\vskip 3mm
Now, we turn to the converse statement. We take a $1/2$-pair $(b,u)$
that satisfies conditions (i) to (iii) in Theorem \ref{theo:hermitianautomorphismcharnot2}, and we prove that $(b,u)$ has the square-zero
splitting property.

We make the Klein group $(\Z/2)^2$ act on $\Irr(\F)$ as follows:
$(a,b).p:=\varphi^a(\psi^b(p))$ where $\varphi : p \mapsto p^\bullet$ and $\psi : p \mapsto p^\op$
(which obviously commute). The orbits of this group action yield a partition $\calP$
of $\Irr(\F)$. Using the classification of $1/2$-pairs, we are able to split
$(b,u) \simeq \underset{\mathcal{O} \in \calP, r \geq 1}{\bot} B_{\mathcal{O},r}$
so that, for all $\mathcal{O} \in \calP$ and all $r \geq 1$:
\begin{itemize}
\item The Jordan numbers of $B_{\mathcal{O},r}$ attached to the polynomials $q \in \Irr(\F) \setminus \mathcal{O}$ are all zero, whereas the ones attached to the polynomials $p \in \mathcal{O}$ and to the integer $r$ are equal to the corresponding ones for $(b,u)$.
    \item The Jordan numbers of $B_{\mathcal{O},r}$ attached to integers that are distinct from $r$ are zero.
\item For all $p \in \mathcal{O}$ such that $p=p^\bullet$, the Hermitian invariant of $B_{\mathcal{O},r}$ attached to $(p,r)$
is equivalent to $(b,u)_{p,r}$.
\end{itemize}
Considering the form of conditions (i) to (iii), it is then clear that each $B_{\mathcal{O},r}$
satisfies them; and it suffices to prove that each $B_{\mathcal{O},r}$ has the square-zero splitting property.

So, in the rest of the proof we only need to consider the case where, for some orbit $\mathcal{O}$
and some $r \geq 1$, the only non-zero Jordan numbers of $(b,u)$ are attached to $(p,r)$ for polynomials
$p$ in $\mathcal{O}$ (still assuming that conditions (i) to (iii) hold).
We split the discussion into several cases.

\vskip 2mm
\noindent \textbf{Case 1.}
$\mathcal{O}=\{p\}$ for some $p \in \Irr(\F)$. Then $p \in \K[t]$ and $p$ is even.
We know from condition (i) that $(b,u)_{p,r}$ is odd-representable. Applying Corollary \ref{cor:oddrepresentable} to the Hermitian form
$(b,u)_{p,r}$, we deduce from the classification of
$1/2$-pairs that $(b,u)$ is isometric to a Hermitian boxed sum, and hence it has the square-zero splitting property.

\noindent \textbf{Case 2.}
$\mathcal{O}=\{p,p^\op\}$ for some $p \in \Irr(\F)$ such that $p^\bullet=p$ and $p \neq p^\op$.
Applying Corollary \ref{cor:hermitianrepresentable} to the Hermitian form
$(b,u)_{p,r}$, we find a Hermitian boxed sum $(H_W^{1/2},B\boxplus C)$
in which $B \boxplus C$ is cyclic with minimal polynomial $p^r (p^\op)^r$ and
$(b,u)_{p,r} \simeq (H_W^{1/2},B\boxplus C)_{p,r}$. Then, by condition (ii) applied to both
$(b,u)$ and $(H_W^{1/2},B\boxplus C)$, we deduce that 
$(b,u)_{p^\op,r} \simeq (H_W^{1/2},B\boxplus C)_{p^\op,r}$. By the classification of $1/2$-pairs, we deduce
that $(b,u) \simeq (H_W^{1/2},B\boxplus C)$. Hence $(b,u)$ is isometric to a Hermitian boxed sum, and
we conclude that it has the square-zero splitting property.

\noindent \textbf{Case 3.}
$\mathcal{O}=\{p,p^\bullet\}$ for some $p \in \Irr(\F)$ such that $p \neq p^\bullet$.

\textbf{Subcase 3.1: $p=p^\op$, i.e.\ $p$ is even.} We choose an endomorphism $v$
whose invariant factors are all equal to $p^r$, and with $n_{p,r}(u)$ such factors.
Then the pair $H_{1/2}(v)$ has exactly two non-zero Jordan numbers, both equal to $n_{p,r}(u)$,
attached to the pairs $(p,r)$ and $(p^\bullet,r)$, and hence it is isometric to $(b,u)$.
By Botha's theorem, $v$ is the sum of two square-zero endomorphisms, and hence
$H_{1/2}(v)$ has the square-zero splitting property. Hence, so does $(b,u)$.

\textbf{Subcase 3.2: $p^\op=p^\bullet$.}
Then we can write $pp^\op=p_0(t^2)$ for some $p_0 \in \K[t]$ which is irreducible over $\F$.
We can find a vector space $W'$ equipped with a pair $(B,C)$ of non-degenerate Hermitian forms
such that all the invariant factors of $L_B^{-1} L_C$ equal $p_0^r$, with $n_{p,r}(u)$ such factors.
Then, by computing the Jordan numbers of $\bigl(H_{W'}^{1/2}, B \boxplus C\bigr)$ we find that this pair
is isometric to $(b,u)$ and we conclude as in the above cases.

\noindent \textbf{Case 4.}
$\mathcal{O}=\{p,p^\bullet,p^\op,(p^\op)^\bullet\}$ for some $p \in \Irr(\F)$ such that $p,p^\bullet,p^\op,(p^\op)^\bullet$
are pairwise distinct.
Then, as in Subcase 3.1, we choose a vector space $V'$ equipped with an endomorphism $v$
whose invariant factors are all equal to $(pp^\op)^r$, and with $n_{p,r}(u)$ such factors.
Then, by computing the Jordan numbers we gather that
$H_{1/2}(v) \simeq (b,u)$. Finally, Botha's theorem shows that
$v$ is the sum of two square-zero endomorphisms, and hence $(b,u)$ has the
square-zero splitting property.

\vskip 3mm
Thus, the proof of Theorem \ref{theo:hermitianautomorphismcharnot2} is now complete.

\subsubsection{Fields of characteristic $2$}

Here, we prove Theorem \ref{theo:hermitianautomorphismchar2}.
The strategy is very similar to the one of the proof of Theorem \ref{theo:hermitianautomorphismcharnot2},
so we shall be brief.

Let $(b,u)$ be a $1/2$-pair with the square-zero splitting property, in which $u$ is an automorphism.
Condition (i) in Theorem \ref{theo:hermitianautomorphismchar2} is immediately satisfied thanks to
Botha's theorem.
By Proposition \ref{prop:caracboxedsumhermitian} and the reduction of Hermitian pairs, we only need to prove condition (ii) when
$(b,u)$ is isometric to $(H_V^{1/2},B \boxplus C)$ for some indecomposable pair $(B,C)$ of non-degenerate Hermitian
forms on a vector space $V$ over $\F$.

Then, there are three cases to consider, depending on $v:=L_B^{-1}L_C$.

\vskip 3mm
\noindent \textbf{Case 1: $v$ is cyclic with minimal polynomial $(pp^\bullet)^r$ for some $p \in \Irr(\F) \setminus \K[t]$
and some $r \geq 1$.} \\
Then $B \boxplus C$ is cyclic with minimal polynomial $p(t^2)^r (p^\bullet(t^2))^r$.
Either the polynomial $p(t^2)$ is irreducible or it is the square of an irreducible polynomial, but in any
case none of its irreducible factors belongs to $\K[t]$, and ditto for $p^\bullet(t^2)$.
Hence, all the Hermitian invariants of $(b,u)$ vanish, and condition (ii) is trivially satisfied.

\vskip 3mm
\noindent \textbf{Case 2: $v$ is cyclic with minimal polynomial $p^r$ for some $p \in \K[t] \cap \Irr(\F)$ and some $r \geq 1$,
and $p(t^2)$ is irreducible over $\F$.} \\
Then $B \boxplus C$ is cyclic with minimal polynomial $p(t^2)^r$,
and Corollary \ref{cor:oddrepresentabledirect} shows that the sole non-zero Hermitian invariant of $(H_V^{1/2},B \boxplus C)$,
that is $(H_V^{1/2},B \boxplus C)_{p(t^2),r}$, is odd-representable.

\vskip 3mm
\noindent \textbf{Case 3: $v$ is cyclic with minimal polynomial $p^r$ for some $p \in \K[t] \cap \Irr(\F)$ and some $r \geq 1$,
and $p(t^2)$ is not irreducible over $\F$.} \\
Then $p(t^2)=q^2$ for some $q \in \K[t] \cap \Irr(\F)$.
Then $B \boxplus C$ is cyclic with minimal polynomial $q^{2r}$, and $q$ is not even. Hence condition (ii) is trivially
satisfied by $\bigl(H_V^{1/2},B \boxplus C\bigr)$.

\vskip 3mm
Conversely, let $(b,u)$ be a $1/2$-pair, where $u$ is an automorphism, that satisfies conditions (i) and (ii)
in Theorem \ref{theo:hermitianautomorphismchar2}. We seek to prove that $(b,u)$
is isometric to an orthogonal sum of pairs that have the square-zero splitting property.
To do so, we can reduce the situation to the following cases:

\vskip 3mm
\noindent \textbf{Case 1: The invariant factors of $u$ all equal $(p(t^2)p^\bullet(t^2))^r$ for some $p \in \F[t] \setminus \K[t]$ such that $p(t^2)$ is irreducible over $\F$, and some $r \geq 1$.} \\
Then we choose a vector space automorphism $v$ whose invariant factors all equal $p(t^2)^r$
and with as many invariant factors as $u$. Then, one sees that
$(b,u)$ and $H_{1/2}(v)$ have exactly the same Jordan numbers and that all their Hermitian invariants are zero. It follows that they are isometric.

\vskip 3mm
\noindent \textbf{Case 2: The invariant factors of $u$ all equal $(p p^\bullet)^{2r}$ for some $p \in \Irr(\F) \setminus \Irr_0(\F)$ such that $p \neq p^\bullet$, and some $r \geq 1$.} \\
    Then we choose a vector space automorphism $v$ whose invariant factors all equal $p^{2r}$ and with as many invariant factors as $u$. Then, one sees that $(b,u)$ and $H_{1/2}(v)$ have exactly the same Jordan numbers and that all their Hermitian invariants are zero. It follows that they are isometric.

\vskip 3mm
\noindent \textbf{Case 3: The invariant factors of $u$ all equal $p^r$ for some $p \in \Irr_0(\F) \cap \K[t]$ and some $r \geq 1$.} \\
Then we write $p=p_0(t^2)$ for some $p_0 \in \Irr(\F) \cap \K[t]$.
In that case, condition (ii) yields that $(b,u)_{p,r}$ is odd-representable.
By Corollary \ref{cor:oddrepresentable}, there exists a space $W$ equipped with a pair $(B,C)$ of bilinear forms such that $B \boxplus C$ has all its invariant factors equal to $p^r$, and
the Hermitian invariant $\bigl(H_W^{1/2},B \boxplus C\bigr)_{p,r}$ is equivalent to $(b,u)_{p,r}$.
It follows from the classification of $1/2$-pairs that
$\bigl(H_W^{1/2},B \boxplus C\bigr)$ is isometric to $(b,u)$.

In any case, $(b,u)$ is isometric to a pair with the square-zero splitting property, and hence it has this property.

This completes the proof of Theorem \ref{theo:hermitiangeneralchar2}.

\subsection{The nilpotent case}

In this final section, we prove Theorem \ref{theo:hermitiannilpotent}.
Thankfully, the proof is very close to the one given in Section \ref{section:nilpotent}, and we will
only focus on the main differences.

We start with the ways of constructing pairs $(b,u)$ with the square-zero splitting property and
in which $u$ is nilpotent. We have already constructed, for $r \geq 1$ even, such pairs in which all the Jordan cells of $u$ are of size $r$, and the Hermitian invariant $(b,u)_{t,r}$ is equivalent to a given non-degenerate Hermitian form.

Just like in Section \ref{section:nilpotent}, the next step is to tackle twisted pairs. Here, the proof is slightly different from the one given in the quadratic case because of the characteristic $2$ case
(if $\charac(\F) \neq 2$ the proof of Lemma \ref{lemma:twisted} applies effortlessly to our case by taking $\eta:=1$).

\begin{lemma}
Let $k \in \N^*$. Let $V$ be a $4k$-dimensional vector space
together with a hyperbolic Hermitian form $h : V^2 \rightarrow \R$.
Then there exists an $h$-selfadjoint endomorphism $u$ of $V$ that is nilpotent with exactly two Jordan cells, one of size $2k-1$ and one of size $2k+1$.
\end{lemma}

\begin{proof}
We can choose $\alpha \in \F$ such that $\alpha^\bullet \neq -\alpha$
(if $\charac(\F) \neq 2$ it suffices to take $\alpha=1$, otherwise the existence follows from the fact that $x \mapsto x^\bullet$ is not the identity of $\F$).

We choose an $h$-hyperbolic basis $\bfB=(e_1,\dots,e_{2k},f_1,\dots,f_{2k})$ of $V$.
The matrix of $h$ in that basis is
$S=\begin{bmatrix}
0 & I_{2k} \\
I_{2k} & 0
\end{bmatrix}$.
Now, we define $a_1 \in \End(V)$ by $a_1(e_{2i})=e_{2i-1}$ and
$a_1(f_{2i-1})=f_{2i}$ for all $i \in \lcro 1,k\rcro$, and $a_1$ maps all the other vectors of $\bfB$ to $0$.
We define $a_2 \in \End(V)$ by $a_2(e_{2i+1})=e_{2i}$ and
$a_2(f_{2i})=f_{2i+1}$ for all $i \in \lcro 1,k-1\rcro$,
$a_2(f_1)=\alpha e_{2k}$ and $a_2(f_{2k})=\alpha^\bullet e_1$, and $a_2$ maps all the other vectors of $\bfB$ to $0$.
It is easily checked on the vectors of $\bfB$ that $a_1^2=0$ and $a_2^2=0$.

The respective matrices $M_1$ and $M_2$ of $a_1$ and $a_2$ in $\bfB$ are of the form
$$M_1=\begin{bmatrix}
A_1 & 0 \\
0 & A_1^\star
\end{bmatrix} \quad \text{and} \quad
M_2=\begin{bmatrix}
A_2 & C \\
0 & A_2^\star
\end{bmatrix}$$
where $C^\star=C$, and $A_1$ and $A_2$ have their entries in $\K$.
One checks that both $SM_1$ and $SM_2$ are Hermitian, to the effect that both $a_1$ and $a_2$ are $h$-selfadjoint.

To conclude, it suffices to check that $u:=a_1+a_2$ is nilpotent and that it has exactly two Jordan cells, one of size $2k+1$
and one of size $2k-1$.
To see this, note first that $u(e_i)=e_{i-1}$ for all $i \in \lcro 2,2k\rcro$, whereas $u(e_1)=0$,
$u(f_1)=f_2+\alpha e_{2k}$, $u(f_i)=f_{i+1}$ for all $i \in \lcro 2,2k-1\rcro$, and
$u(f_{2k})=\alpha^\bullet e_1$.

From $u(f_1)=f_2+\alpha e_{2k}$, we gather that
$u^i(f_1)=f_{i+1}+\alpha e_{2k-i+1}$ for all $i \in \lcro 1,2k-1\rcro$,
in particular $u^{2k-1}(f_1)=f_{2k}+\alpha e_2$. Hence, $u^{2k}(f_1)=(\alpha+\alpha^\bullet) e_1 \neq 0$, and finally $u^{2k+1}(f_1)=0$.
Hence, the subspace $W_1:=\Vect(e_1,f_1) \oplus \Vect(f_{i+1}+\alpha\,e_{2k-i+1})_{1 \leq i \leq 2k-1}$
is stable under $u$ and the resulting endomorphism is a Jordan cell of size $2k+1$.

Besides, one checks that $u^i(\alpha^{\bullet} e_{2k}-f_2)=\alpha^{\bullet} e_{2k-i}-f_{i+2}$ for all $i \in \lcro 1,2k-2\rcro$,
and in particular $u^{2k-2}(\alpha^{\bullet} e_{2k}-f_2)=\alpha^{\bullet} e_2-f_{2k}$ and then $u^{2k-1}(\alpha^\bullet e_{2k}-f_2)=0$.
Obviously, $W_2:=\Vect(\alpha^\bullet e_{2k-i+1}-f_{i+1})_{1 \leq i \leq 2k-1}$ is a $(2k-1)$-dimensional subspace, it is stable
under $u$ and the resulting endomorphism is a Jordan cell of size $2k-1$.
Finally, using the fact that $\alpha^\bullet \neq -\alpha$, one checks that $V=W_1 \oplus W_2$, and the claimed result follows.
\end{proof}

The same arguments as in Section \ref{section:nilpotent} (with $\eta=1$) then show that,
in the previous lemma, the Hermitian invariants $(b,u)_{t,2k+1}$ and $(b,u)_{t,2k-1}$
satisfy $(b,u)_{t,2k+1} \simeq -(b,u)_{t,2k-1}$.

From there, the proof that, in Theorem \ref{theo:hermitiannilpotent}, condition (ii) implies condition (i)
is a word for word adaptation of the one given in Section \ref{section:nilpotent} for Theorem \ref{theo:symformnilpotent}, where one takes $\eta:=1$.

\vskip 3mm
It remains to prove the converse implication. Once again, the techniques with induced forms are adapted effortlessly
to the Hermitian case by taking $\eta:=1$, and one is entirely reduced to the following lemma:

\begin{lemma}\label{lemma:keylemmanilpotenthermitian}
Let $(b,u)$ be a $1/2$-pair in which $u^3=0$.
Assume that $u$ is the sum of two square-zero $b$-selfadjoint endomorphisms.
Set $B_1:=(b,u)_{t,1}$ and $B_3:=(b,u)_{t,3}$.
Then $B_1$ Witt-simplifies $B_3$.
\end{lemma}

\begin{proof}
Denote by $V$ the underlying vector space of $(b,u)$.

The proof is globally similar to the one of Lemma \ref{lemma:keylemmanilpotent}, with only subtle differences.
Step 1 is adapted effortlessly, and we are only left with the case where $B_3$ is non-hyperbolic and the pair
$(b,u)$ is indecomposable as a $1/2$-pair with the square-zero splitting property.

In the rest of the proof, one takes a $b$-selfadjoint square-zero endomorphism $a$ such that $u-a$ has square zero.

Step 2 is slightly different: by taking advantage of the classification of $1/2$-pairs, one shows that if such a $1/2$-pair $(b',u')$, with underlying vector space $V'$, satisfies $(u')^2=0$ and $u' \neq 0$, then there exists a pair $(x,y)$ of vectors of $V'$ such that $b(x,x)=b(y,y)=0$, $b(x,y)=b(y,x)=1$, and $u'(x)=\alpha y$ for some $\alpha \in \K \setminus \{0\}$
(and consequently $u'(y)=0$).

Now, for Step 3 one easily adapts the case where $\eta=1$. There, one proves that the assumption that $a$ does not map $V$ into $\Ker u^2$
would lead to a contradiction with the assumption that $B_3$ is non-hyperbolic, thanks to the assumption that $(b,u)$ is indecomposable as a $1/2$-pair with the square-zero splitting property. We conclude at this point that $a$ maps $V$ into $\Ker u^2$.

In Step 4, we consider the endomorphism $\overline{a}$ of $(\Ker u^2)/(\im u+\Ker u)$,
and we obtain that it is skew-selfadjoint for the Hermitian form $B_2:=(b,u)_{t,2}$.
Let us choose a nonzero $\alpha \in \F$ such that $\alpha^\bullet=-\alpha$. Then $\alpha \overline{a}$ is $B_2$-selfadjoint, and its square equals zero.
Assume now that $\overline{a} \neq 0$. Using the classification of $1/2$-pairs, we would obtain a pair
$(x,y)$ of vectors of $(\Ker u^2)/(\im u+\Ker u)$ such that $B_2(x,x)=B_2(y,y)=0$, $B_2(x,y)=B_2(y,x)=1$, together with a non-zero scalar $\beta \in \K$ such that $\alpha \overline{a}(x)=\beta y$, leading to $\overline{a}(x)=\lambda y$ for some nonzero $\lambda \in \F$. From there, the rest of Step 4 is adapted to the present situation by lifting $x$ to a vector $x'$ of $\Ker u^2$, and by showing that $U:=\Vect(x',a(x'),u(x'),u(a(x')))$ is $b$-regular and stable under $a$ and $u$.

It follows that $a$ maps $\Ker u^2$ into $\Ker u+\im u$.

The remaining two steps can then be adapted word for word (with $\eta=1$), the only difference being that
$(x,y)\in V_3 \times a(V_1) \mapsto b(x,u(y))$ is sesquilinear, not bilinear, but this has no impact on the validity of the arguments.
This completes the adaptation of the proof of Lemma \ref{lemma:keylemmanilpotent}.
\end{proof}

\end{document}